\documentclass[11pt]{amsart}

\usepackage{
	amsmath,
	amsfonts,
	amssymb,
	amsthm,
	amscd,
	graphicx,
	comment,      
	etoolbox,     
	mathtools,    
	enumitem,
	mathdots,     
	stmaryrd,     
	booktabs,     
	stackrel,     
}
\usepackage[usenames,dvipsnames]{xcolor}

\usepackage[scaled=0.88]{beraserif} 
\usepackage[scaled=0.85]{berasans}  
\usepackage[scaled=0.84]{beramono}  
\usepackage[T1]{fontenc}
\usepackage{textcomp}               
\usepackage[sc]{mathpazo}           
\usepackage[T1,small,euler-digits,euler-hat-accent]{eulervm}  
\usepackage{mathrsfs}               

\usepackage[colorlinks=true, pdfstartview=FitV, linkcolor=blue, citecolor=blue, urlcolor=blue, breaklinks=true]{hyperref}

\usepackage[capitalize]{cleveref}   

\usepackage{tikz}
\usetikzlibrary{arrows.meta}
\usetikzlibrary{decorations.markings}
\usetikzlibrary{decorations.pathreplacing}

\newcommand\redcircle[1]{\filldraw[fill=white, draw=red] #1 circle (2pt)}

\newcommand\bluedot[1]{\filldraw[blue] #1 circle (2pt)}
\newcommand\greensquare[1]{\filldraw[green,xshift=-2pt,yshift=-2pt] #1 rectangle ++(4pt,4pt)}

\newcommand\squarelabel[1]{$\scriptstyle{#1}$}
\newcommand\dotlabel[1]{$\scriptstyle{#1}$}

\newcommand\cbubble[3]{
	\begin{tikzpicture}[anchorbase]
	\draw[->] (0,0.3) arc (90:-270:0.3);
	\bluedot{(0.3,0)} node[anchor=west,color=black] {#1};
	\draw{(0,-0.3)} node[anchor=north] {#2};
	\redcircle{(-0.3,0)} node[anchor=east,color=black] {#3};
	\end{tikzpicture}
}

\newcommand\ccbubble[3]{
	\begin{tikzpicture}[anchorbase]
	\draw[->] (0,0.3) arc (90:450:0.3);
	\bluedot{(-0.3,0)} node[anchor=east,color=black] {#1};
	\draw{(0,-0.3)} node[anchor=north] {#2};
	\redcircle{(0.3,0)} node[anchor=west,color=black] {#3};
	\end{tikzpicture}
}

\tikzset{anchorbase/.style={>=To,baseline={([yshift=-0.5ex]current bounding box.center)}}}

\tikzset{->-/.style={decoration={
			markings,
			mark=at position #1 with {\arrow{>}}},postaction={decorate}}}
\tikzset{-<-/.style={decoration={
			markings,
			mark=at position #1 with {\arrow{<}}},postaction={decorate}}}

\crefname{eg}{Example}{Examples}
\crefname{lem}{Lemma}{Lemmas}
\crefname{theo}{Theorem}{Theorems}
\crefname{equation}{}{}
\crefname{enumi}{}{}

\newcommand\Z{\mathbb{Z}}

\newcommand\kk{\Bbbk}

\newcommand\Heis{\mathcal{H}\mathit{eis}}

\newcommand\one{\mathbb{1}}

\newcommand\tr{\mathrm{tr}}

\newcommand\sQ{\mathsf{Q}}

\def\chk#1{#1^{\smash{\scalebox{.8}[1.4]{\rotatebox{90}{\textnormal{\guilsinglleft}}}}}} 

\DeclareMathOperator{\End}{End}

\DeclareMathOperator{\id}{id}
\DeclareMathOperator{\im}{im}       

\newtheorem{theo}{Theorem}[section]

\newtheorem{lem}[theo]{Lemma}
\newtheorem*{lem*}{Lemma}
\newtheorem{cor}[theo]{Corollary}

\theoremstyle{definition}
\newtheorem{defin}[theo]{Definition}

\numberwithin{equation}{section}
\allowdisplaybreaks

\setenumerate[1]{label=(\alph*)}  
\setenumerate[2]{label=(\roman*)}

\newtoggle{comments}
\newtoggle{details}
\newtoggle{detailsnote}

\toggletrue{comments}   
\togglefalse{details}   

\toggletrue{detailsnote}   

\iftoggle{comments}{%
	\newcommand{\comments}[1]{
		\ \\
		{\color{red}
			\textbf{RG:} #1
		}
		\\
	}
}{%
	\newcommand{\comments}[1]{}
}

\iftoggle{details}{%
	\newcommand{\details}[1]{
		\ \\
		{\color{OliveGreen}
			\textbf{Details:} #1
		}
		\\
	}
}{%
	\newcommand{\details}[1]{}
}

\begin{document}
	%
	
	\title{Decomposing Frobenius Heisenberg categories}
	
	\author{Raj Gandhi}
	\address{
		Department of Mathematics and Statistics \\
		University of Ottawa
	}
	\email{rgand037@uottawa.ca}

	\begin{abstract}
	We give two alternate presentations of the Frobenius Heisenberg category, $\Heis_{F,k}$, defined by Savage, when the Frobenius algebra $F=F_1\oplus\dotsb\oplus F_n$ decomposes as a direct sum of Frobenius subalgebras. In these alternate presentations, the morphism spaces of $\Heis_{F,k}$ are given in terms of planar diagrams consisting of strands "colored" by integers $i=1,\dotsc,n$, where a strand of color $i$ carries tokens labelled by elements of $F_i.$ In addition, we prove that when $F$ decomposes this way, the tensor product of Frobenius Heisenberg categories, $\Heis_{F_1,k}\otimes\dotsb\otimes\Heis_{F_n,k},$ is equivalent to a certain subcategory of the Karoubi envelope of $\Heis_{F,k}$ that we call the $\textit{partial}$ Karoubi envelope of $\Heis_{F,k}$.
	\end{abstract}
	
	
	\maketitle
	\thispagestyle{empty}
	
	
	\section{Introduction}
		In \cite{Kho14}, Khovanov developed a graphical calculus for the induction and restriction functors related to the representation theory of the symmetric groups. This led him to define a monoidal category, $\mathcal{H},$ whose Grothendieck ring, he conjectured, is isomorphic to the infinite-dimensional Heisenberg algebra, and he called $\mathcal{H}$ the $\textit{Heisenberg category}.$ This conjecture has recently been proven in \cite{BSW18}. Khovanov's construction has been generalized to q-deformations in \cite{LS13}, to categories depending on a graded Frobenius superalgebra in \cite{CL12} and \cite{RS17}, and to higher level in \cite{MS17}. 
		
		In \cite{Bru17}, Brundan gave two alternate presentations of the higher level Heisenberg categories of \cite{MS17}. One involved an "inversion relation," while the other was similar to the original presentation of the Heisenberg category, given in \cite{Kho14}, but with fewer relations.
		
		In \cite{S17}, Savage used the inversion relation approach of Brundan to unify the higher level categories of \cite{MS17} and the categories of \cite{RS17}. He associated to each graded Frobenius superalgebra, $F,$ and integer, $k,$ a graded Heisenberg supercategory, which he called a $\textit{Frobenius Heisenberg}$ $\textit{category}$ and denoted $\Heis_{F,k}$. 
		
		In the present paper, we give two alternate presentations of  $\Heis_{F,k}$ in the case where $F=F_1\oplus\dotsb\oplus F_n$ decomposes as a direct sum of Frobenius subalgebras. Throughout this paper, we assume the trace map of $F$ is symmetric, we ignore gradings, and we work in the non-super setting. Although we expect analogous presentations of $\Heis_{F,k}$ to hold without these assumptions, we make these assumptions to simplify exposition. In the original presentation of $\Heis_{F,k}$ found in \cite[Def. 1.1]{S17}, the morphism space of $\Heis_{F,k}$ is given in terms of planar diagrams consisting of strands carrying tokens labelled by elements of $F$. In our new presentations of  $\Heis_{F,k}$, the morphism spaces of $\Heis_{F,k}$ are given in terms of planar diagrams consisting of strands \textit{colored} by integers $i=1,\dotsc,n$, where a strand of color $i$ carries tokens labelled by elements of $F_i$. Colored morphisms are common in the categorification literature. For example, when $F$ is a zigzag algebra and $k=-1$, $\Heis_{F,k}$ is the Heisenberg category of Cautis and Licata in \cite{CL12}. The presentation given in \cite{CL12} is in terms of colored planar diagrams. The morphism spaces of Kac-Moody 2-categories are also colored (by elements of the weight lattice). See, for example, \cite{R} and \cite{KL10}. Our new presentation of $\Heis_{F,k}$ takes advantage of the decomposition of $F$ to decompose the morphism spaces of $\Heis_{F,k},$ which results in a simpler way to view $\Heis_{F,k}$. An example of the simplifications that can be made using these new presentations of $\Heis_{F,k}$ is given by the proof of \cref{cor:equivalence} in \cref{section:equivalence}, where we use one of these new presentations to show that the tensor product of Frobenius Heisenberg categories, $\Heis_{F_1,k}\otimes\dotsb\otimes \Heis_{F_n,k}$, is equivalent to a certain subcategory of the Karoubi envelope of $\Heis_{F,k}$ that we call the $\textit{partial}$ Karoubi envelope of $\Heis_{F,k}$.
		
		This paper is organized as follows. In \cref{section:review}, we review the definition and basic facts about the Frobenius Heisenberg category, $\Heis_{F,k}$, introduced in \cite{S17}, and we specialize $F$ to a Frobenius algebra that decomposes as a direct sum of Frobenius subalgebras, $F_1\oplus\dotsb\oplus F_n$. In \cref{section:alternate-presentation}, we prove two alternate presentations of $\Heis_{F,k}$. In \cref{section:equivalence} we recall the definition of the tensor product of categories, and we define the partial Karoubi envelope of $\Heis_{F,k},$ denoted $PK\left(\Heis_{F,k}\right).$ We then show that $PK\left(\Heis_{F,k}\right)$ is equivalent to $\Heis_{F_1,k}\otimes\dotsb\otimes \Heis_{F_n,k}$.
		
		\subsection*{Acknowledgements}
		This research was supported by an Undergraduate Student Research Award from the Natural Sciences and Engineering Research Council of Canada and was supervised by Professor Alistair Savage. The author is grateful to Professor Savage for his patience and guidance throughout this project.
	\section{The Frobenius Heisenberg category}\label{section:review}
	In this section, we recall the definition and basic properties of the Frobenius Heisenberg category first defined in \cite{S17}.
	Fix a commutative ground ring $\kk$. Let $F=F_1\oplus\dotsb\oplus F_n$, be the direct sum of Frobenius $\kk$-algebras, and view $F$ as a Frobenius algebra with trace map $\tr(f_1,\dotsc,f_n) = \sum_{i=1}^{n}\tr_i(f_i)$, where $\tr_i$ denotes the trace of $F_i$. View $F_i$ as a subalgebra of $F$ by identifying an element $f\in F_i$ with the element $(0,\dotsc,0,f,0,\dotsc,0)\in F$, where $f$ appears in the $i^{th}$ position, and define $e_i:=1_{F_i}$.
	For each $i$, we fix a basis $B_i$ of $F_i.$ Thus, $B:=B_1\cup\dotsb\cup B_n$ is a basis of $F$. 
	Furthermore, if $b_{i}\in B_i$ and $b_{j}\in B_j$, $i\neq j,$ then it is easy to see that $b_ib_j=0,$ and hence
	\begin{equation} \label{eq:dual-basis-def}
	\tr(b_i b_j) = 0
	\end{equation}
	as well. We let $\{\chk{b} | b\in B\}$ denote the left dual basis of $B,$ so that 
	\begin{equation}
		\tr(\chk{a}b)=\delta_{a,b}, \quad a,b\in B.
	\end{equation}
	
	Fix $k \in \Z$. We will recall the definition of $\Heis_{F,k}$ given in \cite[Def. 1.1]{S17}, and impose some additional assumptions. Namely, we assume the trace map of $F$ is symmetric, we ignore gradings, and we are in the non-super setting.
	\begin{defin} \label{def:H1}
		The category $\Heis_{F,k}$ is the strict $\kk$-linear monoidal category defined as follows. The objects are generated by $\sQ_+$ and $\sQ_-$, and we use juxtaposition to denote tensor product. We will denote the unit object by $\mathbb{1}$. The morphisms of $\Heis_{F,k}$ are generated by
		\[
		x =
		\begin{tikzpicture}[anchorbase]
		\draw[->] (0,0) -- (0,0.6);
		\redcircle{(0,0.3)};
		\end{tikzpicture} \ \colon Q_+\to Q_+
		\ ,\quad
		s =
		\begin{tikzpicture}[anchorbase]
		\draw [->](0,0) -- (0.6,0.6);
		\draw [->](0.6,0) -- (0,0.6);
		\end{tikzpicture} \ \colon Q_+Q_+\to Q_+Q_+
		\ ,\quad
		c =
		\begin{tikzpicture}[anchorbase]
		\draw[->] (0,.2) -- (0,0) arc (180:360:.3) -- (.6,.2);
		\end{tikzpicture} \ \colon \mathbb{1}\to Q_-Q_+
		\ ,\]
		\[
		d =
		\begin{tikzpicture}[anchorbase]
		\draw[->] (0,-.2) -- (0,0) arc (180:0:.3) -- (.6,-.2);
		\end{tikzpicture} \ \colon Q_+Q_-\to \mathbb{1}
		\ ,\quad
		\beta_f =
		\begin{tikzpicture}[anchorbase]
		\draw[->] (0,0) -- (0,0.6);
		\bluedot{(0,0.3)} node[anchor=west, color=black] {$f$};
		\end{tikzpicture} \ \colon Q_+\to Q_+
		\ ,\ f \in F.
		\]
		We refer to the decoration representing $x$ as a \emph{dot} and the decorations representing $\beta_f$, $f \in F$, as \emph{tokens}. For $n \ge 1$, we denote the $n$-th power $x^n$ of $x$ by labelling the dot with the exponent $n$:
		\[
		x^n =
		\begin{tikzpicture}[anchorbase]
		\draw[->] (0,0) -- (0,1);
		\redcircle{(0,0.5)} node[anchor=west, color=black] {$n$};
		\end{tikzpicture}
		\]
		We also define
		\begin{equation} \label{eq:t-def1}
		t \colon \sQ_+ \sQ_- \to \sQ_- \sQ_+,\quad
		t =
		\begin{tikzpicture}[anchorbase]
		\draw [->](0,0) -- (0.6,0.6);
		\draw [<-](0.6,0) -- (0,0.6);
		\end{tikzpicture}
		\ :=\
		\begin{tikzpicture}[anchorbase,scale=0.6]
		\draw[->] (0.3,0) -- (-0.3,1);
		\draw[->] (-0.75,1) -- (-0.75,0.5) .. controls (-0.75,0.2) and (-0.5,0) .. (0,0.5) .. controls (0.5,1) and (0.75,0.8) .. (0.75,0.5) -- (0.75,0);
		\end{tikzpicture}
		\ .
		\end{equation}
		The identity morphisms on $Q_+$ and $Q_-$ will be denoted by single undecorated upwards or downwards strands, respectively:
		\[
		\begin{tikzpicture}[anchorbase]
		\draw[->] (0,0) -- (0,1);
		\end{tikzpicture} = id_{Q_+}
		, \qquad 
		\begin{tikzpicture}[anchorbase]
		\draw[<-] (0,0) -- (0,1);
		\end{tikzpicture} = id_{Q_-}. 
		\]
		We impose three sets of relations:
		\begin{enumerate}[wide]
			\item \emph{Affine wreath product algebra relations}:  We have a homomorphism of algebras
			\begin{equation} \label{rel:token-homom1}
			F \to \End \sQ_+,\quad f \mapsto \beta_f,
			\end{equation}
			so that, in particular,
			\begin{equation} \label{rel:token-colide-up1}
			\begin{tikzpicture}[anchorbase]
			\draw[->] (0,0) -- (0,1);
			\bluedot{(0,0.35)} node[anchor=east,color=black] {$g$};
			\bluedot{(0,0.7)} node[anchor=east,color=black] {$f$};
			\end{tikzpicture}
			\ = \
			\begin{tikzpicture}[anchorbase]
			\draw[->] (0,0) -- (0,1);
			\bluedot{(0,0.5)} node[anchor=west,color=black] {$fg$};
			\end{tikzpicture}
			,\quad f,g \in F.
			\end{equation}		
			Note also that, for $a\in \kk,$ $f_j\in F,$ $j=1,\dotsc,m,$ $m>0,$ we have
			
			\noindent\begin{minipage}{0.5\linewidth}
			\begin{equation} \label{rel:token-sum-up1}
			\begin{tikzpicture}[anchorbase]
			\draw[->] (0,0) -- (0,1);
			\bluedot{(0,0.5)} node[anchor=east,color=black] {$\left(\sum_{j=1}^{m}f_j\right)$};
			\end{tikzpicture}
			\ = \
			\sum_{j=1}^{m}\left(
			\begin{tikzpicture}[anchorbase]
			\draw[->] (0,0) -- (0,1);
			\bluedot{(0,0.5)} node[anchor=west,color=black] {$f_j$};
			\end{tikzpicture}\right)
			\ ,
			\end{equation}	
			\end{minipage}
			\begin{minipage}{0.5\linewidth}
				\begin{equation} \label{rel:token-scale-up1}
				\begin{tikzpicture}[anchorbase]
				\draw[->] (0,0)  -- (0,1);
				\bluedot{(0,0.5)} node[anchor=east,color=black] {$\left(af_1\right)$};
				\end{tikzpicture}
				\ = \
				a\left(
				\begin{tikzpicture}[anchorbase]
				\draw[->] (0,0) -- (0,1);
				\bluedot{(0,0.5)} node[anchor=west,color=black] {$f_1$};
				\end{tikzpicture}\right)
				\ .
				\end{equation}		
			\end{minipage}
		\indent	Furthermore, we impose the following relations for all $f \in F$:
			
			\noindent\begin{minipage}{0.33\linewidth}
				\begin{equation} \label{rel:braid-up1}
				\begin{tikzpicture}[anchorbase]
				\draw[->] (0,0) -- (1,1);
				\draw[->] (1,0) -- (0,1);
				\draw[->] (0.5,0) .. controls (0,0.5) .. (0.5,1);
				\end{tikzpicture}
				\ =\
				\begin{tikzpicture}[anchorbase]
				\draw[->] (0,0) -- (1,1);
				\draw[->] (1,0) -- (0,1);
				\draw[->] (0.5,0) .. controls (1,0.5) .. (0.5,1);
				\end{tikzpicture}\ ,
				\end{equation}
			\end{minipage}%
			\begin{minipage}{0.33\linewidth}
				\begin{equation} \label{rel:doublecross-up1}
				\begin{tikzpicture}[anchorbase]
				\draw[->] (0,0) .. controls (0.5,0.5) .. (0,1);
				\draw[->] (0.5,0) .. controls (0,0.5) .. (0.5,1);
				\end{tikzpicture}
				\ =\
				\begin{tikzpicture}[anchorbase]
				\draw[->] (0,0) --(0,1);
				\draw[->] (0.5,0) -- (0.5,1);
				\end{tikzpicture}\ ,
				\end{equation}
			\end{minipage}
			\begin{minipage}{0.33\linewidth}
				\begin{equation} \label{rel:dot-token-up-slide1}
				\begin{tikzpicture}[anchorbase]
				\draw[->] (0,0) -- (0,1);
				\redcircle{(0,0.3)};
				\bluedot{(0,0.6)} node[anchor=east, color=black] {$f$};
				\end{tikzpicture}
				\ =\
				\begin{tikzpicture}[anchorbase]
				\draw[->] (0,0) -- (0,1);
				\redcircle{(0,0.6)};
				\bluedot{(0,0.3)} node[anchor=west, color=black] {$f$};
				\end{tikzpicture}\ ,
				\end{equation}
			\end{minipage}\par\vspace{\belowdisplayskip}
			
			\noindent\begin{minipage}{0.4\linewidth}
				\begin{equation} \label{rel:tokenslide-up-right1}
				\begin{tikzpicture}[anchorbase]
				\draw[->] (0,0) -- (1,1);
				\draw[->] (1,0) -- (0,1);
				\bluedot{(.25,.25)} node [anchor=south east, color=black] {$f$};
				\end{tikzpicture}
				\ =\
				\begin{tikzpicture}[anchorbase]
				\draw[->](0,0) -- (1,1);
				\draw[->](1,0) -- (0,1);
				\bluedot{(0.75,.75)} node [anchor=north west, color=black] {$f$};
				\end{tikzpicture}\ ,
				\end{equation}
			\end{minipage}%
			\begin{minipage}{0.6\linewidth}
				\begin{equation} \label{rel:dotslide1.1}
				\begin{tikzpicture}[anchorbase]
				\draw[->] (0,0) -- (1,1);
				\draw[->] (1,0) -- (0,1);
				\redcircle{(0.25,.75)};
				\end{tikzpicture}
				\ -\
				\begin{tikzpicture}[anchorbase]
				\draw[->] (0,0) -- (1,1);
				\draw[->] (1,0) -- (0,1);
				\redcircle{(.75,.25)};
				\end{tikzpicture}
				\ =\
				\sum_{b \in B}
				\begin{tikzpicture}[anchorbase]
				\draw[->] (0,0) -- (0,1);
				\draw[->] (0.5,0) -- (0.5,1);
				\bluedot{(0,0.5)} node[anchor=east, color=black] {$\chk{b}$};
				\bluedot{(0.5,0.5)} node[anchor=west, color=black] {$b$};
				\end{tikzpicture}\ .
				\end{equation}
			\end{minipage}\par\vspace{\belowdisplayskip}
			
			It follows that we also have the relations:
			
			\noindent\begin{minipage}{0.4\linewidth}
				\begin{equation} \label{rel:tokenslide-up-left1}
				\begin{tikzpicture}[anchorbase]
				\draw[->] (0,0) -- (1,1);
				\draw[->] (1,0) -- (0,1);
				\bluedot{(.75,.25)} node [anchor=south west, color=black] {$f$};
				\end{tikzpicture}
				\ =\
				\begin{tikzpicture}[anchorbase]
				\draw[->] (0,0) -- (1,1);
				\draw[->] (1,0) -- (0,1);
				\bluedot{(0.25,.75)} node [anchor=north east, color=black] {$f$};
				\end{tikzpicture}\ ,
				\end{equation}
			\end{minipage}%
			\begin{minipage}{0.6\linewidth}
				\begin{equation} \label{rel:dotslide2.1}
				\begin{tikzpicture}[anchorbase]
				\draw[->] (0,0) -- (1,1);
				\draw[->] (1,0) -- (0,1);
				\redcircle{(0.25,.25)};
				\end{tikzpicture}
				\ -\
				\begin{tikzpicture}[anchorbase]
				\draw[->] (0,0) -- (1,1);
				\draw[->] (1,0) -- (0,1);
				\redcircle{(.75,.75)};
				\end{tikzpicture}
				\ =\
				\sum_{b \in B}
				\begin{tikzpicture}[anchorbase]
				\draw[->] (0,0) -- (0,1);
				\draw[->] (0.5,0) -- (0.5,1);
				\bluedot{(0,0.5)} node[anchor=east, color=black] {$b$};
				\bluedot{(0.5,0.5)} node[anchor=west, color=black] {$\chk{b}$};
				\end{tikzpicture}\ .
				\end{equation}
			\end{minipage}\par\vspace{\belowdisplayskip}
			
			\item \emph{Right adjunction relations}:  We impose the following relations:
			
			\noindent\begin{minipage}{0.5\linewidth}
				\begin{equation} \label{rel:right-adjunction-up1}
				\begin{tikzpicture}[anchorbase]
				\draw[->] (0,0) -- (0,0.6) arc(180:0:0.2) -- (0.4,0.4) arc(180:360:0.2) -- (0.8,1);
				\end{tikzpicture}
				\ =\
				\begin{tikzpicture}[anchorbase]
				\draw[->] (0,0) -- (0,1);
				\end{tikzpicture}\ ,
				\end{equation}
			\end{minipage}%
			\begin{minipage}{0.5\linewidth}
				\begin{equation} \label{rel:right-adjunction-down1}
				\begin{tikzpicture}[anchorbase]
				\draw[->] (0,1) -- (0,0.4) arc(180:360:0.2) -- (0.4,0.6) arc(180:0:0.2) -- (0.8,0);
				\end{tikzpicture}
				\ =\
				\begin{tikzpicture}[anchorbase]
				\draw[<-] (0,0) -- (0,1);
				\end{tikzpicture}\ .
				\end{equation}
			\end{minipage}\par\vspace{\belowdisplayskip}
			
			\item \emph{Inversion relation}: The following matrix of morphisms is an isomorphism in the additive envelope of $\Heis_{F,k}$:
			\begin{gather} \label{eq:inversion-relation-pos-l.1}
			\left[
			\begin{tikzpicture}[anchorbase]
			\draw [->](0,0) -- (0.6,0.6);
			\draw [<-](0.6,0) -- (0,0.6);
			\end{tikzpicture}
			\quad
			\begin{tikzpicture}[anchorbase]
			\draw[->] (0,0) -- (0,0.7) arc (180:0:.3) -- (0.6,0);
			\redcircle{(0,0.6)} node[anchor=west,color=black] {$r$};
			\bluedot{(0,0.2)} node[anchor=west,color=black] {$\chk{b}$};
			\end{tikzpicture}\ ,\
			0 \le r \le k-1,\ b \in B
			\right]^T
			\ \colon \sQ_+ \sQ_- \to \sQ_- \sQ_+ \oplus \one ^{\oplus k \dim F} \quad \text{if } k \ge 0,
			\\ \label{eq:inversion-relation-neg-l.1}
			\left[
			\begin{tikzpicture}[anchorbase]
			\draw [->](0,0) -- (0.6,0.6);
			\draw [<-](0.6,0) -- (0,0.6);
			\end{tikzpicture}
			\quad
			\begin{tikzpicture}[anchorbase]
			\draw[->] (0,1) -- (0,0.3) arc (180:360:.3) -- (0.6,1);
			\redcircle{(0.6,0.7)} node[anchor=east,color=black] {$r$};
			\bluedot{(0.6,0.3)} node[anchor=east,color=black] {$\chk{b}$};
			\end{tikzpicture}\ ,\
			0 \le r \le -k-1,\ b \in B
			\right]
			\ \colon \sQ_+ \sQ_- \oplus \one^{\oplus (-k \dim F)} \to \sQ_- \sQ_+ \quad \text{if } k < 0.
			\end{gather}
			(The matrix \cref{eq:inversion-relation-pos-l.1} is of size $(1 + k \dim F) \times 1$, while the matrix \cref{eq:inversion-relation-neg-l.1} is of size $1 \times (1 + k \dim F)$.)  Note that these conditions are independent of the choice of basis $B$ of $F$.
		\end{enumerate}
	\end{defin}
	
	The affine wreath product relations correspond to the defining relations of affine wreath product algebras. See \cite{Sav17Updated} for further discussion of these algebras. In our case, the category $\Heis_{F,k}$ is strictly pivotal (see \cite[page 4]{S17}). In particular, there is isotopy invariance of all affine wreath product relations, giving a larger set of relations that hold in $\Heis_{F,k}$.
	
	For the sake of having notation, we define (as in \cite[page 11]{S17})
	\begin{equation}\label{eq:left-crossing1}
	t' =
	\begin{tikzpicture}[anchorbase]
	\draw [<-](0,0) -- (0.6,0.6);
	\draw [->](0.6,0) -- (0,0.6);
	\end{tikzpicture}
	\ \colon \sQ_- \sQ_+ \to \sQ_+ \sQ_-,
	\end{equation}
	and
	\begin{equation} \label{eq:decorated-left-cap-cup1}
	\begin{tikzpicture}[>=To,baseline={([yshift=1ex]current bounding box.center)}]
	\draw[<-] (0,0.2) -- (0,0) arc (180:360:.3) -- (0.6,0.2);
	\greensquare{(0.3,-0.3)} node[anchor=north,color=black] {\squarelabel{(r,b)}};
	\end{tikzpicture}
	\ \colon \one \to \sQ_+ \sQ_-
	,\qquad
	\begin{tikzpicture}[>=To,baseline={([yshift=-2ex]current bounding box.center)}]
	\draw[<-] (0,-0.2) -- (0,0) arc (180:0:.3) -- (0.6,-0.2);
	\greensquare{(0.3,0.3)} node[anchor=south,color=black] {\squarelabel{(r,b)}};
	\end{tikzpicture}
	\ \colon \sQ_+ \sQ_- \to \one,
	\end{equation}
	for $0 \le r < k$ or $0 \le r < -k$, respectively, by declaring that
	\begin{equation} \label{eq:inversion-leftcup-def1}
	\left[
	\begin{tikzpicture}[anchorbase]
	\draw [<-](0,0) -- (0.6,0.6);
	\draw [->](0.6,0) -- (0,0.6);
	\end{tikzpicture}
	\quad
	\begin{tikzpicture}[>=stealth,baseline={([yshift=1ex]current bounding box.center)}]
	\draw[<-] (0,0.2) -- (0,0) arc (180:360:.3) -- (0.6,0.2);
	\greensquare{(0.3,-0.3)} node[anchor=north,color=black] {\squarelabel{(r,b)}};
	\end{tikzpicture},\
	0 \le r \le k-1,\ b \in B
	\right]
	=
	\left(\left[
	\begin{tikzpicture}[anchorbase]
	\draw [->](0,0) -- (0.6,0.6);
	\draw [<-](0.6,0) -- (0,0.6);
	\end{tikzpicture}
	\quad
	\begin{tikzpicture}[anchorbase]
	\draw[->] (0,0) -- (0,0.7) arc (180:0:.3) -- (0.6,0);
	\redcircle{(0,0.6)} node[anchor=west,color=black] {$r$};
	\bluedot{(0,0.2)} node[anchor=west,color=black] {$\chk{b}$};
	\end{tikzpicture}\ ,\
	0 \le r \le k-1,\ b \in B
	\right]^T \right)^{-1}
	\quad \text{if } k \ge 0,
	\end{equation}
	or
	\begin{equation} \label{eq:inversion-leftcap-def1}
	\left[
	\begin{tikzpicture}[anchorbase]
	\draw [<-](0,0) -- (0.6,0.6);
	\draw [->](0.6,0) -- (0,0.6);
	\end{tikzpicture}
	\quad
	\begin{tikzpicture}[>=stealth,baseline={([yshift=-2ex]current bounding box.center)}]
	\draw[<-] (0,-0.2) -- (0,0) arc (180:0:.3) -- (0.6,-0.2);
	\greensquare{(0.3,0.3)} node[anchor=south,color=black] {\squarelabel{(r,b)}};
	\end{tikzpicture},\
	0 \le r \le -k-1,\ b \in B
	\right]^T
	=
	\left[
	\begin{tikzpicture}[anchorbase]
	\draw [->](0,0) -- (0.6,0.6);
	\draw [<-](0.6,0) -- (0,0.6);
	\end{tikzpicture}
	\quad
	\begin{tikzpicture}[anchorbase]
	\draw[->] (0,1) -- (0,0.3) arc (180:360:.3) -- (0.6,1);
	\redcircle{(0.6,0.7)} node[anchor=east,color=black] {$r$};
	\bluedot{(0.6,0.3)} node[anchor=east,color=black] {$\chk{b}$};
	\end{tikzpicture}\ ,\
	0 \le r \le -k-1,\ b \in B
	\right]^{-1}
	\quad \text{if } k < 0.
	\end{equation}	
	As in \cite[page 11]{S17}, we extend the definition of the decorated left cups and caps by linearity in the second argument of the label.	
	
	The relations \cref{rel:down-up-doublecross1,rel:up-down-doublecross1} hold in $\Heis_{F,k}$ for all $k:$
	
	\noindent \begin{minipage}{0.5\linewidth}
		\begin{equation} \label{rel:up-down-doublecross1}
		\begin{tikzpicture}[anchorbase]
		\draw[->] (0,0) -- (0,1);
		\draw[<-] (0.5,0) -- (0.5,1);
		\end{tikzpicture}
		\ =\
		\begin{tikzpicture}[anchorbase]
		\draw[->] (0,0) .. controls (0.5,0.5) .. (0,1);
		\draw[<-] (0.5,0) .. controls (0,0.5) .. (0.5,1);
		\end{tikzpicture}
		\ + \sum_{r=0}^{k-1} \sum_{b \in B}\
		\begin{tikzpicture}[anchorbase]
		\draw[->] (0,0) -- (0,0.7) arc (180:0:0.3) -- (0.6,0);
		\draw[<-] (0,2) -- (0,1.8) arc (180:360:0.3) -- (0.6,2);
		\redcircle{(0,0.6)} node[anchor=west,color=black] {$r$};
		\bluedot{(0,0.2)} node[anchor=west,color=black] {$\chk{b}$};
		\greensquare{(0.3,1.5)} node[anchor=north west,color=black] {\squarelabel{(r,b)}};
		\end{tikzpicture}
		\ ,
		\end{equation}
	\end{minipage}%
	\noindent\begin{minipage}{0.5\linewidth}
		\begin{equation} \label{rel:down-up-doublecross1}
		\begin{tikzpicture}[anchorbase]
		\draw[<-] (0,0) -- (0,1);
		\draw[->] (0.5,0) -- (0.5,1);
		\end{tikzpicture}
		\ =\
		\begin{tikzpicture}[anchorbase]
		\draw[<-] (0,0) .. controls (0.5,0.5) .. (0,1);
		\draw[->] (0.5,0) .. controls (0,0.5) .. (0.5,1);
		\end{tikzpicture}
		\ + \sum_{r=0}^{-k-1} \sum_{b \in B}\
		\begin{tikzpicture}[anchorbase]
		\draw[->] (0,2) -- (0,1.3) arc (180:360:0.3) -- (0.6,2);
		\draw[<-] (0,0) -- (0,0.2) arc (180:0:0.3) -- (0.6,0);
		\redcircle{(0.6,1.7)} node[anchor=east,color=black] {$r$};
		\bluedot{(0.6,1.3)} node[anchor=east,color=black] {$\chk{b}$};
		\greensquare{(0.3,0.5)} node[anchor=west,color=black] {\squarelabel{(r,b)}};
		\end{tikzpicture}
		\ .
		\end{equation}
	\end{minipage}\par\vspace{\belowdisplayskip}
	\noindent When $k > 0$, the following relations, \cref{rel:inversion1-xi-greater1,rel:inversion2-xi-greater1,rel:inversion3-xi-greater1}, hold in $\Heis_{F,k}$. For all $0 \le r,s \le k-1$ and $b,c\in B$,
	
	\noindent\begin{minipage}{0.33\linewidth}
		\begin{equation} \label{rel:inversion1-xi-greater1}
		\begin{tikzpicture}[anchorbase]
		\draw[<-] (0,0) -- (0,1);
		\draw[->] (0.5,0) -- (0.5,1);
		\end{tikzpicture}
		\ =\
		\begin{tikzpicture}[anchorbase]
		\draw[<-] (0,0) .. controls (0.5,0.5) .. (0,1);
		\draw[->] (0.5,0) .. controls (0,0.5) .. (0.5,1);
		\end{tikzpicture}
		\end{equation}
	\end{minipage}
	\begin{minipage}{0.33\linewidth}
		\begin{equation} \label{rel:inversion2-xi-greater1}
		\begin{tikzpicture}[anchorbase]
		\draw[->] (0,0) .. controls (0.3,-0.3) and (0.6,-0.3) .. (0.6,-0.6) arc(360:180:0.3) .. controls (0,-0.3) and (0.3,-0.3) .. (0.6,0);
		\greensquare{(0.3,-0.9)} node[anchor=north,color=black] {$(r,b)$};
		\end{tikzpicture}
		\ =\
		\begin{tikzpicture}[anchorbase]
		\draw[->] (0.6,-0.5) .. controls (0.1,0) and (0,-0.1) .. (0,0.3) -- (0,0.7) arc (180:0:.3) -- (0.6,0.3) .. controls (0.6,-0.1) and (0.5,0) .. (0,-0.5);
		\redcircle{(0,0.6)} node[anchor=west,color=black] {$r$};
		\bluedot{(0,0.3)} node[anchor=west,color=black] {$\chk{b}$};
		\end{tikzpicture}
		\ =0,
		\end{equation}
	\end{minipage}%
	\begin{minipage}{0.33\linewidth}
		\begin{equation} \label{rel:inversion3-xi-greater1}
		\begin{tikzpicture}[anchorbase]
		\draw[<-] (0,0.3) arc(90:450:0.3);
		\redcircle{(-0.3,0)} node[anchor=east,color=black] {$r$};
		\bluedot{(0.3,0)} node[anchor=west,color=black] {$\chk{b}$};
		\greensquare{(0,-0.3)} node[anchor=north,color=black] {$(s,c)$};
		\end{tikzpicture}
		= \delta_{r,s} \delta_{b,c}.
		\end{equation}
	\end{minipage}\par\vspace{\belowdisplayskip}

	\noindent When $k < 0$, the following relations, \cref{rel:inversion1-xi-smaller1,rel:inversion2-xi-smaller1,rel:inversion3-xi-smaller1}, hold in $\Heis_{F,k}$. For all $0 \le r,s \le -k-1$ and $b,c\in B$,

	\noindent\begin{minipage}{0.33\linewidth}
		\begin{equation} \label{rel:inversion1-xi-smaller1}
		\begin{tikzpicture}[anchorbase]
		\draw[->] (0,0) -- (0,1);
		\draw[<-] (0.5,0) -- (0.5,1);
		\end{tikzpicture}
		\ =\
		\begin{tikzpicture}[anchorbase]
		\draw[->] (0,0) .. controls (0.5,0.5) .. (0,1);
		\draw[<-] (0.5,0) .. controls (0,0.5) .. (0.5,1);
		\end{tikzpicture},
		\end{equation}
	\end{minipage}
	\begin{minipage}{0.33\linewidth}
		\begin{equation} \label{rel:inversion2-xi-smaller1}
		\begin{tikzpicture}[anchorbase]
		\draw[->] (0,0) .. controls (0.3,0.3) and (0.6,0.3) .. (0.6,0.6) arc(0:180:0.3) .. controls (0,0.3) and (0.3,0.3) .. (0.6,0);
		\greensquare{(0.3,0.9)} node[anchor=south,color=black] {$(r,b)$};
		\end{tikzpicture}
		\ =\
		\begin{tikzpicture}[anchorbase]
		\draw[->] (0.6,0.5) .. controls (0.1,0) and (0,0.1) .. (0,-0.3) -- (0,-0.7) arc (180:360:.3) -- (0.6,-0.3) .. controls (0.6,0.1) and (0.5,0) .. (0,0.5);
		\redcircle{(0.6,-0.3)} node[anchor=east,color=black] {$r$};
		\bluedot{(0.6,-0.6)} node[anchor=west,color=black] {$\chk{b}$};
		\end{tikzpicture}
		\  = 0,
		\end{equation}
	\end{minipage}%
	\begin{minipage}{0.33\linewidth}
		\begin{equation} \label{rel:inversion3-xi-smaller1}
		\begin{tikzpicture}[anchorbase]
		\draw[->] (0,-0.3) arc(270:630:0.3);
		\redcircle{(0.3,0)} node[anchor=west,color=black] {$r$};
		\bluedot{(-0.3,0)} node[anchor=east,color=black] {$\chk{b}$};
		\greensquare{(0,0.3)} node[anchor=south,color=black] {$(s,c)$};
		\end{tikzpicture}
		= \delta_{r,s} \delta_{b,c}.
		\end{equation}
	\end{minipage}\par\vspace{\belowdisplayskip}
\noindent 	The inversion relation in $\Heis_{F,k}$ is equivalent to the relations \cref{rel:up-down-doublecross1,rel:inversion1-xi-greater1,rel:inversion2-xi-greater1,rel:inversion3-xi-greater1} when $k>0$, the relations \cref{rel:down-up-doublecross1,rel:inversion1-xi-smaller1,rel:inversion2-xi-smaller1,rel:inversion3-xi-smaller1} when $k<0,$ and the relations \cref{rel:down-up-doublecross1,rel:up-down-doublecross1} when $k=0.$ Note that when $k=0$, the relations \cref{rel:down-up-doublecross1,rel:up-down-doublecross1} imply that
\[
t' =
\begin{tikzpicture}[anchorbase]
\draw [<-](0,0) -- (0.6,0.6);
\draw [->](0.6,0) -- (0,0.6);
\end{tikzpicture}
\colon \sQ_- \sQ_+ \to \sQ_+ \sQ_-,
\]
is a generating morphism that is inverse to $t$, so that
\[
\begin{tikzpicture}[anchorbase]
\draw[->] (0,0) .. controls (0.5,0.5) .. (0,1);
\draw[<-] (0.5,0) .. controls (0,0.5) .. (0.5,1);
\end{tikzpicture}
\ =\
\begin{tikzpicture}[anchorbase]
\draw[->] (0,0) -- (0,1);
\draw[<-] (0.5,0) -- (0.5,1);
\end{tikzpicture}
,\qquad
\begin{tikzpicture}[anchorbase]
\draw[<-] (0,0) .. controls (0.5,0.5) .. (0,1);
\draw[->] (0.5,0) .. controls (0,0.5) .. (0.5,1);
\end{tikzpicture}
\ =\
\begin{tikzpicture}[anchorbase]
\draw[<-] (0,0) -- (0,1);
\draw[->] (0.5,0) -- (0.5,1);
\end{tikzpicture}
\ .
\]

	Note also that the following relations hold in $\Heis_{F,k}$, which will be useful in the proof of \cref{lem:colored-decorated-cups-caps}. The relations \cref{rel:tokenslide-right-crossing-up-left,rel:tokenslide-right-crossing-up-right,rel:tokenslide-left-crossing-up-left,rel:tokenslide-left-crossing-up-right} are a result of the relations \cref{rel:tokenslide-up-right1,rel:tokenslide-up-left1} due to isotopy invariance of all affine wreath product relations in $\Heis_{F,k}$. The relations \cref{rel:colored-right-cap-token-move,rel:colored-right-cup-token-move} are equivalent to the relations given by equation number $(2.8)$ in \cite{S17}.  
	
	\noindent\begin{minipage}{0.5\linewidth}
		\begin{equation} \label{rel:tokenslide-right-crossing-up-left}
		\begin{tikzpicture}[anchorbase]
		\draw[->] (0,0) -- (1,1);
		\draw[<-] (1,0) -- (0,1);
		\bluedot{(.75,.25)} node [anchor=south west, color=black] {$f$};
		\end{tikzpicture}
		\ =\
		\begin{tikzpicture}[anchorbase]
		\draw[->] (0,0) -- (1,1);
		\draw[<-] (1,0) -- (0,1);
		\bluedot{(0.25,.75)} node [anchor=north east, color=black] {$f$};
		\end{tikzpicture}\ ,
		\end{equation}
	\end{minipage}%
	\noindent\begin{minipage}{0.5\linewidth}
		\begin{equation} \label{rel:tokenslide-right-crossing-up-right}
		\begin{tikzpicture}[anchorbase]
		\draw[->] (0,0) -- (1,1);
		\draw[<-] (1,0) -- (0,1);
		\bluedot{(.25,.25)} node [anchor=south east, color=black] {$f$};
		\end{tikzpicture}
		\ =\
		\begin{tikzpicture}[anchorbase]
		\draw[->] (0,0) -- (1,1);
		\draw[<-] (1,0) -- (0,1);
		\bluedot{(0.75,.75)} node [anchor=north west, color=black] {$f$};
		\end{tikzpicture}\ ,
		\end{equation}
	\end{minipage}%

	\noindent\begin{minipage}{0.5\linewidth}
		\begin{equation} \label{rel:tokenslide-left-crossing-up-left}
		\begin{tikzpicture}[anchorbase]
		\draw[<-] (0,0) -- (1,1);
		\draw[->] (1,0) -- (0,1);
		\bluedot{(.75,.25)} node [anchor=south west, color=black] {$f$};
		\end{tikzpicture}
		\ =\
		\begin{tikzpicture}[anchorbase]
		\draw[<-](0,0) -- (1,1);
		\draw[->](1,0) -- (0,1);
		\bluedot{(0.25,.75)} node [anchor=north east, color=black] {$f$};
		\end{tikzpicture}\ ,
		\end{equation}
	\end{minipage}%
	\noindent\begin{minipage}{0.5\linewidth}
		\begin{equation} \label{rel:tokenslide-left-crossing-up-right}
		\begin{tikzpicture}[anchorbase]
		\draw[<-] (0,0) -- (1,1);
		\draw[->] (1,0) -- (0,1);
		\bluedot{(.25,.25)} node [anchor=south east, color=black] {$f$};
		\end{tikzpicture}
		\ =\
		\begin{tikzpicture}[anchorbase]
		\draw[<-](0,0) -- (1,1);
		\draw[->](1,0) -- (0,1);
		\bluedot{(0.75,.75)} node [anchor=north west, color=black] {$f$};
		\end{tikzpicture}\ ,
		\end{equation}
	\end{minipage}%

	\noindent \begin{minipage}{0.5\linewidth}
		\begin{equation} \label{rel:colored-right-cap-token-move}
		\begin{tikzpicture}[anchorbase]
		\draw[->] (0,0) -- (0,0.3) arc (180:0:0.3) -- (0.6,0);
		\bluedot{(0,0.3)} node[anchor=west,color=black] {$b$};
		\end{tikzpicture}
		=\begin{tikzpicture}[anchorbase]
		\draw[->] (0,0) -- (0,0.3) arc (180:0:0.3) -- (0.6,0);
		\bluedot{(0.6,0.3)} node[anchor=west,color=black] {$b$};
		\end{tikzpicture}
		\ ,
		\end{equation}
	\end{minipage}%
	\noindent\begin{minipage}{0.5\linewidth}
		\begin{equation} \label{rel:colored-right-cup-token-move}
		\begin{tikzpicture}[anchorbase]
		\draw[->] (0,1) -- (0,0.7) arc (180:360:0.3) -- (0.6,1);
		\bluedot{(0,0.7)} node[anchor=east,color=black] {$b$};
		\end{tikzpicture}
		=\begin{tikzpicture}[anchorbase]
		\draw[->] (0,1) -- (0,0.7) arc (180:360:0.3) -- (0.6,1);
		\bluedot{(0.6,0.7)} node[anchor=east,color=black] {$b$};
		\end{tikzpicture}
		\ .
		\end{equation}
	\end{minipage}\par\vspace{\belowdisplayskip}

	\begin{lem}\label{lem:colored-decorated-cups-caps}
		For $b\in B_i$, the following relations hold in $\Heis_{F,k}:$
		
		\noindent\begin{minipage}{0.5\linewidth}
		\begin{equation}\label{eq:colored-left-decorated-cup}
		\begin{tikzpicture}[>=To,baseline={([yshift=1ex]current bounding box.center)}]
		\draw[<-] (0,0.3) -- (0,0) arc (180:360:.3) -- (0.6,0.3);
		\greensquare{(0.3,-0.3)} node[anchor=north,color=black] {\squarelabel{(r,b)}};
		\bluedot{(0.6,0)} node[anchor=west,color=black] {$e_l$};
		\bluedot{(0,0)} node[anchor=east,color=black] {$e_j$};
		\end{tikzpicture}
		=\delta_{i,j}\delta_{i,l}\begin{tikzpicture}[>=To,baseline={([yshift=1ex]current bounding box.center)}]
		\draw[<-] (0,0.2) -- (0,0) arc (180:360:.3) -- (0.6,0.2);
		\greensquare{(0.3,-0.3)} node[anchor=north,color=black] {\squarelabel{(r,b)}};
		\end{tikzpicture}, \quad k>0,
		\end{equation}
		\end{minipage}
		\begin{minipage}{0.5\linewidth}
		\begin{equation}\label{eq:colored-left-decorated-cap}
		\begin{tikzpicture}[>=To,baseline={([yshift=-2ex]current bounding box.center)}]
		\draw[<-] (0,-0.3) -- (0,0) arc (180:0:.3) -- (0.6,-0.3);
		\greensquare{(0.3,0.3)} node[anchor=south,color=black] {\squarelabel{(r,b)}};
		\bluedot{(0.6,0)} node[anchor=west,color=black] {$e_l$};
		\bluedot{(0,0)} node[anchor=east,color=black] {$e_j$};
		\end{tikzpicture}
		=\delta_{i,j}\delta_{i,l}\begin{tikzpicture}[>=To,baseline={([yshift=-2ex]current bounding box.center)}]
		\draw[<-] (0,-0.2) -- (0,0) arc (180:0:.3) -- (0.6,-0.2);
		\greensquare{(0.3,0.3)} node[anchor=south,color=black] {\squarelabel{(r,b)}};
		\end{tikzpicture}, \quad k<0.
		\end{equation}
		\end{minipage}
	\end{lem}
	\begin{proof}
	Attaching the morphism 
	\[
	\begin{tikzpicture}[>=To,baseline={([yshift=1ex]current bounding box.center)}]
	\draw[<-] (0,0.3) -- (0,0) arc (180:360:.3) -- (0.6,0.3);
	\greensquare{(0.3,-0.3)} node[anchor=north,color=black] {\squarelabel{(r,b)}};
	\bluedot{(0.6,0)} node[anchor=west,color=black] {$e_l$};
	\bluedot{(0,0)} node[anchor=east,color=black] {$e_j$};
	\end{tikzpicture}
	\]
	below both sides of \cref{rel:up-down-doublecross1} and using \cref{rel:inversion2-xi-greater1,rel:inversion3-xi-greater1,rel:dot-token-up-slide1,rel:tokenslide-right-crossing-up-left,rel:tokenslide-right-crossing-up-right,rel:colored-right-cap-token-move} gives \cref{eq:colored-left-decorated-cup}. A  similar argument shows that \cref{eq:colored-left-decorated-cap} holds in $\Heis_{F,k}$.
	\end{proof}
	
	\section{Alternate Presentations of $\Heis_{F,k}$}\label{section:alternate-presentation}
	
	In this section, we give two alternate presentations of $\Heis_{F,k}$. First we introduce notation. For $f\in F$ and $i=1,\dotsc,n$, define \begin{tikzpicture}[anchorbase]
	\draw[->] (0,0) node[anchor=north] {$i$}  -- (0,0.6);
	\bluedot{(0,0.3)} node[anchor=west,color=black] {$f$};
	\end{tikzpicture}:=\begin{tikzpicture}[anchorbase]
	\draw[->] (0,0)  -- (0,0.6);
	\bluedot{(0,0.3)} node[anchor=east,color=black] {$fe_i$};
	\end{tikzpicture}. We say that a strand indexed by $i$ is of $\textit{color}$ $i$. In particular, for $f=(f_1,\dotsc,f_n)\in F$, we can write 
	\[
	\begin{tikzpicture}[anchorbase]
	\draw[->] (0,0)  -- (0,0.6);
	\bluedot{(0,0.3)} node[anchor=east,color=black] {$f=(f_1,\dotsc,f_n)$};
	\end{tikzpicture}=\begin{tikzpicture}[anchorbase]
	\draw[->] (0,0) node[anchor=north] {$1$}  -- (0,0.6);
	\bluedot{(0,0.3)} node[anchor=east,color=black] {$f_1$};
	\end{tikzpicture}+\dotsb+\begin{tikzpicture}[anchorbase]
	\draw[->] (0,0) node[anchor=north] {$n$}  -- (0,0.6);
	\bluedot{(0,0.3)} node[anchor=east,color=black] {$f_n$};
	\end{tikzpicture}. 
	\]
	That is, we can write any $\beta_f\in \Heis_{F,k}$ as a sum of uni-colored strands. Additionally, we define the following notation for morphisms in $\Heis_{F,k}$. For $b\in B,$
	\[
	\begin{tikzpicture}[anchorbase]
	\draw [->](0,0) node[anchor=north] {$i$} -- (0.6,0.6);
	\draw [->](0.6,0) node[anchor=north] {$j$} -- (0,0.6);
	\end{tikzpicture}
	:=\begin{tikzpicture}[anchorbase]
	\draw [->](0,0)  -- (0.7,0.7);
	\draw [->](0.7,0) -- (0,0.7);
	\bluedot{(0.15,0.15)} node[anchor=east,color=black] {$e_i$}; 
	\bluedot{(0.55,0.15)} node[anchor=west,color=black] {$e_j$}; 
	\end{tikzpicture}
	\ , \quad
	\begin{tikzpicture}[anchorbase]
	\draw[->] (0,.2) node[anchor=south] {$i$} -- (0,0) arc (180:360:.3) -- (.6,.2);
	\end{tikzpicture}
	:=\begin{tikzpicture}[anchorbase]
	\draw[->] (0,.2) -- (0,0) arc (180:360:.3) -- (.6,.2);
	\bluedot{(0,0)} node[anchor=east,color=black] {$e_i$}; 
	\end{tikzpicture}
	\ ,\quad
	\begin{tikzpicture}[anchorbase]
	\draw[->] (0,-.2) node[anchor=north] {$i$} -- (0,0) arc (180:0:.3) -- (.6,-.2);
	\end{tikzpicture}
	:=\begin{tikzpicture}[anchorbase]
	\draw[->] (0,-.2) -- (0,0) arc (180:0:.3) -- (.6,-.2);
	\bluedot{(0,0)} node[anchor=east,color=black] {$e_i$}; 
	\end{tikzpicture}
	\ , \quad
	\begin{tikzpicture}[anchorbase]
	\draw[->] (0,0)  node[anchor=north] {$i$} -- (0,0.8);
	\redcircle{(0,0.4)};
	\end{tikzpicture}
	:=\begin{tikzpicture}[anchorbase]
	\draw[->] (0,0) -- (0,1);
	\redcircle{(0,0.7)};
	\bluedot{(0,0.3)} node[anchor= west,color=black] {$e_i$};
	\end{tikzpicture},
	\]
	\[
	\begin{tikzpicture}[>=To,baseline={([yshift=1ex]current bounding box.center)}]
	\draw[<-] (0,0.2) -- (0,0) arc (180:360:.3) -- (0.6,0.2) node[anchor=south] {$i$};
	\greensquare{(0.3,-0.3)} node[anchor=north,color=black] {\squarelabel{(r,b)}};
	\end{tikzpicture}	
	:=\begin{tikzpicture}[>=To,baseline={([yshift=1ex]current bounding box.center)}]
	\draw[<-] (0,0.3) -- (0,0) arc (180:360:.3) -- (0.6,0.3);
	\greensquare{(0.3,-0.3)} node[anchor=north,color=black] {\squarelabel{(r,b)}};
	\bluedot{(0.6,0)} node[anchor=west,color=black] {$e_i$};
	\bluedot{(0,0)} node[anchor=east,color=black] {$e_i$};
	\end{tikzpicture}, \quad k>0, 
	\ , \qquad
	\begin{tikzpicture}[>=To,baseline={([yshift=-2ex]current bounding box.center)}]
	\draw[<-] (0,-0.2) -- (0,0) arc (180:0:.3) -- (0.6,-0.2) node[anchor=north] {$i$};
	\greensquare{(0.3,0.3)} node[anchor=south,color=black] {\squarelabel{(r,b)}};
	\end{tikzpicture}
	:=\begin{tikzpicture}[>=To,baseline={([yshift=-2ex]current bounding box.center)}]
	\draw[<-] (0,-0.3) -- (0,0) arc (180:0:.3) -- (0.6,-0.3);
	\greensquare{(0.3,0.3)} node[anchor=south,color=black] {\squarelabel{(r,b)}};
	\bluedot{(0.6,0)} node[anchor=west,color=black] {$e_i$};
	\bluedot{(0,0)} node[anchor=east,color=black] {$e_i$};
	\end{tikzpicture}, \quad k<0,
	\ , \qquad 
	\begin{tikzpicture}[anchorbase]
	\draw [<-](0,0)  -- (0.6,0.6) node[anchor=south] {$i$};
	\draw [->](0.6,0) node[anchor=north] {$j$} -- (0,0.6);
	\end{tikzpicture}
	:=\begin{tikzpicture}[anchorbase]
	\draw [<-](0,0)  -- (0.7,0.7);
	\draw [->](0.7,0) -- (0,0.7);
	\bluedot{(0.55,0.55)} node[anchor=south,color=black] {$e_i$}; 
	\bluedot{(0.55,0.15)} node[anchor=north,color=black] {$e_j$}; 
	\end{tikzpicture}.
	\]
	This motivates one alternate presentation of $\Heis_{F,k}$, which we state below and prove in \cref{theo:isomorphic}.
	\begin{defin} \label{def:H2}
		The category $\Heis_{F,k}'$ is the strict $\kk$-linear monoidal category defined as follows. The objects are generated by $Q_{+}$ and $Q_{-}$ and we use juxaposition to denote tensor product. We will denote the unit object by $\mathbb{1}$. The morphisms of $\Heis_{F,k}'$ are generated by
		\[s_{i,j}=
		\begin{tikzpicture}[anchorbase]
		\draw [->](0,0) node[anchor=north] {$i$} -- (0.6,0.6);
		\draw [->](0.6,0) node[anchor=north] {$j$} -- (0,0.6);
		\end{tikzpicture}
		\colon Q_{+} Q_{+} \to Q_{+} Q_{+}
		\ , \quad
		c_i=
		\begin{tikzpicture}[anchorbase]
		\draw[->] (0,.2) node[anchor=south] {$i$} -- (0,0) arc (180:360:.3) -- (.6,.2);
		\end{tikzpicture}
		\colon \mathbb{1} \to Q_{-} Q_{+}
		\ ,\quad
		d_i=
		\begin{tikzpicture}[anchorbase]
		\draw[->] (0,-.2) node[anchor=north] {$i$} -- (0,0) arc (180:0:.3) -- (.6,-.2);
		\end{tikzpicture}
		\colon Q_{+} Q_{-} \to \mathbb{1},
		\]
		\[
		x_i=
		\begin{tikzpicture}[anchorbase]
		\draw[->] (0,0)  node[anchor=north] {$i$} -- (0,0.6);
		\redcircle{(0,0.3)};
		\end{tikzpicture}
		\colon Q_{+} \to Q_{+}	
		\ , \quad
		\beta_{(f,i)}=
		\begin{tikzpicture}[anchorbase]
		\draw[->] (0,0)  node[anchor=north] {$i$} -- (0,0.6);
		\bluedot{(0,0.3)} node[anchor=west,color=black] {$f$};
		\end{tikzpicture}
		\colon Q_{+} \to Q_{+}, \quad f\in F_i, \quad i,j=1,\dotsc,n.
		\]
		
		For $n \ge 1$, we denote the $n$-th power $x_i^n$ of $x_i$, $i=1,\dotsc,n,$ by labelling the dot with the exponent $n$:
		
		\[
		x_i^n
		=\begin{tikzpicture}[anchorbase]
		\draw[->] (0,0) node[anchor=north] {$i$} -- (0,1);
		\redcircle{(0,0.5)} node[anchor=west, color=black] {$n$};
		\end{tikzpicture}	
		\ .
		\]
		We also define, for $i,j=1,\dotsc,n$,
		\begin{equation} \label{eq:t-def2}
		t_{i,j} \colon \sQ_+ \sQ_- \to \sQ_- \sQ_+,\quad
		t_{i,j}=
		\begin{tikzpicture}[anchorbase]
		\draw [->](0,0) node[anchor=north] {$i$} -- (0.6,0.6);
		\draw [<-](0.6,0) -- (0,0.6) node[anchor=south] {$j$};
		\end{tikzpicture} 
		\ :=\
		\begin{tikzpicture}[anchorbase,scale=0.6]
		\draw[->] (0.3,0) node[anchor=north] {$i$} -- (-0.3,1);
		\draw[->] (-0.75,1)  node[anchor=south] {$j$} -- (-0.75,0.5) .. controls (-0.75,0.2) and (-0.5,0) .. (0,0.5) .. controls (0.5,1) and (0.75,0.8) .. (0.75,0.5) -- (0.75,0);
		\end{tikzpicture}.
		\end{equation}
		We impose three sets of relations:
		\begin{enumerate}[wide]
			\item \emph{Colored affine wreath product algebra relations}:  We have a homomorphisms of algebras for $i=1,\dotsc,n$,
			\begin{equation} \label{rel:token-homom2}
			F_i \to \End \sQ_{+},\quad f \mapsto \beta_{(f,i)},
			\end{equation}
			so that, in particular,
			\begin{equation} \label{rel:token-colide-up2}
			\begin{tikzpicture}[anchorbase]
			\draw[->] (0,0) node[anchor=north] {$i$} -- (0,1);
			\bluedot{(0,0.35)} node[anchor=east,color=black] {$g$};
			\bluedot{(0,0.7)} node[anchor=east,color=black] {$f$};
			\end{tikzpicture}
			\ = \
			\begin{tikzpicture}[anchorbase]
			\draw[->] (0,0) node[anchor=north] {$i$}  -- (0,1);
			\bluedot{(0,0.5)} node[anchor=west,color=black] {$fg$};
			\end{tikzpicture}
			,\quad f,g \in F_i.
			\end{equation}
			Note also that, for $a\in\kk,$ $f_j\in F_i,$ $j=1,\dotsc,m,$ $m>0,$ we have
			
			\noindent 	\begin{minipage}{0.5\linewidth}
			\begin{equation} \label{rel:token-sum-up2}
			\begin{tikzpicture}[anchorbase]
			\draw[->] (0,0) node[anchor=north] {$i$} -- (0,1);
			\bluedot{(0,0.5)} node[anchor=east,color=black] {$\left(\sum_{j=1}^{m}f_j\right)$};
			\end{tikzpicture}
			\ = \
			\sum_{j=1}^{m}\left(
			\begin{tikzpicture}[anchorbase]
			\draw[->] (0,0) node[anchor=north] {$i$} -- (0,1);
			\bluedot{(0,0.5)} node[anchor=west,color=black] {$f_j$};
			\end{tikzpicture}\right)
			\ ,
			\end{equation}	
			\end{minipage}
			\begin{minipage}{0.5\linewidth}
			\begin{equation} \label{rel:token-scale-up2}
			\begin{tikzpicture}[anchorbase]
			\draw[->] (0,0) node[anchor=north] {$i$} -- (0,1);
			\bluedot{(0,0.5)} node[anchor=east,color=black] {$\left(af_1\right)$};
			\end{tikzpicture}
			\ = \
			a\left(
			\begin{tikzpicture}[anchorbase]
			\draw[->] (0,0) node[anchor=north] {$i$} -- (0,1);
			\bluedot{(0,0.5)} node[anchor=west,color=black] {$f_1$};
			\end{tikzpicture}\right)
			\ .
			\end{equation}		
			\end{minipage}
		
			\indent Furthermore, we impose the following relations for all $i,j,l=1,\dotsc,n \colon$
			
			\noindent\begin{minipage}{0.33\linewidth}
				\begin{equation} \label{rel:braid-up2}
				\begin{tikzpicture}[anchorbase]
				\draw[->] (0,0) node[anchor=north] {$i$} -- (1,1);
				\draw[->] (1,0) node[anchor=north] {$l$} -- (0,1);
				\draw[->] (0.5,0) node[anchor=north] {$j$} .. controls (0,0.5) .. (0.5,1);
				\end{tikzpicture}
				\ =\
				\begin{tikzpicture}[anchorbase]
				\draw[->] (0,0) node[anchor=north] {$i$} -- (1,1);
				\draw[->] (1,0) node[anchor=north] {$l$} -- (0,1);
				\draw[->] (0.5,0) node[anchor=north] {$j$} .. controls (1,0.5) .. (0.5,1);
				\end{tikzpicture}\ ,
				\end{equation}
			\end{minipage}%
			\begin{minipage}{0.33\linewidth}
				\begin{equation} \label{rel:doublecross-up2}
				\begin{tikzpicture}[anchorbase]
				\draw[->] (0,0) node[anchor=north] {$i$} .. controls (0.5,0.5) .. (0,1);
				\draw[->] (0.5,0) node[anchor=north] {$j$} .. controls (0,0.5) .. (0.5,1);
				\end{tikzpicture}
				\ =\
				\begin{tikzpicture}[anchorbase]
				\draw[->] (0,0) node[anchor=north] {$i$} --(0,1);
				\draw[->] (0.5,0) node[anchor=north] {$j$} -- (0.5,1);
				\end{tikzpicture}\ ,
				\end{equation}
			\end{minipage}	
			\begin{minipage}{0.33\linewidth}
				\begin{equation} \label{rel:dot-token-up-slide2}
				\begin{tikzpicture}[anchorbase]
				\draw[->] (0,0) node[anchor=north] {$i$} -- (0,1);
				\redcircle{(0,0.3)};
				\bluedot{(0,0.6)} node[anchor=east, color=black] {$f$};
				\end{tikzpicture}
				\ =\
				\begin{tikzpicture}[anchorbase]
				\draw[->] (0,0) node[anchor=north] {$i$} -- (0,1);
				\redcircle{(0,0.6)};
				\bluedot{(0,0.3)} node[anchor=west, color=black] {$f$};
				\end{tikzpicture}\ , \quad f\in F_i,
				\end{equation}
			\end{minipage}\par\vspace{\belowdisplayskip}
			\noindent\begin{minipage}{0.5\linewidth}
				\begin{equation} \label{rel:tokenslide-up-right2}
				\begin{tikzpicture}[anchorbase]
				\draw[->] (0,0) node[anchor=north] {$i$} -- (1,1);
				\draw[->] (1,0) node[anchor=north] {$j$} -- (0,1);
				\bluedot{(.25,.25)} node [anchor=south east, color=black] {$f$};
				\end{tikzpicture}
				\ =\
				\begin{tikzpicture}[anchorbase]
				\draw[->](0,0) node[anchor=north] {$i$} -- (1,1);
				\draw[->](1,0)  node[anchor=north] {$j$} -- (0,1);
				\bluedot{(0.75,.75)} node [anchor=north west, color=black] {$f$};
				\end{tikzpicture}\ , \quad f\in F_i,
				\end{equation}
			\end{minipage}%
			\begin{minipage}{0.5\linewidth}	
				\begin{equation} \label{rel:dotslide1.2}
				\begin{tikzpicture}[anchorbase]
				\draw[->] (0,0) node[anchor=north] {$i$} -- (1,1);
				\draw[->] (1,0) node[anchor=north] {$j$} -- (0,1);
				\redcircle{(0.25,.75)};				
				\end{tikzpicture}
				\ -\
				\begin{tikzpicture}[anchorbase]
				\draw[->] (0,0) node[anchor=north] {$i$} -- (1,1);
				\draw[->] (1,0) node[anchor=north] {$j$} -- (0,1);
				\redcircle{(.75,.25)};
				\end{tikzpicture}
				\ =\
				\delta_{i,j} \sum_{b \in B_i}
				\begin{tikzpicture}[anchorbase]
				\draw[->] (0,0) -- (0,1);
				\draw[->] (0.5,0) -- (0.5,1);
				\bluedot{(0,0.5)} node[anchor=east, color=black] {$\chk{b}$};
				\bluedot{(0.5,0.5)} node[anchor=west, color=black] {$b$};
				\end{tikzpicture}\ .
				\end{equation}
			\end{minipage}

			It follows that we also have the relations:

			\noindent\begin{minipage}{0.4\linewidth}
				\begin{equation} \label{rel:tokenslide-up-left2}
				\begin{tikzpicture}[anchorbase]
				\draw[->] (0,0) node[anchor=north] {$i$} -- (1,1);
				\draw[->] (1,0) node[anchor=north] {$j$}-- (0,1);
				\bluedot{(.75,.25)} node [anchor=south west, color=black] {$f$};
				\end{tikzpicture}
				\ =\
				\begin{tikzpicture}[anchorbase]
				\draw[->] (0,0) node[anchor=north] {$i$}-- (1,1);
				\draw[->] (1,0) node[anchor=north] {$j$} -- (0,1);
				\bluedot{(0.25,.75)} node [anchor=north east, color=black] {$f$};
				\end{tikzpicture}\ ,  \quad f\in F_i,
				\end{equation}
			\end{minipage}%
			\begin{minipage}{0.6\linewidth}
				\begin{equation} \label{rel:dotslide2}
				\begin{tikzpicture}[anchorbase]
				\draw[->] (0,0) node[anchor=north] {$i$} -- (1,1);
				\draw[->] (1,0) node[anchor=north] {$j$} -- (0,1);
				\redcircle{(0.25,.25)};
				\end{tikzpicture}
				\ -\
				\begin{tikzpicture}[anchorbase]
				\draw[->] (0,0) node[anchor=north] {$i$} -- (1,1);
				\draw[->] (1,0) node[anchor=north] {$j$} -- (0,1);
				\redcircle{(.75,.75)};
				\end{tikzpicture}
				\ =\
				\delta_{i,j}\sum_{b \in B_i}
				\begin{tikzpicture}[anchorbase]
				\draw[->] (0,0) node[anchor=north] {$i$}-- (0,1);
				\draw[->] (0.5,0)  node[anchor=north] {$i$} -- (0.5,1);
				\bluedot{(0,0.5)} node[anchor=east, color=black] {$b$};
				\bluedot{(0.5,0.5)} node[anchor=west, color=black] {$\chk{b}$};					
				\end{tikzpicture}\ .
				\end{equation}
			\end{minipage}
			
			\item \emph{Right adjunction relations}:  We impose the following relations for all $i=1,\dotsc,n$:
			
			\noindent\begin{minipage}{0.5\linewidth}
				\begin{equation} \label{rel:right-adjunction-up2}
				\begin{tikzpicture}[anchorbase]
				\draw[->] (0,0) node[anchor=north] {$i$} -- (0,0.6) arc(180:0:0.2) -- (0.4,0.4) arc(180:360:0.2) -- (0.8,1);
				\end{tikzpicture}
				\ =\
				\begin{tikzpicture}[anchorbase]
				\draw[->] (0,0) node[anchor=north] {$i$} -- (0,1);
				\end{tikzpicture}\ ,
				\end{equation}
			\end{minipage}%
			\begin{minipage}{0.5\linewidth}
				\begin{equation} \label{rel:right-adjunction-down2}
				\begin{tikzpicture}[anchorbase]
				\draw[->] (0,1) node[anchor=south] {$i$} -- (0,0.4) arc(180:360:0.2) -- (0.4,0.6) arc(180:0:0.2) -- (0.8,0);
				\end{tikzpicture}
				\ =\
				\begin{tikzpicture}[anchorbase]
				\draw[<-] (0,0) -- (0,1) node[anchor=south] {$i$};
				\end{tikzpicture}\ .
				\end{equation}
			\end{minipage}\par\vspace{\belowdisplayskip}
			
			\item \emph{Inversion relation}: 
			Suppose $i\neq j$, $i,j=1,\dotsc,n$. Then $t_{i,j}$ is invertible in $\Heis_{F,k}$.	
			Additionally, the following matrix of morphisms is an isomorphism in the additive envelope of $\Heis_{F,k}$ for $i=1,\dotsc,n$:
			\begin{gather} \label{eq:inversion-relation-pos-l2}
			\left[
			\begin{tikzpicture}[anchorbase]
			\draw [->](0,0) node[anchor=north] {$i$} -- (0.6,0.6);
			\draw [<-](0.6,0) -- (0,0.6) node[anchor=south] {$i$};
			\end{tikzpicture}
			\quad
			\begin{tikzpicture}[anchorbase]
			\draw[->] (0,0) node[anchor=north] {$i$} -- (0,0.7) arc (180:0:.3) -- (0.6,0);
			\redcircle{(0,0.6)} node[anchor=west,color=black] {$r$};
			\bluedot{(0,0.2)} node[anchor=west,color=black] {$\chk{b}$};
			\end{tikzpicture}\ ,\
			0 \le r \le k-1,\ b \in B_i
			\right]^T
			\ \colon \sQ_{+}\sQ_{-} \to \sQ_{-}\sQ_{+} \oplus \one ^{\oplus k \dim F_i} \quad \text{if } k \ge 0,
			\\ \label{eq:inversion-relation-neg-l2}
			\left[
			\begin{tikzpicture}[anchorbase]
			\draw [->](0,0)  node[anchor=north] {$i$} -- (0.6,0.6);
			\draw [<-](0.6,0) -- (0,0.6)  node[anchor=south] {$i$};
			\end{tikzpicture}
			\quad
			\begin{tikzpicture}[anchorbase]
			\draw[->] (0,1)  node[anchor=south] {$i$} -- (0,0.3) arc (180:360:.3) -- (0.6,1);
			\redcircle{(0.6,0.7)} node[anchor=east,color=black] {$r$};
			\bluedot{(0.6,0.3)} node[anchor=east,color=black] {$\chk{b}$};
			\end{tikzpicture}\ ,\
			0 \le r \le -k-1,\ b \in B_i
			\right]
			\ \colon \sQ_{+} \sQ_{-} \oplus \one^{\oplus (-k \dim F_i)} \to \sQ_{-} \sQ_{+} \quad \text{if } k < 0.
			\end{gather}
			(The matrices \cref{eq:inversion-relation-pos-l2} are of size $(1 + k \dim F_i) \times 1$, while the matrices \cref{eq:inversion-relation-neg-l2} are of size $1 \times (1 + k \dim F_i)$.)  Note that these conditions are independent of the choice of basis $B_i$ of $F_i$.
		\end{enumerate}
	\end{defin}
	
	For the sake of having notation, we define
	\begin{equation}\label{eq:left-crossing2}
	t_{i,j}' =
	\begin{tikzpicture}[anchorbase]
	\draw [<-](0,0) -- (0.6,0.6) node[anchor=south] {$i$};
	\draw [->](0.6,0) node[anchor=north] {$j$} -- (0,0.6);
	\end{tikzpicture}
	\ \colon \sQ_- \sQ_+ \to \sQ_+ \sQ_-
	\ ,
	\end{equation}
	and
	\begin{equation} \label{eq:decorated-left-cap-cup2}
	\begin{tikzpicture}[>=To,baseline={([yshift=1ex]current bounding box.center)}]
	\draw[<-] (0,0.2) -- (0,0) arc (180:360:.3) -- (0.6,0.2) node[anchor=south] {$i$};
	\greensquare{(0.3,-0.3)} node[anchor=north,color=black] {\squarelabel{(r,b)}};
	\end{tikzpicture}
	\ \colon \one \to \sQ_+ \sQ_-
	\ ,\qquad
	\begin{tikzpicture}[>=To,baseline={([yshift=-2ex]current bounding box.center)}]
	\draw[<-] (0,-0.2) -- (0,0) arc (180:0:.3) -- (0.6,-0.2) node[anchor=north] {$i$};
	\greensquare{(0.3,0.3)} node[anchor=south,color=black] {\squarelabel{(r,b)}};
	\end{tikzpicture}
	\ \colon \sQ_+ \sQ_- \to \one, \quad b\in B_i,
	\end{equation} 
	for all $i,j=1,\dotsc,n$, and $0 \le r < k$ or $0 \le r < -k,$ respectively, by declaring that
	\begin{equation} \label{eq:inversion-leftcup-def2}
	\left[
	\begin{tikzpicture}[anchorbase]
	\draw [<-](0,0) -- (0.6,0.6) node[anchor=south] {$i$};
	\draw [->](0.6,0) node[anchor=north] {$i$} -- (0,0.6);
	\end{tikzpicture}
	\quad
	\begin{tikzpicture}[>=stealth,baseline={([yshift=1ex]current bounding box.center)}]
	\draw[<-] (0,0.2) -- (0,0) arc (180:360:.3) -- (0.6,0.2) node[anchor=south] {$i$};
	\greensquare{(0.3,-0.3)} node[anchor=north,color=black] {\squarelabel{(r,b)}};
	\end{tikzpicture},\
	0 \le r \le k-1,\ b \in B_i
	\right]
	=
	\left(\left[
	\begin{tikzpicture}[anchorbase]
	\draw [->](0,0) node[anchor=north] {$i$} -- (0.6,0.6);
	\draw [<-](0.6,0) -- (0,0.6) node[anchor=south] {$i$};
	\end{tikzpicture}
	\quad
	\begin{tikzpicture}[anchorbase]
	\draw[->] (0,0) node[anchor=north] {$i$} -- (0,0.7) arc (180:0:.3) -- (0.6,0);
	\redcircle{(0,0.6)} node[anchor=west,color=black] {$r$};
	\bluedot{(0,0.2)} node[anchor=west,color=black] {$\chk{b}$};
	\end{tikzpicture}\ ,\
	0 \le r \le k-1,\ b \in B_i
	\right]^T \right)^{-1}
	\quad \text{if } k \ge 0,
	\end{equation}
	or
	\begin{equation} \label{eq:inversion-leftcap-def2}
	\left[
	\begin{tikzpicture}[anchorbase]
	\draw [<-](0,0) -- (0.6,0.6) node[anchor=south] {$i$};
	\draw [->](0.6,0)  node[anchor=north] {$i$} -- (0,0.6);
	\end{tikzpicture}
	\quad
	\begin{tikzpicture}[>=stealth,baseline={([yshift=-2ex]current bounding box.center)}]
	\draw[<-] (0,-0.2) -- (0,0) arc (180:0:.3) -- (0.6,-0.2) node[anchor=north] {$i$};
	\greensquare{(0.3,0.3)} node[anchor=south,color=black] {\squarelabel{(r,b)}};
	\end{tikzpicture},\
	0 \le r \le -k-1,\ b \in B_i
	\right]^T
	=
	\left[
	\begin{tikzpicture}[anchorbase]
	\draw [->](0,0) node[anchor=north] {$i$} -- (0.6,0.6);
	\draw [<-](0.6,0) -- (0,0.6) node[anchor=south] {$i$};
	\end{tikzpicture}
	\quad
	\begin{tikzpicture}[anchorbase]
	\draw[->] (0,1) node[anchor=south] {$i$} -- (0,0.3) arc (180:360:.3) -- (0.6,1);
	\redcircle{(0.6,0.7)} node[anchor=east,color=black] {$r$};
	\bluedot{(0.6,0.3)} node[anchor=east,color=black] {$\chk{b}$};
	\end{tikzpicture}\ ,\
	0 \le r \le -k-1,\ b \in B_i
	\right]^{-1}
	\quad \text{if } k < 0,
	\end{equation}
	when $i=1,\dotsc,n$, and
	\begin{equation}\label{eq:right-crossing}
	\begin{tikzpicture}[anchorbase]
	\draw [->](0,0) node[anchor=north] {$i$} -- (0.6,0.6);
	\draw [<-](0.6,0) -- (0,0.6) node[anchor=south] {$j$};
	\end{tikzpicture}=
	\left(\begin{tikzpicture}[anchorbase]
	\draw [<-](0,0)  -- (0.6,0.6) node[anchor=south] {$i$};
	\draw [->](0.6,0) node[anchor=north] {$j$} -- (0,0.6) ;
	\end{tikzpicture}\right)^{-1},
	\end{equation}
	when $i\neq j,$ $i,j=1,\dotsc,n$.
	As in \cite[page 11]{S17}, we extend the definition of the decorated left cups and caps by linearity in the second argument of the label. 
	
	The relations \cref{rel:up-down-doublecross2,rel:down-up-doublecross2} hold in $\Heis_{F,k}'$ for all $k.$ For $i,j=1,\dotsc,n,$
	
	\noindent \begin{minipage}{0.5\linewidth}
		\begin{equation} \label{rel:up-down-doublecross2}
		\begin{tikzpicture}[anchorbase]
		\draw[->] (0,0) node[anchor=north] {$i$} -- (0,1);
		\draw[<-] (0.5,0) -- (0.5,1) node[anchor=south] {$j$};
		\end{tikzpicture}
		\ =\
		\begin{tikzpicture}[anchorbase]
		\draw[->] (0,0) node[anchor=north] {$i$} .. controls (0.5,0.5) .. (0,1);
		\draw[<-] (0.5,0) .. controls (0,0.5) .. (0.5,1) node[anchor=south] {$j$};
		\end{tikzpicture}
		\ + \delta_{i,j}\sum_{r=0}^{k-1} \sum_{b \in B_i}\
		\begin{tikzpicture}[anchorbase]
		\draw[->] (0,0) node[anchor=north] {$i$} -- (0,0.7) arc (180:0:0.3) -- (0.6,0);
		\draw[<-] (0,2) -- (0,1.8) arc (180:360:0.3) -- (0.6,2) node[anchor=south] {$i$};
		\redcircle{(0,0.6)} node[anchor=west,color=black] {$r$};
		\bluedot{(0,0.2)} node[anchor=west,color=black] {$\chk{b}$};
		\greensquare{(0.3,1.5)} node[anchor=north west,color=black] {\squarelabel{(r,b)}};
		\end{tikzpicture}
		\ ,
		\end{equation}
	\end{minipage}%
	\noindent\begin{minipage}{0.5\linewidth}
		\begin{equation} \label{rel:down-up-doublecross2}
		\begin{tikzpicture}[anchorbase]
		\draw[<-] (0,0) -- (0,1) node[anchor=south] {$i$};
		\draw[->] (0.5,0) node[anchor=north] {$j$} -- (0.5,1);
		\end{tikzpicture}
		\ =\
		\begin{tikzpicture}[anchorbase]
		\draw[<-] (0,0) .. controls (0.5,0.5) .. (0,1) node[anchor=south] {$i$};
		\draw[->] (0.5,0) node[anchor=north] {$j$}.. controls (0,0.5) .. (0.5,1);
		\end{tikzpicture}
		\ + \delta_{i,j}\sum_{r=0}^{-k-1} \sum_{b \in B_i}\
		\begin{tikzpicture}[anchorbase]
		\draw[->] (0,2) node[anchor=south] {$i$} -- (0,1.3) arc (180:360:0.3) -- (0.6,2);
		\draw[<-] (0,0) -- (0,0.2) arc (180:0:0.3) -- (0.6,0) node[anchor=north] {$i$};
		\redcircle{(0.6,1.7)} node[anchor=east,color=black] {$r$};
		\bluedot{(0.6,1.3)} node[anchor=east,color=black] {$\chk{b}$};
		\greensquare{(0.3,0.5)} node[anchor=west,color=black] {\squarelabel{(r,b)}};
		\end{tikzpicture}
		\ .
		\end{equation}
	\end{minipage}\par\vspace{\belowdisplayskip}
	
	\noindent When $k > 0$, the following relations, \cref{rel:inversion1-xi-greater2,rel:inversion2-xi-greater2,rel:inversion3-xi-greater2}, hold in $\Heis_{F,k}'$. For all $0 \le r,s \le k-1$, $i,j=1,\dotsc,n$, and $b,c\in B_i$,
	
	\noindent\begin{minipage}{0.33\linewidth}
		\begin{equation} \label{rel:inversion1-xi-greater2}
		\begin{tikzpicture}[anchorbase]
		\draw[<-] (0,0) -- (0,1) node[anchor=south] {$i$};
		\draw[->] (0.5,0)  node[anchor=north] {$j$} -- (0.5,1);
		\end{tikzpicture}
		\ =\
		\begin{tikzpicture}[anchorbase]
		\draw[<-] (0,0) .. controls (0.5,0.5) .. (0,1)  node[anchor=south] {$i$};
		\draw[->] (0.5,0)  node[anchor=north] {$j$} .. controls (0,0.5) .. (0.5,1);
		\end{tikzpicture},
		\end{equation}
	\end{minipage}
	\begin{minipage}{0.33\linewidth}
		\begin{equation} \label{rel:inversion2-xi-greater2}
		\begin{tikzpicture}[anchorbase]
		\draw[->] (0,0)  node[anchor=south] {$i$} .. controls (0.3,-0.3) and (0.6,-0.3) .. (0.6,-0.6) arc(360:180:0.3) .. controls (0,-0.3) and (0.3,-0.3) .. (0.6,0);
		\greensquare{(0.3,-0.9)} node[anchor=north,color=black] {$(r,b)$};
		\end{tikzpicture}
		\ =\
		\begin{tikzpicture}[anchorbase]
		\draw[->] (0.6,-0.5)  node[anchor=north] {$i$} .. controls (0.1,0) and (0,-0.1) .. (0,0.3) -- (0,0.7) arc (180:0:.3) -- (0.6,0.3) .. controls (0.6,-0.1) and (0.5,0) .. (0,-0.5);
		\redcircle{(0,0.6)} node[anchor=west,color=black] {$r$};
		\bluedot{(0,0.3)} node[anchor=west,color=black] {$\chk{b}$};
		\end{tikzpicture}
		\ =0,
		\end{equation}
	\end{minipage}%
	\begin{minipage}{0.33\linewidth}
		\begin{equation} \label{rel:inversion3-xi-greater2}
		\begin{tikzpicture}[anchorbase]
		\draw[<-] (0,0.3) arc(90:450:0.3);
		\draw{(0,0.3)} node[anchor=south] {$i$};
		\redcircle{(-0.3,0)} node[anchor=east,color=black] {$r$};
		\bluedot{(0.3,0)} node[anchor=west,color=black] {$\chk{b}$};
		\greensquare{(0,-0.3)} node[anchor=north,color=black] {$(s,c)$};
		\end{tikzpicture}
		= \delta_{r,s} \delta_{b,c}.
		\end{equation}
	\end{minipage}\par\vspace{\belowdisplayskip}
	
	\noindent When $k < 0$, the following relations, \cref{rel:inversion1-xi-smaller2,rel:inversion2-xi-smaller2,rel:inversion3-xi-smaller2}, hold in $\Heis_{F,k}'$. For all $0 \le r,s \le -k-1$, $i=1,\dotsc,n$, and $b,c\in B_i$,
	
	\noindent\begin{minipage}{0.33\linewidth}
		\begin{equation} \label{rel:inversion1-xi-smaller2}
		\begin{tikzpicture}[anchorbase]
		\draw[->] (0,0)  node[anchor=north] {$i$}-- (0,1);
		\draw[<-] (0.5,0) -- (0.5,1)  node[anchor=south] {$j$};
		\end{tikzpicture}
		\ =\
		\begin{tikzpicture}[anchorbase]
		\draw[->] (0,0)  node[anchor=north] {$i$} .. controls (0.5,0.5) .. (0,1);
		\draw[<-] (0.5,0) .. controls (0,0.5) .. (0.5,1)  node[anchor=south] {$j$};
		\end{tikzpicture},
		\end{equation}
	\end{minipage}
	\begin{minipage}{0.33\linewidth}
		\begin{equation} \label{rel:inversion2-xi-smaller2}
		\begin{tikzpicture}[anchorbase]
		\draw[->] (0,0)  node[anchor=north] {$i$} .. controls (0.3,0.3) and (0.6,0.3) .. (0.6,0.6) arc(0:180:0.3) .. controls (0,0.3) and (0.3,0.3) .. (0.6,0);
		\greensquare{(0.3,0.9)} node[anchor=south,color=black] {$(r,b)$};
		\end{tikzpicture}
		\ =\
		\begin{tikzpicture}[anchorbase]
		\draw[->] (0.6,0.5)  node[anchor=south] {$i$} .. controls (0.1,0) and (0,0.1) .. (0,-0.3) -- (0,-0.7) arc (180:360:.3) -- (0.6,-0.3) .. controls (0.6,0.1) and (0.5,0) .. (0,0.5);
		\redcircle{(0.6,-0.3)} node[anchor=east,color=black] {$r$};
		\bluedot{(0.6,-0.6)} node[anchor=west,color=black] {$\chk{b}$};
		\end{tikzpicture}
		\  = 0,
		\end{equation}
	\end{minipage}%
	\begin{minipage}{0.33\linewidth}
		\begin{equation} \label{rel:inversion3-xi-smaller2}
		\begin{tikzpicture}[anchorbase]
		\draw[->] (0,-0.3) arc(270:630:0.3);
		\draw{(0,-0.3)} node[anchor=north] {$i$};
		\redcircle{(0.3,0)} node[anchor=west,color=black] {$r$};
		\bluedot{(-0.3,0)} node[anchor=east,color=black] {$\chk{b}$};
		\greensquare{(0,0.3)} node[anchor=south,color=black] {$(s,c)$};
		\end{tikzpicture}
		= \delta_{r,s} \delta_{b,c}.
		\end{equation}
	\end{minipage}\par\vspace{\belowdisplayskip}
	\noindent 	The inversion relation in $\Heis_{F,k}'$ is equivalent to the relations \cref{rel:up-down-doublecross2,rel:inversion1-xi-greater2,rel:inversion2-xi-greater2,rel:inversion3-xi-greater2} when $k>0$, the relations \cref{rel:down-up-doublecross2,rel:inversion1-xi-smaller2,rel:inversion2-xi-smaller2,rel:inversion3-xi-smaller2} when $k<0,$ and the relations \cref{rel:down-up-doublecross2,rel:up-down-doublecross2} when $k=0.$ Note that when $k=0$, the relations \cref{rel:down-up-doublecross2,rel:up-down-doublecross2} imply that, for $i,j=1,\dotsc,n,$
	\[
	t'_{i,j}=
	\begin{tikzpicture} [anchorbase]
	\draw[<-] (0,0) -- (1,1) node[anchor=south] {$i$} ;
	\draw[->] (1,0) node[anchor=north] {$j$} -- (0,1);
	\end{tikzpicture}
	\colon \sQ_-\sQ_+\to \sQ_+\sQ_-
	\ ,
	\] 
	
	\noindent is a generating morphism that is inverse to $t_{i,j}$, so that
	
	\[
	\begin{tikzpicture}[anchorbase]
	\draw[->] (0,0) node[anchor=north] {$i$} .. controls (0.5,0.5) .. (0,1);
	\draw[<-] (0.5,0) .. controls (0,0.5) .. (0.5,1) node[anchor=south] {$j$};
	\end{tikzpicture}
	\ =\
	\begin{tikzpicture}[anchorbase]
	\draw[->] (0,0) node[anchor=north] {$i$} -- (0,1);
	\draw[<-] (0.5,0) -- (0.5,1) node[anchor=south] {$j$};
	\end{tikzpicture}
	,\qquad
	\begin{tikzpicture}[anchorbase]
	\draw[<-] (0,0) .. controls (0.5,0.5) .. (0,1) node[anchor=south] {i};
	\draw[->] (0.5,0) node[anchor=north] {$j$} .. controls (0,0.5) .. (0.5,1);
	\end{tikzpicture}
	\ =\
	\begin{tikzpicture}[anchorbase]
	\draw[<-] (0,0) -- (0,1) node[anchor=south] {$i$};
	\draw[->] (0.5,0) node[anchor=north] {$j$} -- (0.5,1);
	\end{tikzpicture}
	\ ,
	\]
	\noindent for all $i,j=1,\dotsc,n.$
	\begin{theo}\label{theo:isomorphic}
		The categories $\Heis_{F,k}$ and $\Heis_{F,k}'$ are isomorphic.
	\end{theo}
	\begin{proof}
		We define a monoidal functor $F:\Heis_{F,k}'\to\Heis_{F,k}$ as follows. On objects define 
		\[
		F\left(\sQ_\pm\right)=\sQ_\pm,
		\]
		and on morphisms, for $f_i\in F_i$, $i,j=1,\dotsc,n$, define
		\[
		F\left(s_{i,j}\right)=s_{i,j}, \quad F\left(c_i\right)=c_i, \quad F\left(d_i\right)=d_i, \quad F\left(x_i\right)=x_i, \quad F\left(\beta_{(f_i,i)}\right)=\beta_{f_i},  \quad F\left(t_{i,j}\right)=t_{i,j},\]
		\[ F\left(
		\begin{tikzpicture}[>=To,baseline={([yshift=-2ex]current bounding box.center)}]
		\draw[<-] (0,-0.2) -- (0,0) arc (180:0:.3) -- (0.6,-0.2) node[anchor=north] {$i$};
		\greensquare{(0.3,0.3)} node[anchor=south,color=black] {\squarelabel{(r,f_i)}};
		\end{tikzpicture}\right)=
		\begin{tikzpicture}[>=To,baseline={([yshift=-2ex]current bounding box.center)}]
		\draw[<-] (0,-0.2) -- (0,0) arc (180:0:.3) -- (0.6,-0.2) node[anchor=north] {$i$};
		\greensquare{(0.3,0.3)} node[anchor=south,color=black] {\squarelabel{(r,f_i)}};
		\end{tikzpicture}, \quad  F\left(\begin{tikzpicture}[>=To,baseline={([yshift=1ex]current bounding box.center)}]
		\draw[<-] (0,0.2) -- (0,0) arc (180:360:.3) -- (0.6,0.2) node[anchor=south] {$i$};
		\greensquare{(0.3,-0.3)} node[anchor=north,color=black] {\squarelabel{(r,f_i)}};
		\end{tikzpicture}\right)=\begin{tikzpicture}[>=To,baseline={([yshift=1ex]current bounding box.center)}]
		\draw[<-] (0,0.2) -- (0,0) arc (180:360:.3) -- (0.6,0.2) node[anchor=south] {$i$};
		\greensquare{(0.3,-0.3)} node[anchor=north,color=black] {\squarelabel{(r,f_i)}};
		\end{tikzpicture}, \quad F\left(t_{i,j}'\right)=t_{i,j}',\]
		and extend $F$ so that it is linear and monoidal.
		We define another monoidal functor $G:\Heis_{F,k}\to\Heis_{F,k}'$ as follows. On objects define 
		\[
		G\left(\sQ_\pm\right)=\sQ_\pm,
		\]
		and on morphisms, define
		\[
		G\left(s\right)=\sum_{i,j=1}^ns_{i,j}, \quad G\left(c\right)=\sum_{i=1}^nc_i,\quad G\left(d\right)=\sum_{i=1}^nd_i, \quad G\left(\beta_{(f_1,\dotsc,f_n)}\right)=\sum_{i=1}^n\beta_{(f_i,i)}, \quad G\left(t\right)=\sum_{i,j=1}^nt_{i,j},
		\]
		\[
		G\left(
		\begin{tikzpicture}[>=To,baseline={([yshift=-2ex]current bounding box.center)}]
		\draw[<-] (0,-0.2) -- (0,0) arc (180:0:.3) -- (0.6,-0.2);
		\greensquare{(0.3,0.3)} node[anchor=south,color=black] {\squarelabel{(r,(f_1,\dotsc,f_n))}};
		\end{tikzpicture}\right)=\sum_{i=1}^n
		\begin{tikzpicture}[>=To,baseline={([yshift=-2ex]current bounding box.center)}]
		\draw[<-] (0,-0.2) -- (0,0) arc (180:0:.3) -- (0.6,-0.2) node[anchor=north] {$i$};
		\greensquare{(0.3,0.3)} node[anchor=south,color=black] {\squarelabel{(r,f_i)}};
		\end{tikzpicture}, \quad  G\left(\begin{tikzpicture}[>=To,baseline={([yshift=1ex]current bounding box.center)}]
		\draw[<-] (0,0.2) -- (0,0) arc (180:360:.3) -- (0.6,0.2);
		\greensquare{(0.3,-0.3)} node[anchor=north,color=black] {\squarelabel{(r,(f_1,\dotsc,f_n))}};
		\end{tikzpicture}\right)=\sum_{i=1}^n\begin{tikzpicture}[>=To,baseline={([yshift=1ex]current bounding box.center)}]
		\draw[<-] (0,0.2) -- (0,0) arc (180:360:.3) -- (0.6,0.2) node[anchor=south] {$i$};
		\greensquare{(0.3,-0.3)} node[anchor=north,color=black] {\squarelabel{(r,f_i)}};
		\end{tikzpicture}, \quad G\left(t'\right)=\sum_{i,j=1}^nt_{i,j}',
		\]
		and extend $G$ so that it is linear and monoidal.
		We claim that $F$ and $G$ are strict $\kk$-linear monoidal functors that are two-sided inverses of each other. If $G$ and $F$ exist (i.e. if they are well-defined), it is clear that $G\circ F=id_{\Heis_{F,k}'}$ and $F\circ G=id_{\Heis_{F,k}}$.	
		So it is enough to show that $F$ and $G$ are well-defined. To do this, we show that $F$ preserves the defining relations of $\Heis_{F,k}'$ and $G$ preserves the defining relations of $\Heis_{F,k}$.
		That is, we show that $G$ preserves the relations
		\cref{rel:token-homom1,rel:braid-up1,rel:doublecross-up1,rel:dot-token-up-slide1,rel:tokenslide-up-right1,rel:dotslide1.1,rel:right-adjunction-up1,rel:right-adjunction-down1}, and the inversion relation in $\Heis_{F,k}$, and that $F$ preserves the relations  \cref{rel:token-homom2,rel:braid-up2,rel:doublecross-up2,rel:dot-token-up-slide2,rel:tokenslide-up-right2,rel:dotslide1.2,rel:right-adjunction-up2,rel:right-adjunction-down2}, and the inversion relation in $\Heis_{F,k}'$.
		We first show that $F$ preserves the defining relations of $\Heis_{F,k}'.$
		
		The functor $F$ preserves the relations \cref{rel:token-homom2,rel:braid-up2,rel:doublecross-up2,rel:dot-token-up-slide2,rel:tokenslide-up-right2,rel:dotslide1.2,rel:right-adjunction-up2,rel:right-adjunction-down2} in $\Heis_{F,k}'$. This is seen by attaching the appropriate idempotents above or below both sides of the relations \cref{rel:token-colide-up1,rel:token-sum-up1,rel:token-scale-up1,rel:braid-up1,rel:doublecross-up1,rel:dot-token-up-slide1,rel:tokenslide-up-right1,rel:dotslide1.1,rel:right-adjunction-up1,rel:right-adjunction-down1} in $\Heis_{F,k},$ and noting that the resulting equations are the images of \cref{rel:token-colide-up2,rel:token-sum-up2,rel:token-scale-up2,rel:braid-up2,rel:doublecross-up2,rel:dot-token-up-slide2,rel:tokenslide-up-right2,rel:dotslide1.2,rel:right-adjunction-up2,rel:right-adjunction-down2} under $F,$ respectively.
			
		The functor $F$ preserves the inversion relation in $\Heis_{F,k}'$.  Showing that $F$ preserves this relation in $\Heis_{F,k}'$ is equivalent to showing that $F$ preserves the relations \cref{rel:up-down-doublecross2,rel:inversion1-xi-greater2,rel:inversion2-xi-greater2,rel:inversion3-xi-greater2} when $k>0$, the relations \cref{rel:down-up-doublecross2,rel:inversion1-xi-smaller2,rel:inversion2-xi-smaller2,rel:inversion3-xi-smaller2} when $k<0,$ and the relations \cref{rel:down-up-doublecross2,rel:up-down-doublecross2} when $k=0.$  
		To show that $F$ preserves the relations \cref{rel:up-down-doublecross2,rel:down-up-doublecross2,rel:inversion1-xi-greater2,rel:inversion1-xi-smaller2} in $\Heis_{F,k}'$, we attach the appropriate idempotents above or below both sides of the relations \cref{rel:up-down-doublecross1,rel:down-up-doublecross1,rel:inversion1-xi-greater1,rel:inversion1-xi-smaller1} in $\Heis_{F,k}$, and note that these equations are the images of \cref{rel:up-down-doublecross2,rel:down-up-doublecross2,rel:inversion1-xi-greater2,rel:inversion1-xi-smaller2} under $F$, respectively. 
		Now, to show that $F$ preserves the relations \cref{rel:inversion2-xi-greater2,rel:inversion3-xi-greater2,rel:inversion2-xi-smaller2,rel:inversion3-xi-smaller2} in $\Heis_{F,k}'$, we note that the relations \cref{rel:inversion2-xi-greater1,rel:inversion3-xi-greater1,rel:inversion2-xi-smaller1,rel:inversion3-xi-smaller1} in $\Heis_{F,k}$ are the images of \cref{rel:inversion2-xi-greater2,rel:inversion3-xi-greater2,rel:inversion2-xi-smaller2,rel:inversion3-xi-smaller2} under $F$, respectively, since for all $b\in B,$ $b\in B_i$ for some $i=1,\dotsc,n$, and since \cref{lem:colored-decorated-cups-caps} holds in $\Heis_{F,k}$. 
		
		Therefore, $F$ preserves all defining relations in $\Heis_{F,k}'$. Next we show that $G$ preserves all defining relations in $\Heis_{F,k}$.
		
		The functor $G$ preserves the relation \cref{rel:braid-up1} in $\Heis_{F,k}$. This follows by summing both sides of the relation \cref{rel:braid-up2} over all $i,j,l=1,\dotsc,n$ in $\Heis_{F,k}',$ and noting that the resulting equation is the image of \cref{rel:braid-up1} under $G$.
		
		The functor $G$ preserves the relations \cref{rel:doublecross-up1,rel:tokenslide-up-right1,rel:dotslide1.1} in $\Heis_{F,k}$. This follows by summing both sides of the relations \cref{rel:doublecross-up2,rel:tokenslide-up-right2,rel:dotslide1.2}, over all $i,j=1,\dotsc,n$ in $\Heis_{F,k}'$, and noting that the resulting equations are the images of \cref{rel:doublecross-up1,rel:tokenslide-up-right1,rel:dotslide1.1} under $G$, respectively. 
		
		The functor $G$ preserves the relations \cref{rel:token-homom1,rel:dot-token-up-slide1,rel:right-adjunction-up1,rel:right-adjunction-down1} in $\Heis_{F,k}$. This follows by summing both sides of the relations \cref{rel:token-colide-up2,rel:token-sum-up2,rel:token-scale-up2,rel:dot-token-up-slide2,rel:right-adjunction-up2,rel:right-adjunction-down2} over all $i=1,\dotsc,n$ in $\Heis_{F,k}'$, and noting that the resulting equations are the images of \cref{rel:token-colide-up1,rel:token-sum-up1,rel:token-scale-up1,rel:dot-token-up-slide1,rel:right-adjunction-up1,rel:right-adjunction-down1} under $G$, respectively. 
		
		The functor $G$ preserves the inversion relation in $\Heis_{F,k}$. Showing that $G$ preserves this relation in $\Heis_{F,k}$ is equivalent to showing that $G$ preserves the relations \cref{rel:up-down-doublecross1,rel:inversion1-xi-greater1,rel:inversion2-xi-greater1,rel:inversion3-xi-greater1} when $k>0$, the relations \cref{rel:down-up-doublecross1,rel:inversion1-xi-smaller1,rel:inversion2-xi-smaller1,rel:inversion3-xi-smaller1} when $k<0,$ and the relations \cref{rel:down-up-doublecross1,rel:up-down-doublecross1} when $k=0.$
		To show that $G$ preserves the relations \cref{rel:up-down-doublecross1,rel:down-up-doublecross1,rel:inversion1-xi-greater1,rel:inversion1-xi-smaller1} in $\Heis_{F,k}$, we sum both sides of \cref{rel:up-down-doublecross2,rel:down-up-doublecross2,rel:inversion1-xi-greater2,rel:inversion1-xi-smaller2}, over all $i,j=1,\dotsc,n$ in $\Heis_{F,k}'$, and note that these equations are the images of \cref{rel:up-down-doublecross1,rel:down-up-doublecross1,rel:inversion1-xi-greater1,rel:inversion1-xi-smaller1} under $G$, respectively. Now, to show that $G$ preserves the relations \cref{rel:inversion2-xi-greater1,rel:inversion3-xi-greater1,rel:inversion2-xi-smaller1,rel:inversion3-xi-smaller1} in $\Heis_{F,k}$, we note that the relations \cref{rel:inversion2-xi-greater2,rel:inversion3-xi-greater2,rel:inversion2-xi-smaller2,rel:inversion3-xi-smaller2} in $\Heis_{F,k}'$ are the images of \cref{rel:inversion2-xi-greater1,rel:inversion3-xi-greater1,rel:inversion2-xi-smaller1,rel:inversion3-xi-smaller1} under $G$, respectively, since for all $b\in B$, $b\in B_i$ for some $i=1,\dotsc,n$, and since \cref{lem:colored-decorated-cups-caps} holds in $\Heis_{F,k}$. Therefore, $G$ preserves all defining relations in $\Heis_{F,k}$. This completes the proof.
	\end{proof}
	
	We recall the following alternate presentation of $\Heis_{F,k}$ introduced in \cite[Theorem 1.2]{S17}, and impose some additional assumptions. Namely, we assume the trace map of $F$ is symmetric, we ignore gradings, and we are in the non-super setting.
	\begin{theo} \label{theo:alternate-presentation1}
		There are unique morphisms $c' \colon \one \to \sQ_+ \sQ_-$ and $d' \colon \sQ_- \sQ_+ \to \one$ in $\Heis_{F,k}$, drawn as
		\[
		c' =
		\begin{tikzpicture}[anchorbase]
		\draw[<-] (0,.2) -- (0,0) arc (180:360:.3) -- (.6,.2);
		\end{tikzpicture}
		\ ,\qquad
		d' =
		\begin{tikzpicture}[anchorbase]
		\draw[<-] (0,-.2) -- (0,0) arc (180:0:.3) -- (.6,-.2);
		\end{tikzpicture}
		\ ,
		\]
		such that the following relations hold:
		\begin{equation} \label{theo-eq:doublecross-up-down1}
		\begin{tikzpicture}[anchorbase]
		\draw[->] (0,0) .. controls (0.5,0.5) .. (0,1);
		\draw[<-] (0.5,0) .. controls (0,0.5) .. (0.5,1);
		\end{tikzpicture}
		\ =\
		\begin{tikzpicture}[anchorbase]
		\draw[->] (0,0) -- (0,1);
		\draw[<-] (0.5,0) -- (0.5,1);
		\end{tikzpicture}
		\ + \sum_{r,s \ge 0} \sum_{a,b \in B}
		\begin{tikzpicture}[anchorbase]
		\draw[->] (0,0) -- (0,0.7) arc (180:0:0.3) -- (0.6,0);
		\draw[<-] (0,2.1) -- (0,1.8) arc (180:360:0.3) -- (0.6,2.1);
		\redcircle{(0,0.6)} node[anchor=west,color=black] {$r$};
		\bluedot{(0,0.2)} node[anchor=west,color=black] {$\chk{b}$};
		\redcircle{(0.6,1.8)} node[anchor=west,color=black] {$s$};
		\bluedot{(0,1.8)} node[anchor=east,color=black] {$a$};
		\draw[->] (1,1.55) arc(90:450:0.3);
		\bluedot{(1.3,1.25)} node[anchor=west,color=black] {$\chk{a} b$};
		\redcircle{(0.7,1.25)} node[anchor=east,color=black] {\dotlabel{-r-s-2}};
		\end{tikzpicture}
		\left( =
		\begin{tikzpicture}[anchorbase]
		\draw[->] (0,0) -- (0,1);
		\draw[<-] (0.5,0) -- (0.5,1);
		\end{tikzpicture}
		\ + \delta_{k,1} \sum_{b \in B}
		\begin{tikzpicture}[anchorbase]
		\draw[<-] (0,1.1) -- (0,0.9) arc(180:360:0.3) -- (0.6,1.1);
		\draw[->] (0,-0.1) -- (0,0.1) arc(180:0:0.3) -- (0.6,-0.1);
		\bluedot{(0,0.1)} node[anchor=west,color=black] {$\chk{b}$};
		\bluedot{(0,0.9)} node[anchor=west,color=black] {$b$};
		\end{tikzpicture}
		\text{ if } k \le 1 \right),
		\end{equation}
		\begin{equation} \label{theo-eq:doublecross-down-up1}
		\begin{tikzpicture}[anchorbase]
		\draw[<-] (0,0) .. controls (0.5,0.5) .. (0,1);
		\draw[->] (0.5,0) .. controls (0,0.5) .. (0.5,1);
		\end{tikzpicture}
		\ =\
		\begin{tikzpicture}[anchorbase]
		\draw[<-] (0,0) -- (0,1);
		\draw[->] (0.5,0) -- (0.5,1);
		\end{tikzpicture}
		\ + \sum_{r,s \ge 0} \sum_{a,b \in B} 
		\begin{tikzpicture}[anchorbase]
		\draw[->] (0,2) -- (0,1.3) arc (180:360:0.3) -- (0.6,2);
		\draw[<-] (0,-0.1) -- (0,0.2) arc (180:0:0.3) -- (0.6,-0.1);
		\redcircle{(0.6,1.7)} node[anchor=east,color=black] {$r$};
		\bluedot{(0.6,1.4)} node[anchor=west,color=black] {$\chk{b}$};
		\redcircle{(0.6,0.2)} node[anchor=west,color=black] {$s$};
		\bluedot{(0,0.2)} node[anchor=east,color=black] {$a$};
		\draw[->] (1,1.1) arc(90:-270:0.3);
		\bluedot{(1.3,0.8)} node[anchor=west,color=black] {$\chk{a}b$};
		\redcircle{(0.7,0.8)} node[anchor=east,color=black] {\dotlabel{-r-s-2}};
		\end{tikzpicture}
		\left( =
		\begin{tikzpicture}[anchorbase]
		\draw[<-] (0,0) -- (0,1);
		\draw[->] (0.5,0) -- (0.5,1);
		\end{tikzpicture}
		\ - \delta_{k,-1} \sum_{b \in B}
		\begin{tikzpicture}[anchorbase]
		\draw[->] (0,1.1) -- (0,0.9) arc(180:360:0.3) -- (0.6,1.1);
		\draw[<-] (0,-0.1) -- (0,0.1) arc(180:0:0.3) -- (0.6,-0.1);
		\bluedot{(0,0.1)} node[anchor=east,color=black] {$b$};
		\bluedot{(0.6,0.9)} node[anchor=east,color=black] {$\chk{b}$};
		\end{tikzpicture}
		\text{ if } k \ge -1 \right),
		\end{equation}
		
		\noindent\begin{minipage}{0.5\linewidth}
			\begin{equation} \label{theo-eq:right-curl1}
			\begin{tikzpicture}[anchorbase]
			\draw[->] (0,-0.75) .. controls (0,0.5) and (0.5,0.5) .. (0.5,0) .. controls (0.5,-0.5) and (0,-0.5) .. (0,0.75);
			\end{tikzpicture}
			= \delta_{k,0}\
			\begin{tikzpicture}[anchorbase]
			\draw[->] (0,-0.75) -- (0,0.75);
			\end{tikzpicture}
			\quad \text{if } k \ge 0,
			\end{equation}
		\end{minipage}%
		\begin{minipage}{0.5\linewidth}
			\begin{equation} \label{theo-eq:clockwise-circ1}
			\begin{tikzpicture}[anchorbase]
			\draw[<-] (0,0.3) arc(90:450:0.3);
			\redcircle{(-0.3,0)} node[anchor=east,color=black] {$r$};
			\bluedot{(0.3,0)} node[anchor=west,color=black] {$f$};
			\end{tikzpicture}
			\ = -\delta_{r,k-1} \tr(f)
			\quad \text{if } 0 \le r < k,
			\end{equation}
		\end{minipage}\par\vspace{\belowdisplayskip}
		
		\noindent\begin{minipage}{0.5\linewidth}
			\begin{equation} \label{theo-eq:left-curl1}
			\begin{tikzpicture}[anchorbase]
			\draw[->] (0,-0.75) .. controls (0,0.5) and (-0.5,0.5) .. (-0.5,0) .. controls (-0.5,-0.5) and (0,-0.5) .. (0,0.75);
			\end{tikzpicture}
			= \delta_{k,0}\
			\begin{tikzpicture}[anchorbase]
			\draw[->] (0,-0.75) -- (0,0.75);
			\end{tikzpicture}
			\quad \text{if } k \le 0,
			\end{equation}
		\end{minipage}%
		\begin{minipage}{0.5\linewidth}
			\begin{equation} \label{theo-eq:counterclockwise-circ1}
			\begin{tikzpicture}[anchorbase]
			\draw[->] (0,0.3) arc(90:450:0.3);
			\redcircle{(0.3,0)} node[anchor=west,color=black] {$r$};
			\bluedot{(-0.3,0)} node[anchor=east,color=black] {$f$};
			\end{tikzpicture}
			\ = \delta_{r,-k-1} \tr(f)
			\quad \text{if } 0 \le r < -k.
			\end{equation}
		\end{minipage}\par\vspace{\belowdisplayskip}
		
		Moreover, $\Heis_{F,k}$ can be presented equivalently as the strict $\kk$-linear monoidal category generated by the objects $\sQ_+$, $\sQ_-$, and morphisms $s,x,c,d,c',d'$, and $\beta_f$, $f \in F$, subject only to the relations \cref{rel:token-homom1,rel:braid-up1,rel:doublecross-up1,rel:dot-token-up-slide1,rel:tokenslide-up-right1,rel:dotslide1.1,rel:right-adjunction-up1,rel:right-adjunction-down1,theo-eq:doublecross-up-down1,theo-eq:doublecross-down-up1,theo-eq:right-curl1,theo-eq:clockwise-circ1,theo-eq:left-curl1,theo-eq:counterclockwise-circ1}.
		In the above relations, in addition to the rightward crossing $t$ defined by \cref{eq:t-def1}, we have used the left crossing $t' \colon \sQ_- \sQ_+ \to \sQ_+ \sQ_-$ defined by
		\begin{equation} \label{eq:t-def-alt1}
		t' =
		\begin{tikzpicture}[anchorbase]
		\draw [<-](0,0) -- (0.6,0.6);
		\draw [->](0.6,0) -- (0,0.6);
		\end{tikzpicture}
		\ :=\
		\begin{tikzpicture}[anchorbase,scale=0.6]
		\draw[<-] (0.3,0) -- (-0.3,-1);
		\draw[<-] (-0.75,-1) -- (-0.75,-0.5) .. controls (-0.75,-0.2) and (-0.5,0) .. (0,-0.5) .. controls (0.5,-1) and (0.75,-0.8) .. (0.75,-0.5) -- (0.75,0);
		\end{tikzpicture}
		\ ,
		\end{equation}
		and the negatively dotted bubbles defined, for $f \in F$, by
		\begin{gather} \label{theo-eq:neg-ccbubble1}
		\ccbubble{$f$}{}{\dotlabel{r-k-1}}
		= \sum_{b_1,\dotsc,b_{r-1} \in B} \det
		\left( \cbubble{$\chk{b}_{j-1}b_j$}{}{\dotlabel{i-j+k}} \right)_{i,j=1}^r,
		\quad \text{if } r \le k,
		\\ \label{theo-eq:neg-cbubble1}
		\cbubble{$f$}{}{\dotlabel{r+k-1}}
		= (-1)^{r+1} \sum_{b_1,\dotsc,b_{r-1} \in B} \det
		\left( \ccbubble{$\chk{b}_{j-1}b_j$}{}{\dotlabel{i-j-k}} \right)_{i,j=1}^r,
		\quad \text{if } r \le -k,
		\end{gather}
		where we adopt the convention that $\chk{b}_0 = f$ and $b_r = 1$, and we interpret the determinants as $\tr(f)$ if $r=0$ and as $0$ if $r < 0$. 
	\end{theo}
	The morphisms $c'$ and $d'$ of \cref{theo:alternate-presentation1} are defined as (see \cite[page 12]{S17}):
	
	\noindent\begin{minipage}{0.5\linewidth}
		\begin{equation} \label{eq:left-cup-def1}
		c' =
		\begin{tikzpicture}[anchorbase]
		\draw[<-] (0,0.2) -- (0,0) arc (180:360:.3) -- (0.6,0.2);
		\end{tikzpicture}
		:=
		\begin{cases}
		-\
		\begin{tikzpicture}[>=stealth,baseline={([yshift=1ex]current bounding box.center)}]
		\draw[<-] (0,0.2) -- (0,0) arc (180:360:.3) -- (0.6,0.2);
		\greensquare{(0.3,-0.3)} node[anchor=north,color=black] {\squarelabel{(k-1,1)}};
		\end{tikzpicture}
		& \text{if } k > 0, \\
		\begin{tikzpicture}[anchorbase]
		\draw[<-] (0,0) .. controls (0.3,-0.3) and (0.6,-0.3) .. (0.6,-0.6) arc(360:180:0.3) .. controls (0,-0.3) and (0.3,-0.3) .. (0.6,0);
		\redcircle{(0.6,-0.6)} node[anchor=west,color=black] {$-k$};
		\end{tikzpicture}
		& \text{if } k \le 0,
		\end{cases}
		\end{equation}
	\end{minipage}%
	\begin{minipage}{0.5\linewidth}
		\begin{equation} \label{eq:left-cap-def1}
		d' =
		\begin{tikzpicture}[anchorbase]
		\draw[<-] (0,-0.2) -- (0,0) arc (180:0:.3) -- (0.6,-0.2);
		\end{tikzpicture}
		:=
		\begin{cases}
		\begin{tikzpicture}[anchorbase]
		\draw[<-] (0,0) .. controls (0.3,0.3) and (0.6,0.3) .. (0.6,0.6) arc(0:180:0.3) .. controls (0,0.3) and (0.3,0.3) .. (0.6,0);
		\redcircle{(0.6,0.6)} node[anchor=west,color=black] {$k$};
		\end{tikzpicture}
		& \text{if } k \ge 0, \\
		\begin{tikzpicture}[>=stealth,baseline={([yshift=-2ex]current bounding box.center)}]
		\draw[<-] (0,-0.2) -- (0,0) arc (180:0:.3) -- (0.6,-0.2);
		\greensquare{(0.3,0.3)} node[anchor=south,color=black] {\squarelabel{(-k-1,1)}};
		\end{tikzpicture}
		& \text{if } k < 0.
		\end{cases}
		\end{equation}
	\end{minipage}\par\vspace{\belowdisplayskip}
	
The following theorem, \cref{theo:alternate-presentation2}, gives another alternate presentation of $\Heis_{F,k}$.
	
	\begin{theo} \label{theo:alternate-presentation2}
		There are unique morphisms 
		$c_{i}'
		\colon \one \to \sQ_{+} \sQ_{-}$ and 
		$d_{i}' 
		\colon \sQ_{-} \sQ_{+} \to \one,$ $i=1,\dotsc,n,$ in $\Heis_{F,k}$, drawn as
		\[
		c_i' =
		\begin{tikzpicture}[anchorbase]
		\draw[<-] (0,.2) -- (0,0) arc (180:360:.3) -- (.6,.2) node[anchor=south] {$i$};
		\end{tikzpicture} 
		\ ,\qquad
		d_i' =
		\begin{tikzpicture}[anchorbase]
		\draw[<-] (0,-.2) -- (0,0) arc (180:0:.3) -- (.6,-.2) node[anchor=north] {$i$};
		\end{tikzpicture}
		\ ,
		\]
		such that the following relations hold for all $i,j=1,\dotsc,n$ and $f_i\in F_i$:
		\begin{equation} \label{theo-eq:doublecross-up-down2}
		\begin{tikzpicture}[anchorbase]
		\draw[->] (0,0) node[anchor=north] {$i$}.. controls (0.5,0.5) .. (0,1);
		\draw[<-] (0.5,0) .. controls (0,0.5) .. (0.5,1) node[anchor=south] {$j$};
		\end{tikzpicture}
		\ =\
		\begin{tikzpicture}[anchorbase]
		\draw[->] (0,0) node[anchor=north] {$i$} -- (0,1);
		\draw[<-] (0.5,0) -- (0.5,1) node[anchor=south] {$j$};
		\end{tikzpicture}
		\ + \delta_{i,j} \sum_{r,s \ge 0} \sum_{a,b \in B_i}
		\begin{tikzpicture}[anchorbase]
		\draw[->] (0,0) node[anchor=north] {$i$} -- (0,0.7) arc (180:0:0.3) -- (0.6,0);
		\draw[<-] (0,2.1) -- (0,1.8) arc (180:360:0.3) -- (0.6,2.1)  node[anchor=south] {$i$};
		\redcircle{(0,0.6)} node[anchor=west,color=black] {$r$};
		\bluedot{(0,0.2)} node[anchor=west,color=black] {$\chk{b}$};
		\redcircle{(0.6,1.8)} node[anchor=west,color=black] {$s$};
		\bluedot{(0,1.8)} node[anchor=east,color=black] {$a$};
		\draw[->] (1,1.55) arc(90:450:0.3);
		\draw{(1,0.95)} node[anchor=north] {$i$};
		\bluedot{(1.3,1.25)} node[anchor=west,color=black] {$\chk{a} b$};
		\redcircle{(0.7,1.25)} node[anchor=east,color=black] {\dotlabel{-r-s-2}};
		\end{tikzpicture}
		\left( =
		\begin{tikzpicture}[anchorbase]
		\draw[->] (0,0)  node[anchor=north] {$i$} -- (0,1);
		\draw[<-] (0.5,0) -- (0.5,1)  node[anchor=south] {$j$} ;
		\end{tikzpicture}
		\ + \delta_{i,j}\delta_{k,1} \sum_{b \in B_i}
		\begin{tikzpicture}[anchorbase]
		\draw[<-] (0,1.1) -- (0,0.9) arc(180:360:0.3) -- (0.6,1.1)  node[anchor=south] {$i$} ;
		\draw[->] (0,-0.1)  node[anchor=north] {$i$}  -- (0,0.1) arc(180:0:0.3) -- (0.6,-0.1);
		\bluedot{(0,0.1)} node[anchor=west,color=black] {$\chk{b}$};
		\bluedot{(0,0.9)} node[anchor=west,color=black] {$b$};
		\end{tikzpicture}
		\text{ if } k \le 1 \right),
		\end{equation}
		\begin{equation} \label{theo-eq:doublecross-down-up2}
		\begin{tikzpicture}[anchorbase]
		\draw[<-] (0,0) .. controls (0.5,0.5) .. (0,1) node[anchor=south] {$i$} ;
		\draw[->] (0.5,0) node[anchor=north] {$j$}  .. controls (0,0.5) .. (0.5,1);
		\end{tikzpicture}
		\ =\
		\begin{tikzpicture}[anchorbase]
		\draw[<-] (0,0) -- (0,1) node[anchor=south] {$i$};
		\draw[->] (0.5,0)  node[anchor=north] {$j$} -- (0.5,1);
		\end{tikzpicture}
		\ + \delta_{i,j}\sum_{r,s \ge 0} \sum_{a,b \in B_i}
		\begin{tikzpicture}[anchorbase]
		\draw[->] (0,2) node[anchor=south] {$i$}  -- (0,1.3) arc (180:360:0.3) -- (0.6,2);
		\draw[<-] (0,-0.1) -- (0,0.2) arc (180:0:0.3) -- (0.6,-0.1) node[anchor=north] {$i$} ;
		\redcircle{(0.6,1.7)} node[anchor=east,color=black] {$r$};
		\bluedot{(0.6,1.4)} node[anchor=west,color=black] {$\chk{b}$};
		\redcircle{(0.6,0.2)} node[anchor=east,color=black] {$s$};
		\bluedot{(0,0.2)} node[anchor=east,color=black] {$a$};
		\draw[->] (1,1.1) arc(90:-270:0.3);
		\draw{(1,0.5)} node[anchor=north] {$i$};
		\bluedot{(1.3,0.8)} node[anchor=west,color=black] {$\chk{a}b$};
		\redcircle{(0.7,0.8)} node[anchor=east,color=black] {\dotlabel{-r-s-2}};
		\end{tikzpicture}
		\left( =
		\begin{tikzpicture}[anchorbase]
		\draw[<-] (0,0) -- (0,1) node[anchor=south] {$i$} ;
		\draw[->] (0.5,0) node[anchor=north] {$j$}  -- (0.5,1);
		\end{tikzpicture}
		\ - \delta_{i,j}\delta_{k,-1} \sum_{b \in B_i}
		\begin{tikzpicture}[anchorbase]
		\draw[->] (0,1.1) node[anchor=south] {$i$}  -- (0,0.9) arc(180:360:0.3) -- (0.6,1.1);
		\draw[<-] (0,-0.1) -- (0,0.1) arc(180:0:0.3) -- (0.6,-0.1) node[anchor=north] {$i$} ;
		\bluedot{(0,0.1)} node[anchor=east,color=black] {$b$};
		\bluedot{(0.6,0.9)} node[anchor=east,color=black] {$\chk{b}$};
		\end{tikzpicture}
		\text{ if } k \ge -1 \right),
		\end{equation}
		\noindent\begin{minipage}{0.5\linewidth}
			\begin{equation} \label{theo-eq:right-curl2}
			\begin{tikzpicture}[anchorbase]
			\draw[->] (0,-0.75) node[anchor=north] {$i$} .. controls (0,0.5) and (0.5,0.5) .. (0.5,0) .. controls (0.5,-0.5) and (0,-0.5) .. (0,0.75);
			\end{tikzpicture}
			= \delta_{k,0}\
			\begin{tikzpicture}[anchorbase]
			\draw[->] (0,-0.75) node[anchor=north] {$i$} -- (0,0.75);			
			\end{tikzpicture}
			\quad \text{if } k \ge 0,
			\end{equation}
		\end{minipage}%
		\begin{minipage}{0.5\linewidth}
			\begin{equation} \label{theo-eq:clockwise-circ2}
			\begin{tikzpicture}[anchorbase]
			\draw[<-] (0,0.3) arc(90:450:0.3);
			\redcircle{(-0.3,0)} node[anchor=east,color=black] {$r$};
			\bluedot {(0.3,0)} node[anchor=west,color=black] {$f_i$};
			\draw {(0,-0.3)} node[anchor=north] {$i$};
			\end{tikzpicture}
			\ = -\delta_{r,k-1} \tr_i(f_i)
			\quad \text{if } 0 \le r < k,
			\end{equation}
		\end{minipage}\par\vspace{\belowdisplayskip}
		\noindent\begin{minipage}{0.5\linewidth}
			\begin{equation} \label{theo-eq:left-curl2}
			\begin{tikzpicture}[anchorbase]
			\draw[->] (0,-0.75) node[anchor=north] {$i$} .. controls (0,0.5) and (-0.5,0.5) .. (-0.5,0) .. controls (-0.5,-0.5) and (0,-0.5) .. (0,0.75);
			\end{tikzpicture}
			= \delta_{k,0}\
			\begin{tikzpicture}[anchorbase]
			\draw[->] (0,-0.75) node[anchor=north] {$i$} -- (0,0.75);
			\end{tikzpicture}
			\quad \text{if } k \le 0,
			\end{equation}
		\end{minipage}%
		\begin{minipage}{0.5\linewidth}
			\begin{equation} \label{theo-eq:counterclockwise-circ2}
			\begin{tikzpicture}[anchorbase]
			\draw[->] (0,0.3) arc(90:450:0.3);
			\redcircle{(0.3,0)} node[anchor=west,color=black] {$r$};
			\bluedot{(-0.3,0)} node[anchor=east, color=black] {$f_i$};
			\draw {(0,-0.3)} node[anchor=north] {$i$};
			\end{tikzpicture}
			\ = \delta_{r,-k-1} \tr_i(f_i)
			\quad \text{if } 0 \le r < -k.
			\end{equation}
		\end{minipage}\par\vspace{\belowdisplayskip}
		
		Moreover, $\Heis_{F,k}$ can be presented equivalently as the strict $\kk$-linear monoidal category generated by the objects $\sQ_{+}$, $\sQ_{-}$ and morphisms $s_{i,j},x_i,c_i,d_i,c_i',d_i'$, $i,j=1,\dotsc,n$, subject only to the relations \cref{rel:token-homom2,rel:braid-up2,rel:doublecross-up2,rel:dot-token-up-slide2,rel:tokenslide-up-right2,rel:dotslide1.2,rel:right-adjunction-up2,rel:right-adjunction-down2,theo-eq:doublecross-up-down2,theo-eq:doublecross-down-up2,theo-eq:right-curl2,theo-eq:clockwise-circ2,theo-eq:left-curl2,theo-eq:counterclockwise-circ2}.
		In the above relations, in addition to the rightward crossing defined by \cref{eq:t-def2}, we have used the left crossing, for $i,j=1,\dotsc,n,$ defined by
		\begin{equation} \label{eq:t-def-alt2}
		t_{i,j}'=
		\begin{tikzpicture}[anchorbase]
		\draw [<-](0,0) -- (0.6,0.6) node[anchor=south] {$i$};
		\draw [->](0.6,0) node[anchor=north] {$j$} -- (0,0.6);
		\end{tikzpicture}
		\ :=\
		\begin{tikzpicture}[anchorbase,scale=0.6]
		\draw[<-] (0.3,0) -- (-0.3,-1) node[anchor=north] {$j$};
		\draw[<-] (-0.75,-1) -- (-0.75,-0.5) .. controls (-0.75,-0.2) and (-0.5,0) .. (0,-0.5) .. controls (0.5,-1) and (0.75,-0.8) .. (0.75,-0.5) -- (0.75,0) node[anchor=south] {$i$};
		\end{tikzpicture},
		\, \quad Q_{-}Q_{+} \mapsto Q_{+}Q_{-}
		\ ,
		\end{equation}
		and the negatively dotted bubbles defined, for $i=1,\dotsc,n$, by
		\begin{gather} \label{theo-eq:neg-ccbubble2}
		\ccbubble{$f$}{$l$}{\dotlabel{r-k-1}}
		= \sum_{b_1,\dotsc,b_{r-1} \in B_l}\det
		\left( \cbubble{$\chk{b_{j-1}}b_j$}{\small $l$}{\dotlabel{i-j+k}} \right)_{i,j=1}^r,
		\quad \text{if } r \le k,
		\\ \label{theo-eq:neg-cbubble2}
		\cbubble{$f$}{$l$}{\dotlabel{r+k-1}}
		= (-1)^{r+1} \sum_{b_1,\dotsc,b_{r-1} \in B_l} \det
		\left(\ccbubble{$\chk{b_{j-1}}b_j$}{\small $l$}{\dotlabel{i-j-k}} \right)_{i,j=1}^r,
		\quad \text{if } r \le -k,
		\end{gather}
		where we adopt the convention that $\chk{b}_0 = f$ and $b_r = 1$, and we interpret the determinants as $\tr_i(f)$ if $r=0$ and as $0$ if $r < 0$. 
	\end{theo}
	\begin{proof}
		
		Let $\Heis_{F,k}''$ be the strict $\kk-$linear monoidal category generated by the objects $\sQ_{+}$, $\sQ_{-}$ and morphisms $s_{i,j},x_i,c_i,d_i,c_i',d_i'$, $i,j=1,\dotsc,n$, subject only to the relations \cref{rel:token-homom2,rel:braid-up2,rel:doublecross-up2,rel:dot-token-up-slide2,rel:tokenslide-up-right2,rel:dotslide1.2,rel:right-adjunction-up2,rel:right-adjunction-down2,theo-eq:doublecross-up-down2,theo-eq:doublecross-down-up2,theo-eq:right-curl2,theo-eq:clockwise-circ2,theo-eq:left-curl2,theo-eq:counterclockwise-circ2}. The rightward crossing, $t_{i,j}$, and leftward crossing, $t_{i,j}'$, are given by \cref{eq:t-def2} and \cref{eq:t-def-alt2}, respectively, and the negatively dotted bubbles are defined by \cref{theo-eq:neg-ccbubble2} and \cref{theo-eq:neg-cbubble2}.
		
		View $\Heis_{F,k}$ in the presentation given in \cref{theo:alternate-presentation1}. That is, $\Heis_{F,k}$ is generated by the objects $\sQ_{+}$, $\sQ_{-}$ and morphisms $s,x,c,d,c',d'$, subject only to the relations \cref{rel:token-homom1,rel:braid-up1,rel:doublecross-up1,rel:dot-token-up-slide1,rel:tokenslide-up-right1,rel:dotslide1.1,rel:right-adjunction-up1,rel:right-adjunction-down1,theo-eq:doublecross-up-down1,theo-eq:doublecross-down-up1,theo-eq:right-curl1,theo-eq:clockwise-circ1,theo-eq:left-curl1,theo-eq:counterclockwise-circ1}, where the rightward crossing, $t$, and leftward crossing, $t'$, are given by $\cref{eq:t-def1}$ and $\cref{eq:t-def-alt1}$, respectively, and the negatively dotted bubbles are defined by $\cref{theo-eq:neg-ccbubble1}$ and $\cref{theo-eq:neg-cbubble1}$. 
		
		We first prove that $\Heis_{F,k}''$ is isomorphic to $\Heis_{F,k}$.
		
		We define a monoidal functor $A:\Heis_{F,k}''\to\Heis_{F,k}$ as follows. On objects define 
		\[
		A\left(\sQ_\pm\right)= \sQ_\pm,
		\]
		and on morphisms, for $f_i\in F_i$, $i,j=1,\dotsc,n$, define
		\[
		A\left(s_{i,j}\right)=s_{i,j}, \quad A\left(c_i\right)=c_i, \quad A\left(d_i\right)=d_i, \quad A\left(x_i\right)=x_i, \quad A\left(\beta_{(f_i,i)}\right)=\beta_{f_i}, \]
		\[ A(c_i')=
		c_i', \quad  A\left(d_i'\right)=d_i',\]
		and extend $A$ so that it is linear and monoidal.
		We define another monoidal functor $B:\Heis_{F,k}\to\Heis_{F,k}''$ as follows. On objects define 
		\[
		B\left(\sQ_\pm\right)= \sQ_\pm,
		\]
		and on morphisms, define
		\[
		B\left(s\right)=\sum_{i,j=1}^ns_{i,j}, \quad B\left(c\right)=\sum_{i=1}^nc_i,\quad B\left(d\right)=\sum_{i=1}^nd_i, \quad B\left(\beta_{(f_1,\dotsc,f_n)}\right)=\sum_{i=1}^n\beta_{(f_i,i)}, 
		\]
		\[
		B\left(c'\right)=\sum_{i=1}^n
		c_i', \quad  B\left(d'\right)=\sum_{i=1}^nd_i',
		\]
		and extend $B$ so that it is linear and monoidal.

		We claim that $A$ and $B$ are strict $\kk$-linear monoidal functors that are two-sided inverses of each other. If $A$ and $B$ exist (i.e. if they are well-defined), it is clear that $B\circ A=id_{\Heis_{F,k}''}$ and $A\circ B=id_{\Heis_{F,k}}$.	
		So it is enough to show that $A$ and $B$ are well-defined. To do this, we show that $A$ preserves the defining relations of $\Heis_{F,k}''$ and $B$ preserves the defining relations of $\Heis_{F,k}$.
		That is, we show that $B$ preserves the relations
		\cref{rel:token-homom1,rel:braid-up1,rel:doublecross-up1,rel:dot-token-up-slide1,rel:tokenslide-up-right1,rel:dotslide1.1,rel:right-adjunction-up1,rel:right-adjunction-down1,theo-eq:doublecross-up-down1,theo-eq:doublecross-down-up1,theo-eq:right-curl1,theo-eq:clockwise-circ1,theo-eq:left-curl1,theo-eq:counterclockwise-circ1} in $\Heis_{F,k}$, and that $A$ preserves the relations  \cref{rel:token-homom2,rel:braid-up2,rel:doublecross-up2,rel:dot-token-up-slide2,rel:tokenslide-up-right2,rel:dotslide1.2,rel:right-adjunction-up2,rel:right-adjunction-down2,theo-eq:doublecross-up-down2,theo-eq:doublecross-down-up2,theo-eq:right-curl2,theo-eq:clockwise-circ2,theo-eq:left-curl2,theo-eq:counterclockwise-circ2} in $\Heis_{F,k}''$.
		We first show that $A$ preserves the defining relations of $\Heis_{F,k}''.$

		The functor $A$ preserves the relations  \cref{rel:token-homom2,rel:braid-up2,rel:doublecross-up2,rel:dot-token-up-slide2,rel:tokenslide-up-right2,rel:dotslide1.2,rel:right-adjunction-up2,rel:right-adjunction-down2} in $\Heis_{F,k}''.$ This follows by the same arguments used in \cref{theo:isomorphic}, where we showed that the functor $F$ preserves these relations in $\Heis_{F,k}'$.
		
		The functor $A$ preserves the relations  \cref{theo-eq:doublecross-down-up2,theo-eq:doublecross-up-down2,theo-eq:right-curl2,theo-eq:left-curl2} in $\Heis_{F,k}''$. This follows directly from \cref{theo:alternate-presentation1} by attaching the appropriate idempotents above and below both sides of the equations \cref{theo-eq:doublecross-up-down1,theo-eq:doublecross-down-up1,theo-eq:right-curl1,theo-eq:left-curl1}, and noting that the resulting equations are the images of \cref{theo-eq:doublecross-down-up2,theo-eq:doublecross-up-down2,theo-eq:right-curl2,theo-eq:left-curl2} under $A$, respectively.
		
		The functor $A$ preserves the relations \cref{theo-eq:clockwise-circ2,theo-eq:counterclockwise-circ2} in $\Heis_{F,k}''$. This follows by letting $f\in F_i$, $i=1,\dotsc,n,$ in equations \cref{theo-eq:clockwise-circ1,theo-eq:counterclockwise-circ1}, and noting that the resulting equations are the images of \cref{theo-eq:clockwise-circ2,theo-eq:counterclockwise-circ2} under $A$, respectively. Therefore, we have shown that $A$ preserves the defining relations of $\Heis_{F,k}''$. Now we show that the $B$ preserves the defining relations of $\Heis_{F,k}$.
		
		The functor $B$ preserves the relations  \cref{rel:token-homom1,rel:braid-up1,rel:doublecross-up1,rel:dot-token-up-slide1,rel:tokenslide-up-right1,rel:dotslide1.1,rel:right-adjunction-up1,rel:right-adjunction-down1} in $\Heis_{F,k}.$ This follows by the same arguments used in \cref{theo:isomorphic}, where we showed that the functor $G$ preserves these relations in $\Heis_{F,k}$.
		
		The functor $B$ preserves the relations  \cref{theo-eq:doublecross-down-up1,theo-eq:doublecross-up-down1} in $\Heis_{F,k}$. This follows by summing both sides of the relations \cref{theo-eq:doublecross-up-down2,theo-eq:doublecross-down-up2} over all $i,j=1,\dotsc,n$, and noting that the resulting equations are the images of \cref{theo-eq:doublecross-down-up1,theo-eq:doublecross-up-down1} under $B$, respectively.
		
		The functor $B$ preserves the relations \cref{theo-eq:right-curl1,theo-eq:left-curl1} in $\Heis_{F,k}$.  This follows by summing both sides of the relations \cref{theo-eq:right-curl2,theo-eq:left-curl2} over all $i=1,\dotsc,n,$ and noting that the resulting equations are the images of \cref{theo-eq:right-curl1,theo-eq:left-curl1} under $B$, respectively.
		
		The functor $B$ preserves the relations \cref{theo-eq:clockwise-circ1,theo-eq:counterclockwise-circ1} in $\Heis_{F,k}$. Let $f=(f_1,\dotsc,f_n),$ and note that the relations \cref{theo-eq:clockwise-circ1,theo-eq:counterclockwise-circ1} can be written as
		\[
		\sum\limits_{i=1}^{n}\begin{tikzpicture}[anchorbase]
		\draw[<-] (0,0.3) arc(90:450:0.3);
		\redcircle{(-0.3,0)} node[anchor=east,color=black] {$r$};
		\bluedot{(0.3,0)} node[anchor=west,color=black] {$f_i$};
		\draw{(0,-0.3)} node[anchor=north] {$i$};
		\end{tikzpicture}
		\ = -\delta_{r,k-1} \sum\limits_{i=1}^{n} \tr_i(f_i)
		\quad \text{if } 0 \le r < k, \quad \text{and}
		\]
		\[
		\sum\limits_{i=1}^{n}\begin{tikzpicture}[anchorbase]
		\draw[->] (0,0.3) arc(90:450:0.3);
		\redcircle{(0.3,0)} node[anchor=west,color=black] {$r$};
		\bluedot{(-0.3,0)} node[anchor=east, color=black] {$f_i$};
		\draw {(0,-0.3)} node[anchor=north] {$i$};
		\end{tikzpicture}
		\ = \delta_{r,-k-1} \sum\limits_{i=1}^{n}\tr_i(f_i)
		\quad \text{if } 0 \le r < -k, f_i\in F_i,
		\]
		respectively. Since the relations \cref{theo-eq:clockwise-circ2,theo-eq:counterclockwise-circ2} hold in $\Heis_{F,k}'',$ $B$ preserves the relations \cref{theo-eq:clockwise-circ1,theo-eq:counterclockwise-circ1} in $\Heis_{F,k}.$ Therefore, we have shown that $B$ preserves the defining relations of $\Heis_{F,k}$. Hence, $\Heis_{F,k}''$ is isomorphic to $\Heis_{F,k}$.
		
		Now we show that there are unique morphisms, $c_i'$ and $d_i'$, $i=1,\dotsc,n$, in $\Heis_{F,k}$ satisfying \cref{theo-eq:doublecross-up-down2,theo-eq:doublecross-down-up2,theo-eq:right-curl2,theo-eq:clockwise-circ2,theo-eq:left-curl2,theo-eq:counterclockwise-circ2}. Define 
		
		\noindent\begin{minipage}{0.5\linewidth}
			\begin{equation} \label{eq:left-cup-def2}
			c_i' :=
			\begin{tikzpicture}[anchorbase]
			\draw[<-] (0,0.2) -- (0,0) arc (180:360:.3) -- (0.6,0.2) node[anchor=south] {$i$};
			\end{tikzpicture}
			=
			\begin{cases}
			-\
			\begin{tikzpicture}[>=stealth,baseline={([yshift=1ex]current bounding box.center)}]
			\draw[<-] (0,0.2) -- (0,0) arc (180:360:.3) -- (0.6,0.2) node[anchor=south] {$i$};
			\greensquare{(0.3,-0.3)} node[anchor=north,color=black] {\squarelabel{(k-1,1)}};
			\end{tikzpicture}
			& \text{if } k > 0, \\
			\begin{tikzpicture}[anchorbase]
			\draw[<-] (0,0) .. controls (0.3,-0.3) and (0.6,-0.3) .. (0.6,-0.6) arc(360:180:0.3) .. controls (0,-0.3) and (0.3,-0.3) .. (0.6,0) node[anchor=south] {$i$};
			\redcircle{(0.6,-0.6)} node[anchor=west,color=black] {$-k$};
			\end{tikzpicture}
			& \text{if } k \le 0,
			\end{cases}
			\end{equation}
		\end{minipage}%
		\begin{minipage}{0.5\linewidth}
			\begin{equation} \label{eq:left-cap-def2}
			d_i' :=
			\begin{tikzpicture}[anchorbase]
			\draw[<-] (0,-0.2) -- (0,0) arc (180:0:.3) -- (0.6,-0.2) node[anchor=north] {$i$};
			\end{tikzpicture}
			=
			\begin{cases}
			\begin{tikzpicture}[anchorbase]
			\draw[<-] (0,0) .. controls (0.3,0.3) and (0.6,0.3) .. (0.6,0.6) arc(0:180:0.3) .. controls (0,0.3) and (0.3,0.3) .. (0.6,0) node[anchor=north] {$i$};
			\redcircle{(0.6,0.6)} node[anchor=west,color=black] {$k$};
			\end{tikzpicture}
			& \text{if } k \ge 0, \\
			\begin{tikzpicture}[>=stealth,baseline={([yshift=-2ex]current bounding box.center)}]
			\draw[<-] (0,-0.2) -- (0,0) arc (180:0:.3) -- (0.6,-0.2) node[anchor=north] {$i$};
			\greensquare{(0.3,0.3)} node[anchor=south,color=black] {\squarelabel{(-k-1,1)}};
			\end{tikzpicture}
			& \text{if } k < 0.
			\end{cases}
			\end{equation}
		\end{minipage}\par\vspace{\belowdisplayskip}
		
		First we show that the morphisms $c_i'$ and $d_i'$, $i=1,\dotsc,n$, satisfy the relations \cref{theo-eq:doublecross-up-down2,theo-eq:doublecross-down-up2,theo-eq:right-curl2,theo-eq:clockwise-circ2,theo-eq:left-curl2,theo-eq:counterclockwise-circ2} in $\Heis_{F,k}$. Since \cref{theo-eq:doublecross-up-down1,theo-eq:doublecross-down-up1,theo-eq:right-curl1,theo-eq:left-curl1} hold in $\Heis_{F,k}$, attaching the appropriate idempotents above or below both sides of these relations gives \cref{theo-eq:doublecross-up-down2,theo-eq:doublecross-down-up2,theo-eq:right-curl2,theo-eq:left-curl2}, respectively, in $\Heis_{F,k}$. Since \cref{theo-eq:clockwise-circ1,theo-eq:counterclockwise-circ1} hold in $\Heis_{F,k}$, we can let the label of the tokens of these relations be $f\in F_i$. This gives the relations \cref{theo-eq:clockwise-circ2,theo-eq:counterclockwise-circ2}, respectively, in $\Heis_{F,k}$. So we have shown that $c_i'$ and $d_i'$ are morphisms in $\Heis_{F,k}$ that satisfy the relations \cref{theo-eq:doublecross-up-down2,theo-eq:doublecross-down-up2,theo-eq:right-curl2,theo-eq:clockwise-circ2,theo-eq:left-curl2,theo-eq:counterclockwise-circ2}. 
		
		Now we show that $c_i'$ and $d_i'$, $i=1,\dotsc,n,$ are $\textit{unique}$ morphisms in $\Heis_{F,k}$ that satisfy \cref{theo-eq:doublecross-up-down2,theo-eq:doublecross-down-up2,theo-eq:right-curl2,theo-eq:clockwise-circ2,theo-eq:left-curl2,theo-eq:counterclockwise-circ2}. Summing both sides of \cref{theo-eq:doublecross-up-down2,theo-eq:doublecross-down-up2} over all $i,j=1,\dotsc,n$ gives the relations \cref{theo-eq:doublecross-up-down1,theo-eq:doublecross-down-up1}, respectively. Summing both sides of \cref{theo-eq:right-curl2,theo-eq:left-curl2} over all $i=1,\dotsc,n$, gives the relations \cref{theo-eq:right-curl1,theo-eq:left-curl1}, respectively. Letting $f=(f_1,\dotsc,f_n)\in F$, and noting that the relations \cref{theo-eq:clockwise-circ1,theo-eq:counterclockwise-circ1} in $\Heis_{F,k}$ are equivalent to the equations
		\[
		\sum\limits_{i=1}^{n}\begin{tikzpicture}[anchorbase]
		\draw[<-] (0,0.3) arc(90:450:0.3);
		\redcircle{(-0.3,0)} node[anchor=east,color=black] {$r$};
		\bluedot{(0.3,0)} node[anchor=west,color=black] {$f_i$};
		\draw{(0,-0.3)} node[anchor=north] {$i$};
		\end{tikzpicture}
		\ = -\delta_{r,k-1} \sum\limits_{i=1}^{n} \tr_i(f_i)
		\quad \text{if } 0 \le r < k,  \quad f_i\in F_i, \quad \text{and}
		\]
		\[
		\sum\limits_{i=1}^{n}\begin{tikzpicture}[anchorbase]
		\draw[->] (0,0.3) arc(90:450:0.3);
		\redcircle{(0.3,0)} node[anchor=west,color=black] {$r$};
		\bluedot{(-0.3,0)} node[anchor=east, color=black] {$f_i$};
		\draw {(0,-0.3)} node[anchor=north] {$i$};
		\end{tikzpicture}
		\ = \delta_{r,-k-1} \sum\limits_{i=1}^{n}\tr_i(f_i)
		\quad \text{if } 0 \le r < -k, \quad f_i\in F_i,
		\]
		respectively, shows that the relations  \cref{theo-eq:clockwise-circ2,theo-eq:counterclockwise-circ2} in $\Heis_{F,k}$ imply the relations \cref{theo-eq:clockwise-circ1,theo-eq:counterclockwise-circ1} in $\Heis_{F,k}$. So the existence of the morphisms $c_i'$ and $d_i'$, $i=1,\dotsc,n$, in $\Heis_{F,k}$ satisfying \cref{theo-eq:doublecross-up-down2,theo-eq:doublecross-down-up2,theo-eq:right-curl2,theo-eq:clockwise-circ2,theo-eq:left-curl2,theo-eq:counterclockwise-circ2} implies the existence of morphisms $c'$ and $d'$ in $\Heis_{F,k}$ satisfying \cref{theo-eq:doublecross-up-down1,theo-eq:doublecross-down-up1,theo-eq:right-curl1,theo-eq:clockwise-circ1,theo-eq:left-curl1,theo-eq:counterclockwise-circ1}. Attaching the appropriate idempotent to the bottom or top of $c'$ and $d'$ gives the morphisms $c_i'$ and $d_i'$ in $\Heis_{F,k},$ respectively. Therefore, since $c'$ and $d'$ are \textit{unique} morphisms in $\Heis_{F,k}$ satisfying \cref{theo-eq:doublecross-up-down1,theo-eq:doublecross-down-up1,theo-eq:right-curl1,theo-eq:clockwise-circ1,theo-eq:left-curl1,theo-eq:counterclockwise-circ1}, we must have that $c_i'$ and $d_i'$, $i=1,\dotsc,n,$ are \textit{unique} morphisms in $\Heis_{F,k}$ satisfying \cref{theo-eq:doublecross-up-down2,theo-eq:doublecross-down-up2,theo-eq:right-curl2,theo-eq:clockwise-circ2,theo-eq:left-curl2,theo-eq:counterclockwise-circ2}.
	\end{proof}
	
	\begin{theo} \label{theo:relations2}
		Using the notation from \cref{theo:alternate-presentation2}, the following relations are consequences of the defining relations.
		\begin{enumerate}[wide]
			\item \emph{Infinite grassmannian relations}: For $i=1,\dotsc,n$ and $f,g\in F_i$ we have
			
			\noindent\begin{minipage}{0.5\linewidth}
				\begin{equation} \label{eq:inf-grass1.2}
				\begin{tikzpicture}[anchorbase]
				\draw[<-] (0,0.3) arc(90:450:0.3);
				\redcircle{(-0.3,0)} node[anchor=east,color=black] {$r$};
				\bluedot{(0.3,0)} node[anchor=west,color=black] {$f$};
				\draw{(0,-0.3)} node[anchor=north] {$i$};
				\end{tikzpicture}
				\ = -\delta_{r,k-1} \tr_i(f)
				\quad \text{if } r \le k-1,
				\end{equation}
			\end{minipage}%
			\begin{minipage}{0.5\linewidth}
				\begin{equation} \label{eq:inf-grass2.2}
				\begin{tikzpicture}[anchorbase]
				\draw[->] (0,0.3) arc(90:450:0.3);
				\redcircle{(0.3,0)} node[anchor=west,color=black] {$r$};
				\bluedot{(-0.3,0)} node[anchor=east,color=black] {$f$};
				\draw{(0,-0.3)} node[anchor=north] {$i$};
				\end{tikzpicture}
				\ = \delta_{r,-k-1} \tr_i(f)
				\quad \text{if } r \le -k-1,
				\end{equation}
			\end{minipage}\par\vspace{\belowdisplayskip}
			
			\begin{equation} \label{eq:inf-grass3.2}
			\sum_{\substack{r,s \in \Z \\ r + s = t-2}} \sum_{b \in B_i}\
			\begin{tikzpicture}[anchorbase]
			\draw[<-] (0,0.3) arc(90:450:0.3);
			\draw[->] (0,-0.5) arc(90:450:0.3);
			\redcircle{(-0.3,0)} node[anchor=east,color=black] {$r$};
			\bluedot{(0.3,0)} node[anchor=west,color=black] {$fb$};
			\redcircle{(0.3,-0.8)} node[anchor=west,color=black] {$s$};
			\bluedot{(-0.3,-0.8)} node[anchor=east,color=black] {$\chk{b}g$};
			\draw{(0,-1.1)} node[anchor=north,color=black] {$i$};
			\draw{(0,0.3)} node[anchor=south,color=black] {$i$};
			\end{tikzpicture}
			\ =    \sum_{\substack{r,s \ge 0 \\ r + s = t}} \sum_{b \in B_i }\
			\begin{tikzpicture}[anchorbase]
			\draw[<-] (0,0.3) arc(90:450:0.3);
			\draw[->] (0,-0.5) arc(90:450:0.3);
			\redcircle{(-0.3,0)} node[anchor=east,color=black] {\dotlabel{r+k-1}};
			\bluedot{(0.3,0)} node[anchor=west,color=black] {$fb$};
			\redcircle{(0.3,-0.8)} node[anchor=west,color=black] {\dotlabel{s-k-1}};
			\bluedot{(-0.3,-0.8)} node[anchor=east,color=black] {$\chk{b}g$};
			\draw{(0,-1.1)} node[anchor=north,color=black] {$i$};
			\draw{(0,0.3)} node[anchor=south,color=black] {$i$};
			\end{tikzpicture}
			\ = - \delta_{t,0} \tr_i(fg).
			\end{equation}
			
			\item \emph{Left adjunction}: For all $i=1,\dotsc,n$
			
			\noindent\begin{minipage}{0.5\linewidth}
				\begin{equation} \label{rel:zigzag-leftdown2}
				\begin{tikzpicture}[anchorbase]
				\draw[<-] (0,0) -- (0,0.6) arc(180:0:0.2) -- (0.4,0.4) arc(180:360:0.2) -- (0.8,1) node[anchor=south] {$i$};
				\end{tikzpicture}
				\ =\
				\begin{tikzpicture}[anchorbase]
				\draw[<-] (0,0) -- (0,1) node[anchor=south] {$i$};
				\end{tikzpicture}\ ,
				\end{equation}
			\end{minipage}%
			\begin{minipage}{0.5\linewidth}
				\begin{equation} \label{rel:zigzag-leftup2}
				\begin{tikzpicture}[anchorbase]
				\draw[<-] (0,1) -- (0,0.4) arc(180:360:0.2) -- (0.4,0.6) arc(180:0:0.2) -- (0.8,0) node[anchor=north] {$i$};
				\end{tikzpicture}
				\ =\
				\begin{tikzpicture}[anchorbase]
				\draw[->] (0,0) node[anchor=north] {$i$} -- (0,1);
				\end{tikzpicture}\ .
				\end{equation}
			\end{minipage}\par\vspace{\belowdisplayskip}
			
			\item \emph{Rotation relations}: For all $f \in F_i$, $i=1,\dotsc,n,$
			
			\noindent\begin{minipage}{0.4\linewidth}
				\begin{equation} \label{rel:token-rotation2}
				\begin{tikzpicture}[anchorbase]
				\draw[<-] (0,0) -- (0,1) node[anchor=south] {$i$};
				\bluedot{(0,0.5)} node[anchor=west,color=black] {$f$};
				\end{tikzpicture}
				\ :=\
				\begin{tikzpicture}[anchorbase]
				\draw[->] (0,1) node[anchor=south] {$i$} -- (0,0.4) arc(180:360:0.2) -- (0.4,0.6) arc(180:0:0.2) -- (0.8,0);
				\bluedot{(0.4,0.5)} node[anchor=west,color=black] {$f$};
				\end{tikzpicture}
				\ =\
				\begin{tikzpicture}[anchorbase]
				\draw[<-] (0,0) -- (0,0.6) arc(180:0:0.2) -- (0.4,0.4) arc(180:360:0.2) -- (0.8,1) node[anchor=south] {$i$};
				\bluedot{(0.4,0.5)};
				\draw (0.4,0.5) node[anchor=west,color=black] {$f$};
				\end{tikzpicture}
				\ ,
				\end{equation}
			\end{minipage}%
			\begin{minipage}{0.6\linewidth}
				\begin{equation} \label{rel:dot-rotation2}
				\begin{tikzpicture}[anchorbase]
				\draw[<-] (0,0) -- (0,1.4) node[anchor=south] {$i$};
				\redcircle{(0,0.7)};
				\end{tikzpicture}
				\ :=\
				\begin{tikzpicture}[anchorbase]
				\draw[->] (0,1.2) node[anchor=south] {$i$} -- (0,0.4) arc(180:360:0.2) -- (0.4,0.6) arc(180:0:0.2) -- (0.8,-0.2);
				\redcircle{(0.4,0.5)};
				\end{tikzpicture}
				\ =\
				\begin{tikzpicture}[anchorbase]
				\draw[<-] (0,-0.2) -- (0,0.6) arc(180:0:0.2) -- (0.4,0.4) arc(180:360:0.2) -- (0.8,1.2) node[anchor=south] {$i$};
				\redcircle{(0.4,0.5)};
				\end{tikzpicture},
				\end{equation}
			\end{minipage}\par\vspace{\belowdisplayskip}
			\begin{equation} \label{rel:crossing-rotation2}
			\begin{tikzpicture}[anchorbase]
			\draw[<-] (0,0) -- (1,1) node[anchor=south] {$i$};
			\draw[<-] (1,0) -- (0,1) node[anchor=south] {$j$};
			\end{tikzpicture}
			\ :=\
			\begin{tikzpicture}[anchorbase,scale=0.5]
			\draw[->] (-1.5,1.5) node[anchor=south] {$i$} .. controls (-1.5,0.5) and (-1,-1) .. (0,0) .. controls (1,1) and (1.5,-0.5) .. (1.5,-1.5);
			\draw[->] (-2,1.5)  node[anchor=south] {$j$} .. controls (-2,-2) and (1.5,-1.5) .. (0,0) .. controls (-1.5,1.5) and (2,2) .. (2,-1.5);
			\end{tikzpicture}
			\ =\
			\begin{tikzpicture}[anchorbase,scale=0.5]
			\draw[->] (1.5,1.5) node[anchor=south] {$j$} .. controls (1.5,0.5) and (1,-1) .. (0,0) .. controls (-1,1) and (-1.5,-0.5) .. (-1.5,-1.5);
			\draw[->] (2,1.5) node[anchor=south] {$i$} .. controls (2,-2) and (-1.5,-1.5) .. (0,0) .. controls (1.5,1.5) and (-2,2) .. (-2,-1.5);
			\end{tikzpicture}
			\ .
			\end{equation}
			
			\item \emph{Curl relations}: For all $r \ge 0$ and $i=1,\dotsc,n,$
			
			\noindent\begin{minipage}{0.5\linewidth}
				\begin{equation} \label{rel:left-dotted-curl2}
				\begin{tikzpicture}[anchorbase]
				\draw[->] (0,-0.75) node[anchor=north] {$i$} .. controls (0,0.5) and (-0.5,0.5) .. (-0.5,0) .. controls (-0.5,-0.5) and (0,-0.5) .. (0,0.75);
				\redcircle{(-0.5,0)} node[anchor=east,color=black] {$r$};
				\end{tikzpicture}
				= \sum_{s \ge 0}\sum_{b \in B_i}
				\begin{tikzpicture}[anchorbase]
				\draw[->] (0,-0.75) node[anchor=north] {$i$} -- (0,0.75);
				\draw[->] (-0.75,0) arc(90:450:0.3);
				\bluedot{(-0.45,-0.30)} node[anchor=east,color=black] {$b$};
				\draw{(-0.75,-0.6)} node[anchor=north] {$i$};
				\redcircle{(-1.05,-0.30)} node[anchor=east,color=black] {\dotlabel{r-s-1}};
				\redcircle{(0,0)} node[anchor=west,color=black] {$s$};
				\bluedot{(0,-0.30)} node[anchor=west,color=black] {$\chk{b}$};
				\end{tikzpicture}\ ,
				\end{equation}
			\end{minipage}%
			\begin{minipage}{0.5\linewidth}
				\begin{equation} \label{rel:right-dotted-curl2}
				\begin{tikzpicture}[anchorbase]
				\draw[->] (0,-0.75) node[anchor=north] {$i$} .. controls (0,0.5) and (0.5,0.5) .. (0.5,0) .. controls (0.5,-0.5) and (0,-0.5) .. (0,0.75);
				\redcircle{(0.5,0)} node[anchor=west,color=black] {$r$};
				\end{tikzpicture}
				= - \sum_{s \ge 0}\sum_{b \in B_i}
				\begin{tikzpicture}[anchorbase]
				\draw[->] (0,-0.75) node[anchor=north] {$i$} -- (0,0.75);
				\draw[<-] (0.95,0) arc(90:450:0.3);
				\bluedot{(0.65,-0.30)} node[anchor=east,color=black] {$\chk{b}$};
				\draw{(0.95,-0.6)} node[anchor=north] {$i$};
				\redcircle{(1.25,-0.30)} node[anchor=west,color=black] {\dotlabel{r-s-1}};
				\redcircle{(0,0)} node[anchor=east,color=black] {$s$};
				\bluedot{(0,0.3)} node[anchor=west,color=black] {$b$};
				\end{tikzpicture}\ .
				\end{equation}
			\end{minipage}\par\vspace{\belowdisplayskip}
			
			\item \emph{Bubble slides}:  For all $f \in F_i,$ $i=1,\dotsc,n,$ and $r \ge 0$,
			\begin{equation} \label{rel:clockwise-bubble-slide2}
			\begin{tikzpicture}[anchorbase]
			\draw[<-] (0,0.3) arc(90:450:0.3);
			\draw[->] (0.8,-1) node[anchor=north] {$i$} -- (0.8,1);
			\redcircle{(0.3,0)} node[anchor=west,color=black] {$r$};
			\bluedot{(-0.3,0)} node[anchor=east,color=black] {$f$};
			\draw{(0,-0.3)} node[anchor=north] {$j$};
			\end{tikzpicture}
			\ =\
			\begin{tikzpicture}[anchorbase]
			\draw[<-] (0,0.3) arc(90:450:0.3);
			\draw[->] (-0.8,-1) node[anchor=north] {$i$} -- (-0.8,1);
			\redcircle{(0.3,0)} node[anchor=west,color=black] {$r$};
			\bluedot{(-0.3,0)} node[anchor=east,color=black] {$f$};
			\draw{(0,-0.3)} node[anchor=north] {$j$};
			\end{tikzpicture}
			\ - \delta_{i,j}\sum_{t \ge 0} \sum\limits_{s=0}^{t}\sum_{a,b \in B_i}\
			\begin{tikzpicture}[anchorbase]
			\draw[<-] (0.3,0.3) arc(90:450:0.3);
			\draw[->] (-0.8,-1) node[anchor=north] {$i$} -- (-0.8,1);
			\redcircle{(0.6,0)} node[anchor=west,color=black] {\dotlabel{r-t-2}};
			\bluedot{(0,0)} node[anchor=east,color=black] {$\chk{a}f$};
			\bluedot{(-0.8,0.6)} node[anchor=west,color=black] {$b a \chk{b}$};
			\redcircle{(-0.8,0.1)} node[anchor=east,color=black] {$t$};
			\draw{(0.3,-0.3)} node[anchor=north] {$i$};
			\end{tikzpicture}
			\ ,
			\end{equation}
			\begin{equation} \label{rel:counterclockwise-bubble-slide2}
			\begin{tikzpicture}[anchorbase]
			\draw[->] (0,0.3) arc(90:450:0.3);
			\draw[->] (-0.8,-1) node[anchor=north] {$i$} -- (-0.8,1);
			\redcircle{(0.3,0)} node[anchor=west,color=black] {$r$};
			\bluedot{(-0.3,0)} node[anchor=east,color=black] {$f$};
			\draw{(0,-0.3)} node[anchor=north] {$j$};
			\end{tikzpicture}
			\ =\
			\begin{tikzpicture}[anchorbase]
			\draw[->] (0,0.3) arc(90:450:0.3);
			\draw[->] (0.8,-1) node[anchor=north] {$i$}-- (0.8,1);
			\redcircle{(0.3,0)} node[anchor=west,color=black] {$r$};
			\bluedot{(-0.3,0)} node[anchor=east,color=black] {$f$};
			\draw{(0,-0.3)} node[anchor=north] {$j$};
			\end{tikzpicture}
			\ - \delta_{i,j}\sum_{t \ge 0}\sum\limits_{s=0}^{t} \sum_{a,b \in B_i} \
			\begin{tikzpicture}[anchorbase]
			\draw[->] (-2.2,0.3) arc(90:450:0.3);
			\draw[->] (-0.5,-1)  node[anchor=north] {$i$} -- (-0.5,1);
			\redcircle{(-1.9,0)} node[anchor=west,color=black] {\dotlabel{r-t-2}};
			\bluedot{(-2.5,0)} node[anchor=east,color=black] {$\chk{a}f$};
			\bluedot{(-0.5,0.6)} node[anchor=west,color=black] {$b a \chk{b}$};
			\redcircle{(-0.5,0.1)} node[anchor=west,color=black] {$t$};
			\draw{(-2.2,-0.3)} node[anchor=north] {$i$};
			\end{tikzpicture}
			\ .
			\end{equation}

			\item \emph{Alternating braid relation}: For all $i,j,l=1,\dotsc,n$
			\begin{equation} \label{rel:braid-alternating2}
			\begin{tikzpicture}[anchorbase]
			\draw[->] (0,0) node[anchor=north] {$i$} -- (1,1);
			\draw[->] (1,0) node[anchor=north] {$l$} -- (0,1);
			\draw[<-] (0.5,0) node[anchor=north] {$j$} .. controls (0,0.5) .. (0.5,1);
			\end{tikzpicture}
			\ -\
			\begin{tikzpicture}[anchorbase]
			\draw[->] (0,0) node[anchor=north] {$i$} -- (1,1);
			\draw[->] (1,0)  node[anchor=north] {$l$}-- (0,1);
			\draw[<-] (0.5,0) node[anchor=north] {$j$} .. controls (1,0.5) .. (0.5,1);
			\end{tikzpicture}
			\ =\
			\begin{cases}
			\delta_{i,j}\delta_{j,l}\sum_{r,s,t \ge 0} \sum_{a,b,e \in B_i}
			\begin{tikzpicture}[anchorbase]
			\draw[<-] (0,2) -- (0,1.7) arc (180:360:0.3) -- (0.6,2) node[anchor=south] {$i$};
			\draw[->] (0,-0.1) node[anchor=north] {$i$} -- (0,0.5) arc (180:0:0.3) -- (0.6,-0.1);
			\draw[<-] (1,1.4) arc(90:-270:0.3);
			\draw[->] (2.3,-0.1)  node[anchor=north] {$i$} -- (2.3,2);
			\draw{(1,0.8)} node[anchor=north] {$i$};
			\redcircle{(0.6,1.7)} node[anchor=east,color=black] {$r$};
			\bluedot{(0,1.7)} node[anchor=east,color=black] {$a$};
			\redcircle{(0.6,0.5)} node[anchor=east,color=black] {$s$};
			\bluedot{(0,0.5)} node[anchor=east,color=black] {$\chk{e}$};
			\bluedot{(1.3,1.1)} node[anchor=west,color=black] {$\chk{a} b$};
			\redcircle{(0.7,1.1)} node[anchor=east,color=black] {\dotlabel{-r-s-t-3}};
			\bluedot{(2.3,1.25)} node[anchor=west,color=black] {$e$};
			\redcircle{(2.3,0.95)} node[anchor=west,color=black] {$t$};
			\bluedot{(2.3,0.65)} node[anchor=west,color=black] {$\chk{b}$};
			\end{tikzpicture}
			& \text{if } k \ge 2,
			\\
			0 & \text{if } -1 \le k \le 1,
			\\
			\delta_{i,j}\delta_{j,l}\sum_{r,s,t \ge 0} \sum_{a,b,e \in B_i} \
			\begin{tikzpicture}[anchorbase]
			\draw[<-] (0,0) -- (0,0.2) arc(180:0:0.3) -- (0.6,0)  node[anchor=north] {$i$};
			\draw[->] (0,2)  node[anchor=south] {$i$} -- (0,1.5) arc(180:360:0.3) -- (0.6,2);
			\draw[->] (1.5,1.15) arc(90:-270:0.3);
			\draw[->] (-0.8,0)  node[anchor=north] {$i$} -- (-0.8,2);
			\bluedot{(0.6,1.8)} node[anchor=west,color=black] {\dotlabel{e\chk{b}}};
			\draw{(1.5,0.55)} node[anchor=north] {$i$};
			\redcircle{(0.6,1.5)} node[anchor=west,color=black] {$r$};
			\redcircle{(0.6,0.2)} node[anchor=west,color=black] {$s$};
			\bluedot{(0,0.2)} node[anchor=west,color=black] {$a$};
			\bluedot{(1.8,0.85)} node[anchor=west,color=black] {$\chk{a} b$};
			\redcircle{(1.2,0.85)} node[anchor=east,color=black] {\dotlabel{-r-s-t-3}};
			\bluedot{(-0.8,0.8)} node[anchor=east,color=black] {$\chk{e}$};
			\redcircle{(-0.8,1.2)} node[anchor=east,color=black] {$t$};
			\end{tikzpicture}
			& \text{if } k \le -2.
			\end{cases}
			\end{equation}
		\end{enumerate}
	\end{theo}
	\begin{proof}
		The infinite Grassmanian relations, \cref{eq:inf-grass1.2,eq:inf-grass2.2,eq:inf-grass3.2}, follow from the infinite Grassmanian relations, $(1.27)$ to $(1.29)$, in \cite[Theorem 1.3]{S17}, respectively, by letting $f\in F_i$, $i=1,\dotsc,n$. The remaining relations, \cref{rel:zigzag-leftdown2} to \cref{rel:braid-alternating2}, follow from \cite[Theorem 1.3]{S17} by attaching the appropriate idempotents to the bottoms or tops of both sides of the relations $(1.30)$ to $(1.39)$ in \cite[Theorem 1.3]{S17}, respectively . 
	\end{proof}

	\section{An Equivalence of Categories}\label{section:equivalence}
	
	Let $\mathcal{C}$ and $\mathcal{D}$ be strict $\kk$-linear monoidal categories.
	
	\begin{defin} 
		Following \cite[Def. 1.1]{GK14}, the $\textit{tensor product}$ of $\mathcal{C}$ and $\mathcal{D}$, denoted $\mathcal{C}\otimes\mathcal{D}$, is defined as follows. For all $U,W,Y\in Ob\left(\mathcal{C}\right)$ and $V,X,Z\in Ob\left(\mathcal{D}\right)$, we have
		
		\begin{itemize}
			\item
			$Ob\left(\mathcal{C}\otimes\mathcal{D}\right)=\{\left(Y,X\right)|Y\in Ob(\mathcal{C})$ and $X\in Ob\left(\mathcal{D}\right)\}$,
			\item $Hom_{\mathcal{C}\otimes\mathcal{D}}\left(\left(W,X\right),\left(Y,Z\right)\right)=Hom_{\mathcal{C}}\left(W,Y\right)\otimes_{\kk}Hom_{\mathcal{D}}\left(X,Z\right)$, and
			\item 
			For all morphisms of the form $g=g_\mathcal{C}\otimes g_\mathcal{D}:\left(U,V\right)\to \left(W,X\right)$ and $f=f_\mathcal{C}\otimes f_\mathcal{D}: \left(W,X\right)\to \left(Y,Z\right)$ in $\mathcal{C}\otimes\mathcal{D}$, the composition $f\circ g$ is defined as
			\begin{center} 
				$f\circ g =  \left(f_\mathcal{C}\otimes f_\mathcal{D}\right)\circ\left(g_\mathcal{C}\otimes g_\mathcal{D}\right)=\left(f_\mathcal{C}\circ g_\mathcal{C}\right)\otimes \left(f_\mathcal{D}\circ g_\mathcal{D}\right): \left(U,V\right)\to \left(Y,Z\right)$.
			\end{center}
			Define composition between any two morphisms of $\mathcal{C}\otimes\mathcal{D}$ by extending the definition on simple tensors by linearity.
		\end{itemize}
	\end{defin}
	It is easy to see that the tensor product, $\mathcal{C}\otimes\mathcal{D}$, is a category.
	\details{
		\begin{lem}	
			The tensor product $\mathcal{C}\otimes_{\kk}\mathcal{D}$ is a category.
		\end{lem}
		\begin{proof}
			\begin{itemize}
				\item ($\textit{Composition of morphisms is associative}$)
				\
				Let $(S,T),(U,V), (W,X), (Y,Z)\in Ob(\mathcal{C}\otimes_{\kk}\mathcal{D})$ and let $h:(S,T)\to (U,V)$, $g:(U,V)\to (W,X)$, and $f:(W,X)\to (Y,Z)$ be morphisms in $Mor(\mathcal{C}\otimes_{\kk}\mathcal{D})$. Assume $f,g$, and $h$ are simple tensors. By definition, there exist morphisms $h_{\mathcal{C}}:S\to U$, $g_{\mathcal{C}}:U\to W$, and $f_{\mathcal{C}}: W\to Y$ in $Mor(\mathcal{C})$ and $h_{\mathcal{D}}:T\to V$, $g_{\mathcal{D}}:V\to X$, and $f_{\mathcal{D}}: X\to Z$ in $Mor(\mathcal{D})$ with $f=f_{\mathcal{C}}\otimes f_{\mathcal{D}}$, $g=g_{\mathcal{C}}\otimes g_{\mathcal{D}}$, and $h=h_{\mathcal{C}}\otimes h_{\mathcal{D}}$. Therefore, we have 
				\begin{align*}
				(f\circ g) \circ h &= ((f_{\mathcal{C}}\otimes f_{\mathcal{D}})\circ (g_{\mathcal{C}}\otimes g_{\mathcal{D}})) \circ (h_{\mathcal{C}}\otimes h_{\mathcal{D}}) = ((f_{\mathcal{C}}\circ g_{\mathcal{C}})\circ h_{\mathcal{C}})\otimes ((f_{\mathcal{D}}\circ g_{\mathcal{D}})\circ h_{\mathcal{D}}) \\&= (f_{\mathcal{C}}\circ (g_{\mathcal{C}}\circ h_{\mathcal{C}}))\otimes (f_{\mathcal{D}}\circ (g_{\mathcal{D}}\circ h_{\mathcal{D}})) = (f_{\mathcal{C}}\otimes f_{\mathcal{D}})\circ ((g_{\mathcal{C}}\otimes g_{\mathcal{D}}) \circ (h_{\mathcal{C}}\otimes h_{\mathcal{D}})) = f\circ (g\circ h).
				\end{align*}
				To show that composition of all morphisms in $\mathcal{C}\otimes_{\kk}\mathcal{D}$ is associative, we extend this property of the simple tensors by linearity.
				\item ($\textit{Existence of identity morphisms}$) 
				\
				Let $(X,Y)\in \mathcal{C}\otimes\mathcal{D}$. Then $1_{X\otimes Y}=1_X\otimes 1_Y$, where $1_X$ is the identity morphism on $X$ and $1_Y$ is the identity morphism on $Y$. Indeed, for all $(x,y)\in (X,Y)$, we have 	
				\[
				1_X\otimes 1_Y(x,y)=(1_X(x),1_Y(y))=(x,y).
				\]
			\end{itemize}
			Therefore, $\mathcal{C}\otimes_{\kk}\mathcal{D}$ is a category.
		\end{proof}
		This completes the proof.
	}
	We add additional structure to $\mathcal{C}\otimes\mathcal{D}$ to turn it into a strict $\kk$-linear monoidal category, and we call resulting category $\mathcal{C}\otimes\mathcal{D}$ (i.e.: we do not change the name):
	
	\begin{itemize}
		\item Equip $\mathcal{C}\otimes\mathcal{D}$ with a bifunctor ${\otimes}: \left(\mathcal{C}\otimes\mathcal{D}\right) \times \left(\mathcal{C}\otimes\mathcal{D}\right) \to \mathcal{C}\otimes\mathcal{D}$. On pairs of objects, we define $\left(\left(A_\mathcal{C},A_\mathcal{D}\right),\left(B_\mathcal{C},B_\mathcal{D}\right)\right)\mapsto \left(A_\mathcal{C}\otimes B_\mathcal{C}, A_\mathcal{D} \otimes B_\mathcal{D}\right)$. On pairs of morphisms that are simple tensors, we define $ \left(\left(f_{\mathcal{C}}\otimes f_{\mathcal{D}}\right),\left(g_{\mathcal{C}}\otimes g_{\mathcal{D}}\right)\right)\mapsto \left(f_\mathcal{C}\otimes g_\mathcal{C}\right)\otimes \left(f_\mathcal{D} \otimes g_\mathcal{D}\right)$, and we extend this definition by linearity, and
		\item Let $a\in\kk$, and $f=f_\mathcal{C}  \otimes f_\mathcal{D}$. We define $\kk$-action by the following equation:
		\[ a\cdot f=a\cdot \left(f_\mathcal{C}\otimes f_\mathcal{D}\right)=\left(a\cdot f_\mathcal{C}\right) \otimes  f_\mathcal{D}= f_\mathcal{C}\otimes \left(a\cdot f_\mathcal{D}\right).
		\]
	\end{itemize}
	It is straightforward to verify that $\mathcal{C}\otimes\mathcal{D}$ is a strict $\kk$-linear monoidal category, with $\mathbf{1}=\left(1_\mathcal{C},1_\mathcal{D}\right)$ being the unit object of $\mathcal{C}\otimes\mathcal{D}$, where $1_\mathcal{C}$ is the unit object of $\mathcal{C}$ and $1_\mathcal{D}$ is the unit object of $\mathcal{D}$.
	\details{
		\begin{lem}
			The tensor product $\mathcal{C}\otimes_{\kk}\mathcal{D}$ is a strict monoidal category.
		\end{lem}
		\begin{proof}
			For all $A=(A_{\mathcal{C}},A_{\mathcal{D}}),B=(B_\mathcal{C},B_\mathcal{D}),C=(C_\mathcal{C},C_\mathcal{D})\in \mathcal{C}\otimes_{\kk}\mathcal{D}$, we have			
			\begin{itemize}
				\item ($\textit{Associativity of bifunctor}$)
				\begin{align*}	
				(A\otimes B)\otimes C&=((A_{\mathcal{C}},A_{\mathcal{D}})\otimes (B_\mathcal{C},B_\mathcal{D}))\otimes (C_\mathcal{C},C_\mathcal{D})=(A_{\mathcal{C}}\otimes B_\mathcal{C} ,A_{\mathcal{D}}\otimes B_\mathcal{D})\otimes(C_\mathcal{C},C_\mathcal{D})\\&=((A_{\mathcal{C}}\otimes B_\mathcal{C})\otimes C_\mathcal{C},(A_{\mathcal{D}}\otimes B_\mathcal{D})\otimes C_\mathcal{D}) =(A_{\mathcal{C}}\otimes B_\mathcal{C}\otimes C_\mathcal{C},A_{\mathcal{D}}\otimes B_\mathcal{D}\otimes C_\mathcal{D})\\&=(A_{\mathcal{C}}\otimes(B_\mathcal{C}\otimes C_\mathcal{C}),A_{\mathcal{D}}\otimes (B_\mathcal{D}\otimes C_\mathcal{D}))=(A_\mathcal{C},A_{\mathcal{D}})\otimes(B_\mathcal{C}\otimes C_\mathcal{C},B_\mathcal{D}\otimes C_\mathcal{D})\\&=(A_\mathcal{C},A_{\mathcal{D}})\otimes((B_\mathcal{C},B_{\mathcal{D}})\otimes(C_\mathcal{C},C_{\mathcal{D}}))=A\otimes (B\otimes C), \quad \text{and}
				\end{align*}
				\item ($\textit{Existence of unit object}$)
				$\mathbf{1}\otimes A=(1_\mathcal{C},1_\mathcal{D})\otimes(A_{\mathcal{C}},A_{\mathcal{D}})=(1_\mathcal{C}\otimes A_\mathcal{C},1_{\mathcal{D}}\otimes A_\mathcal{D})=(A_\mathcal{C}, A_\mathcal{D})=A,$ and\\	
				$A\otimes \mathbf{1}=(A_{\mathcal{C}},A_{\mathcal{D}})\otimes (1_\mathcal{C},1_\mathcal{D})=(A_\mathcal{C}\otimes 1_\mathcal{C},A_\mathcal{D}\otimes 1_{\mathcal{D}})=(A_\mathcal{C},A_\mathcal{D})=A$.
			\end{itemize}
			
			Therefore, $\mathcal{C}\otimes_{\kk}\mathcal{D}$ is a strict monoidal category.
		\end{proof}	
		\begin{lem}
			The tensor product $\mathcal{C}\otimes_{\kk}\mathcal{D}$ is a strict $\kk$-linear monoidal category.
		\end{lem}
		\begin{proof}
			Since both $Mor(\mathcal{C})$ and $Mor(\mathcal{D})$ are $\kk$-modules and composition of morphisms in both $Mor(\mathcal{C})$ and $Mor(\mathcal{D})$ is $\kk$-linear, $Mor(\mathcal{C}\otimes_{\kk}\mathcal{D})$ is a $\kk$-module, and composition of morphisms in $Mor(\mathcal{C}\otimes_\kk \mathcal{D})$ is $\kk$-linear. Therefore, $\mathcal{C}\otimes_{\kk}\mathcal{D}$ is a strict $\kk$-linear monoidal category.
		\end{proof}
		This completes the proof.
	}

	\begin{defin}\label{def:category-splits}
		Suppose $\mathcal{C}$ is a strict $\kk$-linear monoidal category and $X=X_1\cup\dotsb\cup X_n,$ $X_i\subseteq Ob\left(\mathcal{C}\right),$ is a subclass of the objects of $\mathcal{C}$ such that the objects of $\mathcal{C}$ are generated by the objects in $X$. Furthermore, suppose $\mathcal{C}_i$ is a subcategory of $\mathcal{C}$ that is monoidal with respect to the restriction of the tensor product from $\mathcal{C},$ and such that $Ob\left(\mathcal{C}_i\right)=X_i$. We define a functor $F:\mathcal{C}_1\otimes\dotsb\otimes \mathcal{C}_n\to\mathcal{C}$ as follows. On objects, we define
		\[F\left(c_1,\dotsc, c_n\right)=c_1\otimes\dotsb\otimes c_n,\]
		and on morphisms that are simple tensors, we define 
		\[F\left(f_1\otimes\dotsb\otimes f_n\right)=f_1\otimes\dotsb\otimes f_n,\]
		and extend $F$ by linearity.
	\end{defin}
	\details{
		\begin{lem}
			The map $F$ of \cref{def:category-splits} is a functor.
		\end{lem}
		\begin{proof}
			\begin{itemize}
				\item \textit{Maps identity to identity:} Let $(c_1,\dotsc, c_k)\in \mathcal{C}_1\otimes\dotsb\otimes \mathcal{C}_k$. We have 
				\[F(1_{c_1}\otimes\dotsb\otimes 1_{c_k})=1_{c_1}\otimes\dotsb\otimes 1_{c_k}.\]
				\item \textit{Preserves composition of morphisms:} For all simple tensors $f_1\otimes\dotsb\otimes f_k:(b_1,\dotsb,b_k)\to(c_1,\dotsc,c_k)$ and $g_1\otimes\dotsb\otimes g_k:(c_1,\dotsc,c_k)\to(d_1,\dotsc,d_k)$ in $Mor(\mathcal{C}_1\otimes\dotsb\otimes \mathcal{C}_k)$, we have
				\begin{align*}
				&F((g_1\otimes\dotsb\otimes g_k)\circ(f_1\otimes\dotsb\otimes f_k))=G((g_1\circ f_1)\otimes\dotsb\otimes(g_k\circ f_k))=(g_1\circ f_1)\otimes\dotsb\otimes(g_k\circ f_k)\\&=(g_1\otimes\dotsb\otimes g_k)\circ(f_1\otimes\dotsb\otimes f_k)=F(g_1\otimes\dotsb\otimes g_k)\circ F(f_1\otimes\dotsb\otimes f_k).
				\end{align*}
				Using the $\kk$-linearity of $F$, we see that $F$ preserves the composition of all morphisms in $Mor(\mathcal{C}_1\otimes\dotsb\otimes \mathcal{C}_k)$.
			\end{itemize}
		\end{proof}
		This completes the proof.
	}
	\begin{theo}\label{theo:equivalence}
		Suppose $\mathcal{C}$ is a strict $\kk$-linear monoidal category and $X=X_1\cup\dotsb\cup X_n$, $X_i\subseteq Ob\left(\mathcal{C}\right),$ is a subclass of the objects of $\mathcal{C}$ such that the objects of $\mathcal{C}$ are generated by the objects in $X,$ and such that $x_i\otimes x_j$ is isomorphic to $x_j\otimes x_i$ whenever $x_i\in X_i$ and $x_j\in X_j$ for $i\neq j$. Suppose $\mathcal{C}_i$ is a subcategory of $\mathcal{C}$ that is monoidal with respect to the restriction of the tensor product from $\mathcal{C},$ and such that $Ob\left(\mathcal{C}_i\right)=X_i,$ and such that the induced functor $F$ from \cref{def:category-splits} is full on $\mathcal{C}_1\otimes\dotsb\otimes \mathcal{C}_n$. Then $\mathcal{C}$ is equivalent to $\mathcal{C}_1\otimes\dotsb\otimes \mathcal{C}_n$.
	\end{theo}
	\begin{proof}
		We claim that the functor $F$ from \cref{def:category-splits} induces an equivalence of categories between $\mathcal{C}$ and $\mathcal{C}_1\otimes\dotsb\otimes \mathcal{C}_n$.\\
		\indent $F$ \textit{is essentially surjective on objects}. Since $x_i\otimes x_j$ and $x_j\otimes x_i$ are isomorphic whenever $x_i\in X_i$ and $x_j\in X_j$ for $i\neq j$, we can construct an invertible morphism from an object in $\mathcal{C}$ to an object in $\im F$ by composing the invertible morphisms between the objects $x_i\otimes x_j$ and $x_j\otimes x_i$ when $i\neq j$. This morphism acts on a tensor in $Ob(\mathcal{C})$ by shifting its components one by one so that all objects $x_i\in X_i$ are to the left of all objects $x_j\in X_j$ for $i<j$. After performing this procedure on a tensor in $Ob(\mathcal{C})$, we see that the resulting object is in $\im F.$\\
		\indent $F$ \textit{is full and faithful on morphisms}. $F$ is full by assumption. Moreover, if $F(g)=F(f)$, then by the definition of $F$, we have $f=g$, since $F\left(g\right)=g$ and $F\left(f\right)=f$. Therefore, $F$ is full and faithful.
	\end{proof}

	Now, suppose $\mathcal{D}$ is a strict $\kk$-linear monoidal category. We are interested in a particular subcategory of the Karoubi envelope, $Kar\left(\mathcal{D}\right)$, which we call the partial Karoubi envelope of $\mathcal{D}$ and denote $PK\left(\mathcal{D}\right)$. 

\begin{defin}\label{def:Karoubi-envelope}
	Suppose $\mathcal{D}$ is a strict $\kk$-linear monoidal category. Denote the domain and codomain of a morphism $f$ in $\mathcal{D}$ by $dom(f)$ and $codom(f),$ respectively, and let $\mathcal{I}$ be a subcollection of idempotent morphisms in $\mathcal{D}$ satisfying the following conditions:
	\begin{itemize} 
		\item For all $e\in\mathcal{I},$ $dom(e)=codom(e)$,
		\item If $D=dom(e)$ for some $e\in\mathcal{I},$ then $\id_D\circ e'= \id_D = e'\circ \id_D$ for all $e'\in\mathcal{I}$ such that $D=dom(e').$
	\end{itemize} 
	Then the $\textit{partial Karoubi envelope}$ of $\mathcal{D}$ with respect to $\mathcal{I}$, denoted $PK(\mathcal{D})$ (or $PK(\mathcal{D})_\mathcal{I}$, if we wish to specify our choice $\mathcal{I}$), is defined as follows:
	
	\begin{itemize}
		\item \
		$Ob\left(PK\left(\mathcal{D}\right)\right)=\{\left(D_1\otimes\dotsb\otimes D_m, e_{1}\otimes\dotsb\otimes e_{m}\right)| e_i\in \mathcal{I}, D_i=dom(e_i), i=1,\dotsc,m, m\geq 0\}$, where for $m=0$, we define $\left(D_1\otimes\dotsb\otimes D_m, e_{1}\otimes\dotsb\otimes e_{m}\right):=\left(\mathbf{1},id_{\mathbb{1}}\right)$, where $\mathbf{1}$ is the unit object of $\mathcal{D}$ and $id_{\mathbf{1}}$ is the identity morphism on $\mathbf{1}$ in $\mathcal{D}$,
		\item \ $Hom_{PK(\mathcal{D})}\left(\left(D_1,e_1\right),\left(D_2,e_2\right)\right)=\{g\in Hom_{\mathcal{D}}\left(D_1,D_2\right)|g\circ e_1=g=e_2\circ g\},$
		\item \
		For all morphisms $g':\left(D_1,e_1\right)\to\left(D_2,e_2\right)$ and $h': \left(D_2,e_2\right)\to\left(D_3,e_3\right)$ in
		
		\noindent $PK\left(\mathcal{D}\right)$, there are morphisms $g:D_1\to D_2$ and $h: D_2\to D_3$ in $\mathcal{D}$ that equal $g'$ and $h'$, respectively. The composition $h'\circ g':\left(D_1,e_1\right)\to \left(D_3,e_3\right)$ in $PK\left(\mathcal{D}\right)$ is then defined as the composition $h\circ g:D_1\to D_3$ in $\mathcal{D}$. Indeed, we have that
		\[\left(h\circ g\right)\circ e_1 = h\circ\left(g\circ e_1\right) =h\circ g=\left(e_3\circ h\right)\circ g=e_3\circ \left(h\circ g\right),\] and
		\item \
		Equip $PK\left(\mathcal{D}\right)$ with a bifunctor $\otimes: \mathcal{D} \times \mathcal{D} \to \mathcal{D}$ sending pairs of objects $\left(\left(D_1,e_1\right),\left(D_2,e_2\right)\right)\mapsto \left(D_1\otimes D_2, e_1 \otimes e_2\right)$ and pairs of morphisms $(h,g)\mapsto h\otimes g.$
	\end{itemize}
\end{defin}
The partial Karoubi envelope, $PK\left(\mathcal{D}\right)$, is a strict monoidal subcategory of the Karoubi envelope $Kar\left(\mathcal{D}\right)$, with the unit object of $PK\left(\mathcal{D}\right)$ being $\left(\mathbf{1},id_{\mathbf{1}}\right)$, where $\mathbf{1}$ is the unit object of $\mathcal{D}$, and $\id_{\mathbf{1}}$ is the identity morphism of $\mathcal{D}$. 

We define a $\kk$-action in $PK\left(\mathcal{D}\right)$ to turn it into a strict $\kk$-linear monoidal category, and we call resulting category $PK\left(\mathcal{D}\right)$ (i.e.: we do not change the name). Let $a\in\kk$, and  $g':\left(D_1,e_1\right)\to\left(D_2,e_2\right)\in Mor\left(\mathcal{D}\right)$. Then there is a morphism $g:D_1\to D_2\in Mor\left(\mathcal{D}\right)$ equalling $g$, with the property $g\circ e_1=g=e_2\circ g$. We define $\kk$- action as follows:

\begin{itemize}
	\item $a\cdot g':\left(D_1,e_1\right)\to\left(D_2,e_2\right)\in Mor\left(\mathcal{D}\right)$ equals the morphism $a\cdot g:D_1\to D_2\in Mor\left(\mathcal{D}\right)$, so that 
	\[\left(a\cdot g\right)\circ e_1=a\cdot\left(g\circ e_1\right)=a\cdot g=a\cdot\left(e_2\circ g\right)=\left(a\cdot e_2\right)\circ g=\left(e_2 \cdot a\right)\circ g=e_2\circ\left(a\cdot g\right).\]
\end{itemize}
\details{
	\begin{lem}
		The partial Karoubi envelope, $PK(Heis_{F,k})$, is a strict monoidal subcategory of $Kar(Heis_{F,k}).$
	\end{lem}
	\begin{proof}
		For all $(Q_{\epsilon_1},e_{\gamma_1}),(Q_{\epsilon_2},e_{\gamma_2}),(Q_{\epsilon_3},e_{\gamma_3})\in PK(Heis_{F,k})$, we have			
		\begin{itemize}
			\item ($\textit{Composition of morphisms is associative}$)
			\
			This follows from the fact that composition of morphisms in $Mor(Heis_{F,k})$ is associative. 
			\item ($\textit{Existence of identity morphisms}$) 
			\
			Let $(Q_{\epsilon},e_{\gamma})\in Ob(PK(Heis_{F,k}))$. We claim $1_{(Q_{\epsilon},e_{\gamma})}=e_\gamma$. Indeed, we have,	
			\[e_\gamma \circ e_{\gamma} = e_{\gamma} , \]
			and for all morphisms $f':(Q_{\epsilon},e_{\gamma})\to (Q_{\nu},e_{\eta})$ and $g':(Q_{\mu},e_{\omega})\to (Q_{\epsilon},e_{\gamma})$ in $Mor(PK(Heis_{F,k}))$, we have $f'=f:Q_{\epsilon}\to Q_{\nu}$ and $g'=g:Q_\mu\to Q_\epsilon$ for some morphism in $Mor(Heis_{F_1,k}\otimes\dotsb\otimes Heis_{F_k,k})$, so
			\[e_\gamma \circ g'= e_\gamma \circ g = g = g',\]
			and 
			\[ f'\circ e_\gamma=  f\circ e_\gamma = f= f'.\] 
			\item ($\textit{Associativity of bifunctor}$)
			\begin{align*}		
			((Q_{\epsilon_1},e_{\gamma_1})&\otimes(Q_{\epsilon_2},e_{\gamma_2}))\otimes (Q_{\epsilon_3},e_{\gamma_3})=(Q_{\epsilon_1}\otimes Q_{\epsilon_2} ,e_{\gamma_1}\otimes e_{\gamma_2})\otimes(Q_{\epsilon_3},e_{\gamma_3})\\&=((Q_{\epsilon_1}\otimes Q_{\epsilon_2})\otimes Q_{\epsilon_3} ,(e_{\gamma_1}\otimes e_{\gamma_2})\otimes e_{\gamma_3})=(Q_{\epsilon_1}\otimes (Q_{\epsilon_2}\otimes Q_{\epsilon_3}) ,e_{\gamma_1}\otimes (e_{\gamma_2}\otimes e_{\gamma_3}))\\& =(Q_{\epsilon_1},e_{\gamma_1})\otimes(Q_{\epsilon_2}\otimes Q_{\epsilon_3},e_{\gamma_2}\otimes e_{\gamma_3})=(Q_{\epsilon_1},e_{\gamma_1})\otimes((Q_{\epsilon_2},e_{\gamma_2})\otimes(Q_{\epsilon_3},e_{\gamma_3})), \quad \text{and}
			\end{align*}
			\item ($\textit{Existence of unit object}$)
			$(\mathbf{1},id_{\mathbf{1}})\otimes(Q_{\epsilon},e_{\gamma})=(\mathbf{1}\otimes Q_{\epsilon},id_{\mathbf{1}}\otimes e_{\gamma})=(Q_{\epsilon}, e_{\gamma}), \quad$ and\\	
			$(Q_{\epsilon},e_{\gamma})\otimes(\mathbf{1},id_{\mathbf{1}})=(Q_{\epsilon}\otimes\mathbf{1},e_{\gamma}\otimes id_{\mathbf{1}})=(Q_{\epsilon}, e_{\gamma}).$			
		\end{itemize}
		Therefore, $PK(Heis_{F,k})$ is a strict monoidal category.
	\end{proof}
	\begin{lem}
		The partial Karoubi envelope, $PK(Heis_{F,k})$, is a strict $\kk$-linear monoidal subcategory of $Kar(Heis_{F,k}).$
	\end{lem}
	\begin{proof}
		Since $Mor(Heis_{F,k})$ is a $\kk$-module and composition of morphisms in $Mor(Heis_{F,k})$ is $\kk$-linear, it follows that $PK(Heis_{F,k})$ is a strict $\kk$-linear monoidal category.
	\end{proof}
	This completes the proof.
}

\begin{defin}\label{def:colored-subcategory}
	Let $\mathcal{D}$ be a strict $\kk$-linear monoidal category, and let $\mathcal{I}$ be a collection of idempotent morphisms in $\mathcal{D}$ satisfying the conditions of \cref{def:Karoubi-envelope}. We define $\mathcal{D}_\mathcal{I}$ as the subcategory of $\mathcal{D}$ that is monoidal with respect to the restriction of the tensor product of $\mathcal{D}$ and such that $Mor\left(\mathcal{D}_\mathcal{I}\right)=Mor\left(PK(\mathcal{D})_\mathcal{I}\right)$ and such that $Ob\left(\mathcal{D}_\mathcal{I}\right)$ is generated by the objects $dom(e),$ where $e\in\mathcal{I}.$
\end{defin}

	\begin{lem}\label{lem:mini-equivalence}	
	Let $\mathcal{D}$ be a strict $\kk$-linear monoidal category. Let $\mathcal{I}$ be a subcollection of idempotent morphisms in $\mathcal{D}$ satisfying the conditions of \cref{def:Karoubi-envelope}, and suppose $A_1,\dotsc,A_n$ are subcollections of elements in $\mathcal{I}$. Furthermore, for each $i,$ impose the condition that for any object $\left(D,e\right)\in Ob\left(PK\left(\mathcal{D}\right)_{A_i}\right)$, $e$ is uniquely determined by $D$.
	Then the categories $PK\left(\mathcal{D}\right)_{A_1}\otimes\dotsb\otimes PK\left(\mathcal{D}\right)_{A_n}$ and $\mathcal{D}_{A_1}\otimes\dotsb\otimes \mathcal{D}_{A_n}$ are isomorphic.
	\end{lem}
	\begin{proof}
	We first show that, for $i=1,\dotsc,n,$ the categories $\mathcal{D}_{A_i}$ and $PK\left(\mathcal{D}\right)_{A_i}$ are isomorphic. 
	
	For $i=1,\dotsc,n,$ define a functor $G_i:\mathcal{D}_{A_i}\to PK\left(\mathcal{D}\right)_{A_i}$ as follows. On objects, define
	\[
	G_i\left(D\right)=\left(D,e\right),
	\]
	where $e$ is the unique morphism in $PK\left(\mathcal{D}\right)_{A_i}$ determined by $D$, and on morphisms, define
	\[
	G_i\left(f\right)=f.
	\]
	
	For $i=1,\dotsc,n,$ define another functor $H_i:PK\left(\mathcal{D}\right)_{A_i}\to \mathcal{D}_{A_i}$ as follows. On objects, define 
	\[
	H_i\left(D,e\right)=D,
	\]
	and on morphisms, define
	\[
	H_i\left(f\right)=f.
	\]
	
	Clearly, $G_i$ preserves the defining relations in $\mathcal{D}_{A_i}$: if $f, g\in Mor\left(\mathcal{D}_{A_i}\right)$ and $f=g$, then
	\[G_i(f)=f=g=G_i(g).\] 
	A similar argument shows that $H_i$ preserves the defining relations in $PK\left(\mathcal{D}\right)_{A_i}$. In addition, 
	\[H_i\circ G_i=id_{\mathcal{D}_{A_i}}, \quad \text{and} \quad G_i\circ H_i= \id_{PK\left(\mathcal{D}\right)_{A_i}},\]	
	\noindent for $i=1,\dotsc,n,$ so the categories in question are isomorphic.

	Now we show that $PK\left(\mathcal{D}\right)_{A_1}\otimes\dotsb\otimes PK\left(\mathcal{D}\right)_{A_n}$ and $\mathcal{D}_{A_1}\otimes\dotsb\otimes \mathcal{D}_{A_n}$ are isomorphic.
		
	Define a functor $G:\mathcal{D}_{A_1}\otimes\dotsb\otimes \mathcal{D}_{A_n}\to PK\left(\mathcal{D}\right)_{A_1}\otimes\dotsb\otimes PK\left(\mathcal{D}\right)_{A_n}$ as follows. On objects, define
	\[
	G\left(D_{1},\dotsc,D_{n}\right)=\left(G_1(D_1),\dotsc, G_n(D_n)\right),
	\]
	and on morphisms that are simple tensors, define
	\[
	G\left(f_1\otimes\dotsb\otimes f_n\right)=G_1(f_1)\otimes\dotsb\otimes G_n(f_n),
	\]
	and extend $G$ by linearity.
	
	Now define another functor $H:PK\left(\mathcal{D}\right)_{A_1}\otimes\dotsb\otimes PK\left(\mathcal{D}\right)_{A_n} \to \mathcal{D}_{A_1}\otimes\dotsb\otimes \mathcal{D}_{A_n}$ as follows. On objects, define 
	\[
	H\left(\left(D_{1},e_1\right),\dotsc, \left(D_{n},e_n\right)\right)=\left(H_1\left(D_{1},e_1\right),\dotsc, H_n\left(D_{n},e_n\right)\right),
	\]
	and on morphisms that are simple tensors, define
	\[
	H\left(f_1\otimes\dotsb\otimes f_n\right)=H_1(f_1)\otimes\dotsb\otimes H_n(f_n),
	\]
	and extend $H$ by linearity.
	
	As the $G_i$ and $H_i$ are isomorphisms that are inverse to each other, it follows that $G$ and $H$ are isomorphisms that are inverse to each other. Therefore, $PK\left(\mathcal{D}\right)_{A_1}\otimes\dotsb\otimes PK\left(\mathcal{D}\right)_{A_n}$ and $\mathcal{D}_{A_1}\otimes\dotsb\otimes \mathcal{D}_{A_n}$ are isomorphic.
	\end{proof}

	\begin{cor}\label{cor:equivalence2}
	Let $\mathcal{D}$ be a strict $\kk$-linear monoidal category. Let $\mathcal{I}$ be a subcollection of the idempotent morphisms in $\mathcal{D}$ satisfying the conditions of \cref{def:Karoubi-envelope}, and suppose $A_1,\dotsc,A_n$ are subcollections of elements in $\mathcal{I}$. Furthermore, for all $i,j=1,\dotsc,n,$ suppose the following conditions are satisfied:
	\begin{itemize}
	\item All objects in $PK\left(\mathcal{D}\right)$ are generated by the objects in $Ob\left(PK\left(\mathcal{D}\right)_{A_1}\right)\cup\dotsb\cup Ob\left(PK\left(\mathcal{D}\right)_{A_n}\right),$
	\item  $x_i\otimes x_j$ is isomorphic to $x_j\otimes x_i$ in $PK\left(\mathcal{D}\right)$ whenever $x_i\in Ob\left(PK\left(\mathcal{D}\right)_{A_i}\right)$ and $x_j\in Ob\left(PK\left(\mathcal{D}\right)_{A_j}\right)$ for $i\neq j$,
	\item The induced functor $F:PK\left(\mathcal{D}\right)_{A_1}\otimes\dotsb\otimes PK\left(\mathcal{D}\right)_{A_n}\to PK\left(\mathcal{D}\right)$ from \cref{def:category-splits} is full on morphisms.
	\end{itemize}	
	Then $PK\left(\mathcal{D}\right)$ is equivalent to $PK\left(\mathcal{D}\right)_{A_1}\otimes\dotsb\otimes PK\left(\mathcal{D}\right)_{A_n}$. 
	
	In addition to all the conditions above, if for any object $\left(D,e\right)\in Ob\left(PK\left(\mathcal{D}\right)_{A_i}\right)$, $e$ is uniquely determined by $D$, then $PK\left(\mathcal{D}\right)$ is equivalent to $\mathcal{D}_{A_1}\otimes\dotsb\otimes \mathcal{D}_{A_n}.$
	\end{cor}
	\begin{proof}
	The three bulleted points constitue the main hypotheses of \cref{theo:equivalence}, so there is an equivalence of categories between $PK\left(\mathcal{D}\right)$ and $PK\left(\mathcal{D}\right)_{A_1}\otimes\dotsb\otimes PK\left(\mathcal{D}\right)_{A_n}.$ In addition, if for any object $\left(D,e\right)\in Ob\left(PK\left(\mathcal{D}\right)_{A_i}\right)$, $e$ is uniquely determined by $D$, then the main hypotheses of \cref{lem:mini-equivalence} are met, so that $PK\left(\mathcal{D}\right)_{A_1}\otimes\dotsb\otimes PK\left(\mathcal{D}\right)_{A_n}$ is isomorphic to $\mathcal{D}_{A_1}\otimes\dotsb\otimes \mathcal{D}_{A_n}$. Therefore, in this case, $PK\left(\mathcal{D}\right)$ is equivalent to $\mathcal{D}_{A_1}\otimes\dotsb\otimes \mathcal{D}_{A_n}$ as well.
	\end{proof}

	Please recall that in $Heis_{F,k}$, $\beta\left(e_i\right)$, $i=1,\dotsc,n$, is a morphism in $Q_+,$ and a downward strand carrying a token labelled by $e_i,$ $i=1,\dotsc,n,$ is a morphism in $Q_-$. For the sake of having notation, denote the downward strand carrying a token labelled by $e_i$ by $\beta'(e_i).$
	
	\begin{defin}
	We denote by $PK\left(\Heis_{F,k}\right)$ the \textit{partial Karoubi envelope} of $\Heis_{F,k}$ with respect to the idempotents $\mathcal{I}=\{\beta(e_1),\dotsc,\beta(e_n),\beta'(e_1),\dotsc,\beta'(e_n)\}.$
	\end{defin}
	\begin{defin}
	For $i=1,\dotsc,n$ we denote by $PK\left(\Heis_{F,k}\right)_i$ the subcategory of $PK\left(\Heis_{F,k}\right)$ that is monoidal with respect to the restriction of the tensor product from $PK\left(\Heis_{F,k}\right)$, and that is generated by the objects $(Q_+,\beta(e_i))$ and $(Q_-,\beta'(e_i)),$ and morphisms of color $i$. In the notation of \cref{def:Karoubi-envelope}, we can write $PK\left(\Heis_{F,k}\right)_i=PK\left(\Heis_{F,k}\right)_{A_i},$ where $A_i=\{\beta(e_i),\beta'(e_i)\}$.
	\end{defin}

	For the final Corollary below, it is worth noting that, in the notation of \cref{def:colored-subcategory}, $\Heis_{F_i,k}=(\Heis_{F,k})_{A_i}$ for $i=1,\dotsc,n,$ where $A_i=\{\beta(e_i),\beta'(e_i)\}.$ This observation allows us to use \cref{cor:equivalence2} in the proof below, with $\mathcal{D}_{A_i}=\Heis_{F_i,k}.$
	
	\begin{cor}\label{cor:equivalence}
	 $PK\left(\Heis_{F,k}\right)$ is equivalent to $\Heis_{F_1,k}\otimes\dotsb\otimes \Heis_{F_n,k}$.
	\end{cor}
	\begin{proof}
		$PK\left(Heis_{F,k}\right)$ is a strict $\kk$-linear monoidal category, and \[X=Ob\left(PK\left(Heis_{F,k}\right)_1\right)\cup\dotsb\cup Ob\left(PK\left(Heis_{F,k}\right)_n\right)\]
		\noindent is a subclass of the objects of $PK\left(\Heis_{F,k}\right)$ such that all objects of $PK\left(\Heis_{F,k}\right)$ are generated by the objects in $X,$ and there is an isomorphism between the objects $q_i\otimes q_j$ and $q_j\otimes q_i$ in $PK(\Heis_{F,k})$ when $q_i\in PK\left(\Heis_{F,k}\right)_i$ and $q_j\in PK\left(\Heis_{F,k}\right)_j$ for $i\neq j$, given by compositions of the diagrams:
		
		\noindent \begin{minipage}{0.25\linewidth}
			\begin{equation}\label{rel:one}
			\begin{tikzpicture}[anchorbase]
			\draw[<-] (0,0) -- (1,1) node[anchor=south] {$i$};
			\draw[->] (1,0) node[anchor=north] {$j$} -- (0,1);	
			\draw{(2,0.5)};
			\draw{(-1,0.5)};			
			\end{tikzpicture},
			\end{equation}
		\end{minipage}
		\begin{minipage}{0.25\linewidth}
			\begin{equation}\label{rel:two}
			\begin{tikzpicture}[anchorbase]
			\draw[->] (0,0) node[anchor=north] {$i$} -- (1,1);
			\draw[<-] (1,0) -- (0,1) node[anchor=south] {$j$};	
			\draw{(2,0.5)};
			\draw{(-1,0.5)};			
			\end{tikzpicture}, 
			\end{equation} 
		\end{minipage}
		\begin{minipage}{0.25\linewidth}
			\begin{equation}\label{rel:three}
			\begin{tikzpicture}[anchorbase]
			\draw[->] (0,0) node[anchor=north] {$i$} -- (1,1);
			\draw[->] (1,0) node[anchor=north] {$j$} -- (0,1);	
			\draw{(2,0.5)};
			\draw{(-1,0.5)};			
			\end{tikzpicture},
			\end{equation}
		\end{minipage}
		\begin{minipage}{0.25\linewidth}
			\begin{equation}\label{rel:four}
			\begin{tikzpicture}[anchorbase]
			\draw[<-] (0,0)  -- (1,1) node[anchor=south] {$i$};
			\draw[<-] (1,0)  -- (0,1) node[anchor=south] {$j$};	
			\draw{(2,0.5)};
			\draw{(-1,0.5)};			
			\end{tikzpicture}.
			\end{equation}
		\end{minipage}
		Following the techniques of \cite[Prop. 8.14]{RS17}, there is a spanning set of $Mor(PK(Heis_{F,k}))$ over $\kk$, denoted $\mathcal{S}$, which is given by diagrams satisfying:
		\begin{itemize}
			\item Any two strands of the diagram intersect at most once,
			\item No strands of the diagram self-intersect,
			\item Each strand consists of at most one cup or cap,
			\item Dots appear above all crossings of strands, and
			\item All bubbles are to the right of the strands of the diagram.
		\end{itemize}
		We have that for all diagrams in $\mathcal{S}$, all bubbles of color $i$ slide through morphisms of color $j$ when $i\neq j$. This property arises from the relations \cref{rel:clockwise-bubble-slide2} and \cref{rel:counterclockwise-bubble-slide2} in \cref{theo:relations2}. Therefore, the induced functor $F$ from \cref{def:category-splits} is full on morphisms of $PK\left(\Heis_{F,k}\right)_1\otimes\dotsb\otimes PK\left(\Heis_{F,k}\right)_n$. Moreover, all $A_i=\{\beta(e_i),\beta'(e_i)\}$ are subsets of $\mathcal{I}=\{\beta(e_1),\dotsc,\beta(e_n),\beta'(e_1),\dotsc,\beta'(e_n)\}$, $\mathcal{I}$ satisfies the conditions in \cref{def:Karoubi-envelope}, and for any object $(Q,e)\in PK(\Heis_{F,k})_i,$ $e$ is uniquely determined by $Q$. Therefore, by \cref{cor:equivalence2}, there is an equivalence of categories between $PK\left(\Heis_{F,k}\right)$ and $\Heis_{F_1,k}\otimes\dotsb\otimes \Heis_{F_n,k}$.
	\end{proof}
	
	
\bibliographystyle{alphaurl}
\bibliography{biblist}

\end{document}